\documentclass[letterpaper,11pt]{article}

\usepackage{amsmath, amsthm,amssymb,bbm, bm,multirow,graphicx}
\usepackage{paralist}
\usepackage{mathpazo}
\usepackage{hyperref}
\usepackage{verbatim}
\usepackage{mathrsfs}
\usepackage{authblk}
\usepackage{subcaption}
\usepackage{longtable}
\usepackage{appendix}
\hypersetup{colorlinks=true}
\usepackage[top=1in, bottom=1in, left=0.8in, right=0.8in]{geometry}
\usepackage{filecontents}
\usepackage{setspace}
\usepackage{mathtools}

\newcommand{\RNum}[1]{\uppercase\expandafter{\romannumeral #1\relax}}
\theoremstyle{definition}
\newtheorem{assumption}{Assumption}[part]

\newtheorem{theorem}{Theorem}[section]
\newtheorem{lemma}[theorem]{Lemma}
\newtheorem{proposition}[theorem]{Proposition}
\newtheorem{corollary}[theorem]{Corollary}
\newtheorem{definition}[theorem]{Definition}
\newtheorem{remark}[theorem]{Remark}

\numberwithin{equation}{section}

\def\nus{\nu^*}
\def\nusl{\nu_{\ell}^*}
\def\nuslp{\nu_{\ell+1}^*}

\def\ss{s^*}
\def\ssl{s^*_\ell}

\def\sslm{s^*_{\ell-1}}
\def\vs{r}
\def\vsl{r_\ell}

\def\vslm{r_{\ell-1}}

\def\ind{\mathbb{I}}

\def\M{K}
\def\R{{\mathbb R}}
\def\N{{\mathbb N}}
\def\Z{{\mathbb Z}}

\def\C{{\mathbb C}}
\def\Z{{\mathbb Z}}

\def\f1{\mathbf{1}}

\newcommand{\twopartdef}[4]{	\left\{		\begin{array}{ll}	#1 & \mbox{if } #2 \\[3mm]	#3 & \mbox{if } #4	 \end{array}	\right.}

\def\xn{X^{(N)}}

\def\j0{-\xn(0)+1}
\def\ki0{-k_{0,i}^{(N)}}

\def\s{\mathbb{S}}

\newcommand{\hc}[0] {[0,\infty)}

\begin{document}
\title{Invariant states of  hydrodynamic limits of randomized load balancing networks}

\author{Pooja Agarwal and Kavita Ramanan}
\affil{Division of Applied Mathematics, Brown University}

\maketitle
\appendixtitleon
\appendixtitletocon

\begin{abstract}  
  Randomized load-balancing algorithms play an important role in improving performance in large-scale 
  networks  at relatively low computational cost.  A common model of such a system is a network of $N$
  parallel queues in which incoming jobs with independent and identically distributed service
  times are routed on arrival using the join-the-shortest-of-$d$-queues routing algorithm.
  Under fairly general conditions, 
  it was shown by Aghajani and Ramanan  that as $N\rightarrow\infty$,
  the state dynamics converges to the unique solution of  a countable system of
  coupled deterministic measure-valued equations called the hydrodynamic equations.  In this article,
  a characterization of  invariant states  of these hydrodynamic equations is obtained and, when   $d=2$, used to construct
  a numerical algorithm to compute the queue length distribution and  mean virtual waiting time in the invariant state.
  Additionally, it is also shown that under a suitable tail condition  on the service distribution, the  queue length distribution of the invariant state
  exhibits a doubly exponential tail decay, 
  thus demonstrating a vast improvement in performance over the case $d=1$, which corresponds to random routing, when the tail decay could even be
  polynomial.  Furthermore, numerical evidence is provided to support the conjecture that the invariant state is the limit of the
  steady-state distributions of the $N$-server models.
  The proof methodology, which  entails analysis of  a coupled system of measure-valued equations, can potentially
  be applied to other many-server systems with general service distributions, where  measure-valued
  representations are useful. 
  \end{abstract}


    \section{Introduction}

    \subsection{Background and Motivation} 
    \label{sec-background}
Randomized load balancing is an effective method that is  used to improve performance
in large scale networks  while incurring relatively low communication overhead and  computation costs.  The  model considered here consists of  a
parallel network  of $N$ servers, to which jobs with independent and identically distributed (i.i.d.) service times 
arrive  according to a renewal process with rate $\lambda N$.  Upon arrival of a job, $d$ out of the $N$ servers are chosen independently and uniformly at random,
and the job is routed to the server with the shortest queue, with ties broken uniformly at random. Each server processes jobs from its queue in a first-come first-serve (FCFS) manner and jobs leave the system on completion of service. A server never idles when there is a job in its queue. The arrival process and service times are assumed to be mutually independent, and service times of jobs have finite mean which, without loss of generality, is taken to be one. The  system described above  will be referred to as the $SQ(d)$ model.

This model was first introduced and analyzed in the case of  Poisson arrivals and exponential service distributions by Vvedenskaya et al. \cite{VveDobKar96} and Mitzenmacher \cite{Mit01}.  They introduced a countable state representation of the process consisting of the fraction of queues with
length greater than $\ell$ for each integer $\ell \geq 0$ and showed that as $N \rightarrow \infty$, the state converges to a fluid limit, characterized as the unique solution to a countable system of coupled ordinary differential equations (ODEs). 
Under the stability condition $\lambda < 1$,  for $d \geq 2$  they also obtained an analytical expression for the 
 the unique invariant state of the fluid limit,  and showed that it exhibits a  doubly exponential decay.
 It was also shown in  \cite{VveDobKar96} that the stationary distributions of the $N$-server systems converge to this
 invariant state, in the limit as $N \rightarrow \infty$. 
On the other hand,  when $d = 1$,  the stationary queue length distribution is the same as that for an M/M/1 queue,
which is well known to exhibit just an exponential tail decay. 
Thus, the works \cite{VveDobKar96, Mit01}  
uncovered the remarkable property that  a dramatic improvement in performance can be achieved by introducing just a little
bit of randomness into the system (i.e., even when $d = 2$), a  phenomenon that has been dubbed ``the power of two choices.''
A variant of this model that has received much attention recently is when the $d = d_N$ grows with $N$, but
in this article we focus on the case when $d$ is fixed and does not grow with $N$.   This involves a very different analysis, entailing a comparison with a system subject to join the shortest queue (JSQ) routing.  In this article we focus on the case when $d$ is fixed and does not grow with $N$, which is particularly relevant in many models where the cost of polling multiple servers is very high.

 Until recently, most work on this model and its variants had focused on  exponential service times. 
An extension to phase-type service distributions was first considered by Li and Lui \cite{LiLui10}, who
analyzed the invariant states of a formal fluid limit.
  More general service distributions were then considered in a   series of works by  Bramson, Lu and Prabhakar
  \cite{Bra11,BraLuPra10,BraLuPra12,BraLuPra13}.   Under very broad
  assumptions  and more general routing
  schemes, it was first shown in \cite{Bra11}   that  under the subcriticality condition $\lambda < 1$,
  each $N$-server system is ergodic with a unique stationary distribution. 
  The works \cite{BraLuPra10,BraLuPra12,BraLuPra13} specialize to the
  SQ($d$) model with 
 Poisson arrivals and service distributions with
  decreasing hazard rate, and   directly establish convergence of  the $N$-server steady-state distributions without first establishing a fluid limit.  
    In particular,  when
  $\lambda < 1$, it is shown in  \cite[Theorem 2.2]{BraLuPra12} that the $N$-server stationary queue length distribution converges to the unique solution
     of a fixed point equation, and this  equation is analyzed to show that  for power law distributions with exponent $-\alpha$, the limiting stationary distribution
     has a  doubly exponential tail if $\alpha > d/(d-1)$ and a power law tail if $\alpha < {d}/{(d - 1)}$.
       However, many realistic service distributions are neither exponential nor phase-type,   and may not have a decreasing hazard rate.   For example, statistical analyses suggest
  that service times follow a log-normal distribution in \cite{BroLaw05}, 
  a Gamma distribution in Automatic Teller Machines \cite{Kol84},
  or a shifted exponential distribution in \cite{chen14,liang14}. 
     While an analogous result iis conjectured to hold for a larger class of service diistributions, this precise class
     has not been identified and, moreover, 
     according to the authors of \cite{BraLuPra12}, it would be challenging to extend their
   approach beyond  service distributions with decreasing hazard rate (see \cite[Paragraph 9, Section 1]{BraLuPra12}).
   
  This motivates taking a different  approach to analyzing the equilibrium behavior of randomized load balancing with general
  service distributions than that adopted in  \cite{BraLuPra10,BraLuPra12,BraLuPra13}. 
  In this article, we consider the approach of first establishing a fluid limit that characterizes
  the limiting dynamics of the $N$-server system, as $N \rightarrow \infty$, and then showing that the
  $N$-server stationary distributions converge to the unique invariant state of the fluid limit.
  This approach has proved fruitful for many network models that admit simple Markovian representations,
  including (as mentioned above), the SQ($d$) model with
  exponential service distributions (see \cite{VveDobKar96}).  
  However, in the case of general service distributions, each of the steps are significantly more challenging due to
  the fact that the dynamics are more compliciated, and there is no common finite or countable-dimensional Markovian
  representation for all $N$-server systems.
However, for a broad class of service distributions,
  recent work of  \cite{AghRam17}  used a convenient state representation in terms of a coupled system of stochastic 
  measure-valued processes and showed that the dynamics converges to the unique solution of a countable
  system of deterministic measure-valued equations, which we referred to as the hydrodynamic equations,
  in accordance with the parlance of interacting particle systems.
 Moreover, 
the works \cite{AghRam17,AghLiRam15,AghLiRam15b} which (under additional conditions on the service distributions) provide a 
  reformulation of the hydrodynamic equations in terms of a countable system of coupled partial differential equations.
  While these papers focused on transient behavior of this network, in our work study equilibrium properties of the
  hydrodynamic limit.  Our first main result,  Theorem \ref{th-reduction}, provides a convenient characterization of
  the invariant state of the hydrodynamic equations in terms of  fixed points of  certain maps (see Remark \ref{rem-s-fixed}).
  The methodology, which  entails analysis of  a system of coupled measure-valued equations, can potentially
  be applied to other many-server systems with general service distributions for which measure-valued representations are 
  useful. 
  This characterization is sufficiently tractable  to construct a numerical algorithm to compute the queue length distribution and virtual
  waiting time in the invariant state, which is described (for simplicity, in the case $d = 2$) in Section \ref{num}. 
  Additionally, we  show in Theorem \ref{thm-dec} that,
  under a tail (or, equivalently, moment) condition on the service distribution,
  the invariant queue length distribution exhibits a doubly
  exponential decay rate.
  Our numerics show that while the decay rate of the invariant queue distribution is of interest,
  it may not manifest itself till far into the tail, and thus it is important to be able to compute and characterize 
  finite queue length exceedance probabilities as well, which is feasible with our algorithm.

  Our work also takes a small step  towards understanding
  equilibrium behavior in large systems. 
    Using our numerical algorithm,  we  provide
  provide  numerical evidence that appears to support the conjecture that for a large class of service distributions,
  the invariant state is the limit of the stationary distributions  of 
  $N$-server systems. 
 Since the $N$-server stationary distributions are known to be tight (due to the results in 
  \cite{Bra11}),   to provide a  rigorous proof of this convergence it would suffice to show
  that the solution to the hydrodynamic equations converges to the invariant state characterized here for a large class of
  initial conditions.   Whereas this is a non-trivial problem that is relegated to future work, 
  it is encouraging that analogous results have been recently established for  a related problem. 
Specifically, the  framework used in
   \cite{AghRam17} to obtain the hydrodynamic limit builds upon 
  simpler  measure-valued representations introduced
  in \cite{KasRam11} and \cite{KanRam10} to analyze many-server systems with a common queue and their fluid limits in
  the absence and presence of abandonments (so-called GI/G/N and GI/G/N+G queues). 
  Invariant states of the associated fluid limits were characterized in Section 6 of \cite{KasRam11} and \cite{KanRam12},
  and more recent work has 
  established convergence to equilibrium  for a large class of general service distributions
  beyond those with decreasing hazard rate \cite{AtaKanKasRam20}. 
  Thus, while a more sophisticated analysis will no doubt be needed to consider the more complicated system of measure-valued
  equations that comprise the hydrodynamic equations, these recent results do offer some hope that this approach is tractable. 
  
   We close by mentioning a related model considered in \cite{VasMukMaz18, VasMukMaz19} of a parallel server network with $SQ(d)$ routing, general service distributions and Poisson arrivals.   Specifically, the work \cite{VasMukMaz18} considers the case when each server has a queue and 
 uses head-of-the-line processor sharing (instead of first-in-first-out, as considered here) to process jobs in the queue, whereas the article \cite{VasMukMaz19} considers a loss network.   Both works employ a measure-valued representation like that in \cite{AghRam17}, establish a hydrodynamic limit, as the number of servers goes to infinity. They also characterize the invariant state and establish insensitivity to the distribution by showing that the  solution to the fixed point  for all service distributions coincides with the one in the exponential case.  
 
  \subsection{Organization of the Paper and Common Notation}
  
  This article is organized as follows. Section \ref{sec-hydro} introduces the
  basic assumptions and the equations that characterize the hydrodynamic limit of the model.
  Section \ref{sec-eqpt} states the main results. Section \ref{num} presents
  numerical approximations of the invariant state performance measures for various distributions,
  and also provides
evidence to support convergence of equilibrium measures to the invariant state.   The proofs of
Theorems \ref{th-reduction} and \ref{thm-dec} are given in 
Sections \ref{proof-thm}  and 
Section \ref{sec-decay}, respectively.
Proofs of certain technical results are relegated to Appendices \ref{subs-abscont}--\ref{sec-unique}.

Throughout,  we will use the following notation. Let $\N$  and $\Z$ denote the set of natural numbers and integers respectively.
For every $E\subset\mathbb{R}$, let $\mathbb{C}_b(E)$ denote the space of continuous bounded functions on $E$, and $\mathcal{B}(E)$ denote the Borel $\sigma$-algebra on $E$.
 Also, for a metric space $\mathbb{X}$, let $\mathbb{D}_{\mathbb{X}}[0, \infty)$ denote the set of $\mathbb{X}$-valued functions on $[0, \infty)$ that are right
     continuous and have finite left limits on $(0, \infty)$, and let
     $\mathbb{C}_{\mathbb{X}}[0, \infty)$ be the subset of $\mathbb{X}$-valued continuous functions on $[0, \infty)$.
         Let $\mathbb{M}_{\leq 1}[0,\infty)$ denote the space of sub-probability measures on $[0,a)$ for some $a \in (0,\infty)$.  Given an interval $A \subset \R$, let   ${\mathcal B}(A)$ denote
             the space of Borel subsets of $A$. 
For any $V\subset {\mathcal B}[0,\infty)$,  a measure $\mu$ on $[0,\infty)$  and an integrable
function $f:V \mapsto \mathbb{R}$, we use the
notation 
\[  \langle f, \mu \rangle_V := \int_V f(x) \mu (dx),  \]
and omit the subscripts $V$ when $V = [0,\infty)$. Also, let $a\wedge b$ represent $\min (a,b)$ for $a,b\in\R$.
  
\section{Hydrodynamic Equations}
\label{sec-hydro}

In this section, we introduce the equations that characterize the hydrodynamic limit of our load balancing model, as established in
\cite{AghRam17}.  Throughout, we make the following assumptions on the service distribution.
\begin{assumption} The service distribution whose cumulative distribution function (cdf) is denoted by $G$,  has density $g$ and finite mean, which can (and will) be set to $1$.
\label{ass-G}
\end{assumption}
Define $\bar{G}(x):=1-G(x)$ and let 
 $L := \sup\{x \in [0, \infty) : \bar{G}(x) >0\}$ be  the right-end of the support of the distribution.
 Also, let $h:[0,L)\mapsto[0,\infty)$ denote the hazard rate function: 
 \begin{equation*}
 h(x)=\frac{g(x)}{\bar{G}(x)},  \quad x \in [0,L). 
 \end{equation*}

 In \cite{AghRam17} the state of an $N$-server randomized load balancing system is represented by the Markov process $\nu^{(N)}(t)=\{\nu^{(N)}_\ell(t)\}_{\ell\in\N}$, where for each $\ell$, $\nu_\ell^{(N)}(t)$ is a (random) finite measure on $[0,\infty)$ that has a unit delta mass at the age (that is, the time spent in service by time $t$) of each job that, at time $t$, is in service at a queue of length  no less than $\ell$.  For every $t \geq 0$, the scaled state $\nu^{(N)}(t)/N$ takes values in $\s$, where 
\begin{align}
\label{def-space}
\s:=&\left\{(\mu_\ell;\ell\in\N)\in\mathbb{M}_{\leq 1}[0,L)^{\mathbb{N}}:\langle f,\mu_\ell\rangle\geq\langle f,\mu_{\ell+1}\rangle,\forall\ell\in\N, f\in\mathbb{C}_b[0,\infty),f\geq 0\right\}, 
\end{align}
where recall  $\mathbb{M}_{\leq 1}[0,L)$ is the space of sub-probability measures on $[0,L)$. 

 For every $d\geq 2$, define
\begin{equation}\label{def_poly}
  \mathfrak{P}_d(x,y):= \frac{x^d-y^d}{x-y}= \sum_{m=0}^{d-1} x^m \;y^{d-1-m}.
\end{equation}
Note that $\mathfrak{P}_2(x,y)=x+y$, and for $0\leq y\leq x$,
\begin{equation}
\label{Pdmax}
\mathfrak{P}_d(x,y)\leq d x^{d-1}.
\end{equation}

We  recall the definition of the hydrodynamic equations given in \cite[Section 2.3]{AghRam17}. 

\begin{definition}\textbf{(Hydrodynamic Equations) }\label{Def_fluid}
  Given $\lambda > 0$ and $\nu(0) \in \s$, $\{\nu(t)=\nu_\ell(t);\ell\in\N,t\geq0\}$ in
  $\C_{\s}[0,\infty)$ is said to solve the  
    \textit{hydrodynamic equations} associated with $(\lambda, \nu(0))$ if and only if
    for every  $t\in\hc$,
\begin{equation}\label{Fluid_bound}
        \int_0^t\langle h,\nu_1(s) \rangle ds<\infty,
\end{equation}
and the following equations are satisfied:
\begin{equation}\label{Fluid_Balance}
  \langle\f1,\nu_\ell(t)\rangle - \langle\f1,\nu_\ell(0)\rangle = D_{\ell+1}(t)+\int_0^t\langle\f1,\eta_\ell(s)\rangle ds -D_\ell(t),\qquad \ell\in\N,
\end{equation}
and, for every $f\in\mathbb{C}_b[0,\infty)$,
\begin{align}\label{Fluid_f}
    \langle f, \nu_\ell(t)\rangle = &\langle f(\cdot+t)\frac{\bar G(\cdot+t)}{\bar G(\cdot)},\nu_\ell(0)\rangle +\int_{[0,t]} f(t-s)\bar G(t-s)dD_{\ell+1}(s)\\
    & + \int_0^t \langle f(\cdot+t-s)\frac{\bar G(\cdot+t-s)}{\bar G(\cdot)},\eta_\ell(s)\rangle ds, \notag
\end{align}
where
\begin{equation}\label{Fluid_D}
    D_\ell(t):=\int_0^t\langle h,\nu_\ell(s) \rangle ds,\quad\quad \ell\in\N,
  \end{equation}
and
\begin{equation}\label{Fluid_R}
  \eta_\ell(t):=\twopartdef{\lambda\left(1-\langle\f1,\nu_1(t)\rangle^d\right)\delta_0}{\ell=1,} {\lambda\mathfrak{P}_d\Big(\langle \f1,\nu_{\ell-1}(t)\rangle,\langle \f1,\nu_\ell(t)\rangle\Big)(\nu_{\ell-1}(t)-\nu_\ell(t))}{\ell\geq 2.}
\end{equation}
\end{definition}

\begin{remark}
The bound \eqref{Fluid_bound} implies that for every $\ell\in\N$, the
process $D_\ell$ in \eqref{Fluid_D} is well defined. 
\end{remark}

We now briefly provide some intuition behind the form  of the hydrodynamic equations. The term $\langle h,\nu_\ell(s)\rangle$ 
represents the limiting mean conditional scaled departure rate from queues of length at least $\ell$ at time $s$, and thus $D_\ell(t)$ given by \eqref{Fluid_D}  represents the limiting cumulative scaled departure rate from queues of length at least $\ell$ at  time $t$. 
The term $\langle \f1,\nu_\ell(t)\rangle$ represents the limiting fraction of queues of length at least $\ell$ at time $t$.  Note that the scaled arrival rate of jobs to the network is given by $\lambda$, and for $\ell \geq 1$, the probability that an arriving job is routed to a queue of length $\ell - 1$ at time $s$ is equal to $\mathfrak{P}_d(\langle\f1,\nu_{\ell-1}(s)\rangle,\langle\f1,\nu_\ell(s)\rangle)(\langle\f1,\nu_{\ell-1}(s)\rangle-\langle\f1,\nu_\ell(s)\rangle)$, with the convention $\langle\f1,\nu_0(s)\rangle = 1$. Thus, $\langle\f1, \eta_\ell (s)\rangle$, where $\eta_\ell$ is defined by \eqref{Fluid_R}, represents the scaled arrival rate at time $s$ of jobs to queues of length $\ell - 1$. Hence,  $\int_0^t\langle\f1,\eta_\ell(s)\rangle ds$ is the total increase in the fraction of queues of length greater than or equal to $\ell$ due to arrivals in the interval $[0, t]$. On the other hand, $D_\ell(t) - D_{\ell+1}(t)$ represents the cumulative decrease in the fraction of queues of length at least $\ell$ due to service completions. The mass balance equation, \eqref{Fluid_Balance} is a result of these observations.
Lastly, the  right-hand side  of equation \eqref{Fluid_f} describes the three terms that contribute to the measure $\nu_\ell(t)$. The first term accounts for jobs that were already in service at time $0$. The second term represents the contribution due to departures from queues of length greater than or equal to  $\ell +1$ in the interval $[0,t]$. The third term represents the contribution to $\nu_\ell(t)$ due to jobs that entered service at a queue of length at least $\ell$ at some time $s \in (0, t]$ and are still in service at time $t$. \\

Along with Assumption \ref{ass-G},
we also make the following assumption throughout. 

\begin{assumption} There exists a unique solution to the hydrodynamic equations.
\label{ass-G2}
\end{assumption}

  \begin{remark}
    A sufficient condition for Assumption \ref{ass-G2}, as established in \cite[Theorem 2.6]{AghRam17},
    is that the density $g$ of the service distribution be bounded on every finite interval of $[0,L)$.
      Examples of distributions satisfying Assumptions \ref{ass-G} and \ref{ass-G2} are  given in Remark 2.1 of \cite{AghLiRam15}. 
 \end{remark}

\section{Main Results}
\label{sec-eqpt}

To state our main results, we will require the following basic definition. 

\begin{definition}[{\bf Invariant State}]
  Given $\lambda > 0$, $\nus\in\s$ is said to be an invariant state of the hydrodynamic equations with arrival rate
  $\lambda$ if for every $t\geq 0$,  $\nu \in \mathbb{C}_{\s}[0,\infty)$ defined by
\begin{equation*}
\nu(t)=\nus\qquad \forall t\geq 0.
\end{equation*}
  solves the hydrodynamic equations associated with $(\lambda, \nu^*)$.
\end{definition}

When $\lambda$ is clear from the context, we will just say it is an invariant
state of the hydrodynamic equations. 
 Given an invariant state  $\nus = (\nu_\ell^*)_{\ell \in\N}$  of the hydrodynamic equations, we define 
\begin{equation}
\label{def-ssl}
  \ssl := \langle \f1, \nu_\ell^* \rangle, \quad   \ell \in\N, 
\end{equation}
which captures the corresponding invariant queue length distribution.
We will only be interested in  invariant states for which the 
invariant queue length distribution satisfies 
\begin{equation}
  \label{eq-tailcond}
  \lim_{\ell \rightarrow \infty} \ssl = 0. 
\end{equation}

Note that this is a necessary condition for the mean of the invariant
queue length distribution to be finite. Moreover, 
  as shown in Lemma \ref{lem-s1eqbm}, the  condition \eqref{eq-tailcond} turns out to be
   equivalent to requiring $s_1^* = \lambda$, which 
 will  be satisfied by
 any ``physically relevant'' invariant state,
 including one that arises as the limit of
 stationary $N$-server systems.

 \begin{remark}
   \label{rem-invstate}
   In the sequel, we will refer to a physical invariant state as any invariant state
   that additionally satisfies \eqref{eq-tailcond}.   The latter  condition is necessary
for uniqueness even in the case of an  exponential service
distribution since, for example, as it is easy to verify, for any
service distribution $G$, $\nu^*_\ell(dx) = \bar{G}(x) dx$ (and, consequently, 
 $s_\ell^* = 1$)  for every $\ell \in \N$, is always an invariant state of the hydrodynamic equations.
\end{remark}

 We now state our first main result on existence and characterization
 of physical invariant states.  Its proof is given in Section \ref{pf-mainresult}.
\begin{theorem}
\label{th-reduction}
Suppose Assumptions \ref{ass-G} and \ref{ass-G2} hold and fix $\lambda \in (0,1]$, $d \in \N$.  
Then  there exists a  physical invariant state of the hydrodynamic equations with arrival rate $\lambda$
if and only if there exist a sequence of measurable functions $\{r_\ell\}_{\ell\in\N}$ on
$[0,\infty)$ that are continuously  differentiable on $[0,L)$ such that for  $x\notin [0,L)$,  $\vsl(x)=0$ 
 and
 for $x\in [0,L)$, $r_1(x)= \lambda$ and for $\ell\geq 2$,
\begin{align}
\label{eqn-vsl}
\vsl (x) =& 
\lambda (\ssl)^d e^{-\lambda \mathfrak{P}_d( \sslm, \ssl) x} + \lambda  \mathfrak{P}_d( \sslm, \ssl)\int_0^x
      e^{-\lambda  \mathfrak{P}_d( \sslm, \ssl) (x-u)} r_{\ell-1} (u)   du,
\end{align}
where for $\ell \geq 1$, 
\begin{equation}
\label{eqn-ssl}
   \ssl   = \int_0^L  \vsl (x) \bar{G} (x) dx. 
\end{equation}
In this case, 
      \begin{equation}
\label{def-vsl}
  \nu_\ell^* (A)  = \int_A \vsl (x) \bar{G}(x) dx, \qquad  A \in {\mathcal
    B}([0,L)), \quad \ell \in  \N, 
\end{equation}
 Furthermore,  a  physical  invariant state of the hydrodynamic equations
 with arrival rate $\lambda$  always exists, and also satisfies 
\begin{equation}
\label{eqn-vslbc2} 
\lambda (\sslm)^d =  \int_0^L g(x) \vsl (x) dx,\qquad\ell\geq 2. 
\end{equation}
\end{theorem}

\begin{remark}
\label{rem-s-fixed}
 Now, fix $\lambda\in(0,1)$ and for $d\geq 2$, define $r_1\equiv\ss_1=\lambda$. Note that $\{\ssl\}_{\ell\geq 1}$ is a physical invariant state or, equivalently, satisfies \eqref{eqn-vsl}-\eqref{eqn-ssl} if and only if for $\ell\geq 2$,  $\ssl$ is a fixed point of the map
\begin{equation}
  \label{eq-fld}
  F_{\ell}(s) := 
\lambda s^d \int_0^L e^{-\lambda \mathfrak{P}_d( \sslm, s) x} \bar{G}(x) dx+ \lambda  \mathfrak{P}_d( \sslm, s)\int_0^L\left(\int_0^x
e^{-\lambda  \mathfrak{P}_d( \sslm, s) (x-u)} r_{\ell-1} (u)   du\right)\bar{G}(x) dx, 
\end{equation}
on $[0,\sslm]$.
\end{remark}

As an immediate consequence of Theorem \ref{th-reduction}, we recover 
the following well-known result for the exponential service distribution (see  \cite[Theorem 1]{VveDobKar96} and \cite[Lemma 2]{Mit01}).
\begin{remark}
\label{cor-exp}
If $\bar{G} (x) = e^{-x}$, then $g(x)=\bar{G}(x)$ and \eqref{eqn-ssl} and \eqref{eqn-vslbc2} show that $s_\ell^*=\lambda(s_{\ell-1}^*)^d$ for all $\ell\geq 2$. Hence, solving the recursion with $s_1^*=\lambda$, we obtain
\[ 
\ssl = \lambda^{\frac{d^\ell-1}{d-1}}, \quad \ell \in \mathbb{N}. 
\]

\end{remark}

Our next result, which concerns the tail behavior of a physical invariant state, requires an additional condition: 
\begin{assumption}
\label{ass-service}
Fix $d\geq 2$. Suppose  there exist $x_0>1/\lambda d$,  $C_0<\infty$, and ${\beta}>d/(d-1)$ such that 
\begin{equation}
\label{Gdecay}
\bar{G}(x)\leq C_0x^{-{\beta}},\qquad \forall x\geq x_0.
\end{equation}
\end{assumption}
\begin{remark}
\label{rem-dec}
 This tail condition is almost equivalent to a moment condition. Specifically, a sufficient condition for the tail condition is that service distribution have finite $(\beta+1)+\varepsilon$ moment for some $\varepsilon>0$, whereas a necessary condition is that the service  distribution
has a finite $(\beta+1)$ moment. Thus, it is immediate that any distribution with all moments finite, such as the Gamma, shifted exponential and lognormal, as well as the Pareto, the latter with tail parameter greater than $\beta+1$), all satisfy the conditions in Assumption \ref{ass-service}.
\end{remark}

We prove the following result in Section \ref{subsec-reduc}.
\begin{theorem} 
  \label{thm-dec}
  Fix $\lambda \in (0,1)$ and $d \geq 2$ and 
suppose Assumptions \ref{ass-G}--\ref{ass-service} hold. Then, the queue length
distribution $(s_\ell^*)_{\ell\in \N}$  associated with any physical invariant state
exhibits a doubly exponential decay, in the sense that   
there exists a constant  $n_d=n_d(\beta)>0$ such that 
\begin{equation}
\label{dbdecay}
\liminf_{\ell\rightarrow\infty}\frac{1}{\ell}\log_d\log\left(\frac{1}{s_\ell^*}\right)\geq n_d.
\end{equation}
\end{theorem}

The result in Theorem \ref{thm-dec}   significantly 
extends previous results on the power of two choices for specific classes of
service distributions such as  exponential \cite{Mit01,VveDobKar96},
phase-type 
\cite{LiLui10} and  Pareto with parameter $\alpha>2$ \cite{BraLuPra13}.  
Besides applying to more general distributions like lognormal and shifted exponential distributions
that are relevant in practice, it 
 provides a unified approach for obtaining all these results as special cases. 
The most  significant contribution of Theorem \ref{thm-dec} is in identifying the tail decay
condition on the service distribution in  Assumption \ref{ass-service} as the precise  property that leads to a  doubly exponential tail decay. 
Indeed, the results for power-law distributions obtained in  \cite{BraLuPra13} provide a
counterexample that shows that when this tail decay condition is not satisfied, then
the asymptotic tails of the invariant queue lengths need not be  doubly exponential, and in fact, could be power law.   In this sense,  the result of Theorem \ref{thm-dec} is tight.

While the tail decay property is an interesting property, the decay
property may not manifest itself till rather far into the tails and so
finite (invariant) queue length exceedance probabilities are often of more
practical interest.  
In Section \ref{num}, we illustrate how for $d=2$, the  characterization of the invariant states of the measure-valued hydrodynamic equations obtained in Theorem \ref{th-reduction}
allows us to compute more general
invariant quantities of relevance besides the queue length, such
as the virtual waiting time, which appears not to have been considered in the literature before (except in the case of exponential service times).

\section{Numerical Results when $d=2$}
\label{num}

We now present numerical results, for simplicity restricting to the case $d=2$.
\subsection{Properties of the physical invariant state}
We first obtain an alternative representation for $(r_\ell)_{\ell\in\N}$ that is more amenable to computation than 
the one given in \eqref{eqn-vsl}-\eqref{eqn-ssl}.  Recall
the definition of a physical invariant state from Remark \ref{rem-invstate}.

\begin{lemma}
Suppose $\nus = (\nu_j^*)_{j \geq 1}\in\mathbb{S}$ is a physical 
invariant state of the  hydrodynamic equations, with an 
associated sequence of measurable functions $(r_j)_{j \in \N}$ as in
Theorem \ref{th-reduction}. 
Then   $r_1\equiv \lambda$ on $[0,L)$ and  for $j\geq 2$,
\begin{equation}
\label{def-rl-rec}
r_j(x)=\lambda+\lambda\sum_{i=2}^jc_{i,j}e^{-\lambda \mathfrak{P}_d(s_{i-1}^*,s_i^*)x}, \qquad\forall x\in[0,L),
\end{equation}
 where $c_{2,2}:=(s_2^*)^d-1$, and for $j\geq 3$, $c_{i,j}$ are recursively defined as follows:
\begin{equation}
\label{def-const}
c_{i,j}:=\left\{\begin{array}{lr}
c_{i,j-1}\frac{\mathfrak{P}_d(s_{j-1}^*,s_j^*)}{\mathfrak{P}_d(s_{j-1}^*,s_j^*)-\mathfrak{P}_d(s_{i-1}^*,s_i^*)}&\quad \text{if }i<j,\\
(s_j^*)^d-1-\sum_{k=2}^{j-1}c_{k,j}&\quad\text{if }i=j.
\end{array}\right.
\end{equation}
\end{lemma}
\begin{proof}
The identity $r_1\equiv \lambda$ on $[0,L)$ follows from Theorem \ref{th-reduction}. Substituting $\ell=2$ in \eqref{eqn-vsl} with $r_1\equiv \lambda$, we see that
\begin{align*}
\vs_2(x)&=\lambda(s_2^*)^de^{-\lambda\mathfrak{P}_d(s_1^*,s_2^*)x}+\lambda\mathfrak{P}_d(s_1^*,s_2^*)\int_0^xe^{-\lambda\mathfrak{P}_d(s_1^*,s_2^*)(x-u)}\lambda\, du\\
&= \lambda +\lambda e^{-\lambda\mathfrak{P}_d(s_1^*,s_2^*)x}\left((s_2^*)^d-1\right),
\end{align*}
which verifies \eqref{def-rl-rec} for $j=2$. Now, suppose \eqref{def-rl-rec} holds for some $j\geq 2$. Substituting $\ell=j+1$ in \eqref{eqn-vsl} 
and using  \eqref{def-rl-rec} and \eqref{def-const}, we have
\begin{align*}
r_{j+1} (x)  &=\lambda (s_{j+1}^*)^d e^{-\lambda \mathfrak{P}_d( s_j^*, s_{j+1}^*) x} + \lambda^2 \mathfrak{P}_d( s_j^*, s_{j+1}^*)\int_0^x[1+\sum_{i=2}^j c_{i,j}e^{-\lambda \mathfrak{P}_d(s_{i-1}^*,s_i^*)u}]
      e^{-\lambda  \mathfrak{P}_d( s_j^*, s_{j+1}^*) (x-u)} du\\
      &=\lambda (s_{j+1}^*)^d e^{-\lambda \mathfrak{P}_d( s_j^*, s_{j+1}^*) x} + \lambda(1-e^{-\lambda \mathfrak{P}_d( s_j^*, s_{j+1}^*) x} )\\
      &\quad+\lambda   \sum_{i=2}^jc_{i,j}\frac{ \mathfrak{P}_d( s_j^*, s_{j+1}^*)}{\mathfrak{P}_d( s_j^*, s_{j+1}^*)-\mathfrak{P}_d( s_{i-1}^*, s_{i}^*)}\left(e^{-\lambda \mathfrak{P}_d(s_{i-1}^*,s_i^*)x}-e^{-\lambda  \mathfrak{P}_d( s_j^*, s_{j+1}^*) x}\right)\\
        &=\lambda+ \lambda( (s_{j+1}^*)^d-1)e^{-\lambda \mathfrak{P}_d( s_j^*, s_{j+1}^*) x}+\lambda   \sum_{i=2}^j c_{i,j+1}
      \left(e^{-\lambda \mathfrak{P}_d(s_{i-1}^*,s_i^*)x}-e^{-\lambda  \mathfrak{P}_d( s_j^*, s_{j+1}^*) x}\right) \\
      &=\lambda+\lambda\sum_{i=2}^{j+1}c_{i,j+1}e^{-\lambda \mathfrak{P}_d(s_{i-1}^*,s_i^*)x}, 
\end{align*}
 which shows that \eqref{def-rl-rec} also holds when $j$ is  replaced with $j+1$.  
               Thus, by the principle of mathematical induction, we are done.
\end{proof}

\subsection{An algorithm for computing the invariant state}
The following result, whose proof is deferred to Appendix \ref{sec-unique}, is the key result that allows numerical computation of the invariant states. 
\begin{proposition}
  \label{th-uniqueness}
  When $d=2$ and Assumptions \ref{ass-G} and \ref{ass-G2} hold, then for each $\lambda\in(0,1)$ there exists a unique physical invariant state $(\ssl)_{\ell\in\N}$ for the hydrodynamic equations with arrival rate $\lambda$. In particular, ${s}_1^*=\lambda$ and for each $\ell\geq 2$, ${s}^*_\ell$ is unique fixed point in $[0,{s}^*_{\ell-1}]$ of the function $F_\ell$ defined in \eqref{eq-fld}.
\end{proposition}
Now, fix $d = 2$,  $\lambda \in (0,1)$ and a service distribution that satisfies
Assumptions \ref{ass-G} and \ref{ass-G2}.
Let $(\ssl)_{\ell\in \N}$ be the  invariant
queue length distribution, which is uniquely defined by Proposition \ref{th-uniqueness}.  For $\ell\geq 2$, given $\{\ss_i\}_{i\leq \ell-1}$ and $\{r_i\}_{i\leq \ell-1}$, and $F_\ell$ as in \eqref{rem-s-fixed}, we compute $\ssl$ as the zero of the function $s\mapsto F_\ell(s)-s$ on the interval $[0,\sslm]$ using the bisection method with an error tolerance of $10^{-12}$. We then use the recursion \eqref{def-rl-rec} using the adaptive quadrature method with a tolerance of $10^{-12}$ to numerically
approximate the integrals therein.

      The resulting numerical estimates
    for $\lambda = 0.5$ and a variety of distributions (see Appendix \ref{distribution} for details) are tabulated in  
    Table \ref{tab:table-sl}, where $s_\ell^*$ is set to zero if the numerical estimate for $s_\ell^*$ is below the precision error, i.e., $s_\ell^*\leq 10^{-12}$.
    Continuing the numerical computation for  $6\leq\ell\leq 8$ for the above distributions, the only non-zero values are $s_6^*=2.7\times 10^{-6}$, and $s_7^*=1.8147\times 10^{-11}$ 
    for the Weibull distribution with $a=0.5$, and $s_6^*=0.0231$, $s_7^*=0.0182$ and $s_8^*=0.015$ for the Pareto distribution with $\alpha=1.5$.
   Even for small values of $\ell$ the decay of $\ssl$ appears to be
      faster for lighter-tailed  as opposed to heavier-tailed distributions
      (such as  Weibull with  $\alpha = 0.5$, and Pareto  with
      $\alpha = 1.5$). In Section \ref{sec-conv-inv}, we show that these approximations to the invariant state agree well with the approximations to the limiting stationary distribution obtained via simulations.

\begin{table}[h!]
\begin{center}
\begin{tabular}{|c|c|c|c|c|c|c|c|c|}
\hline
&$s_2^*$&$s_3^*$&$s_4^*$&$s_5^*$\\
\hline
Exponential &$0.125$&$0.0078$&$3.0518\times 10^{-5}$&$4.6566 \times 10^{-10}$ \\
\hline
Gamma &$0.0871$& $0.0022$ &$1.1892\times 10^{-6}$&$3.629\times10^{-11}$\\
($\alpha=3$)&&&&\\
\hline
Weibull&$0.0838$&$0.0018$&$7.1889\times 10^{-7}$&$0$\\
($a=2$)&&&&\\
\hline
Weibull &$0.2425$&$0.0872$&$0.0166$&$0.0008$\\
($a=0.5$)&&&&\\
\hline
Lognormal&$0.0738$&$0.0012$&$2.4229\times 10^{-7}$&$2.985\times 10^{-11}$\\
($\sigma=\frac{1}{3}$)&&&&\\
\hline
 Pareto&$0.0812$&$0.0024$&$1.2208\times 10^{-5}$&$0$\\
 ($\alpha=3$)&&&&\\
 \hline
  Pareto&$0.1820$&$0.0797$&$0.0460$&$0.0311$\\
  ($\alpha=1.5$)&&&&\\
  \hline
  Burr&$0.1117$&$0.0076$&$9.8831\times 10^{-5}$&$1.1969\times 10^{-7}$\\
($c=2$, &&&&\\
 $k=1.5$)&&&&\\
\hline
 \end{tabular}
 \caption{Numerical estimates for $s_\ell^*$, $2\leq\ell\leq 5$,  for various distributions.}
 \label{tab:table-sl}
 \end{center}
 \end{table}

To understand the rate of decay of the invariant queue length with respect to $\ell$, 
Table \ref{tab:table-h}  tabulates
the function $H(\ell)={\log_2(\log({1}/{s_\ell^*}))}/{\ell}$
for the above distributions, where the entry NaN denotes that $H(\ell)$
could not be computed due to loss of precision while estimating $s_\ell^*$.  
For service distributions exhibiting a doubly exponential tail decay, by Theorem \ref{thm-dec} 
  $H(\ell)$ should converge to (or be lower bounded by) a constant as $\ell \rightarrow \infty$. For example, it is easy to check that for the exponential service distribution $H(\ell)$ converges to 1 as $\ell\rightarrow\infty$.
   This assertion cannot be inferred from Table \ref{tab:table-h}, thus showing
  that the decay rate does not manifest itself till far into the tail.   Thus, from a practical point of view, this underscores the importance of being able to compute $s_\ell^*$ for finite $\ell$, as our approach
  allows.  
 
\begin{table}[ht!]
\begin{center}
\begin{tabular}{|c|c|c|c|c|c|c|c|c|}
\hline
&$H(2)$&$H(3)$&$H(4)$&$H(5)$\\
\hline
Exponential&$0.5281$&$0.7595$&$0.8445$&$0.8851$\\
\hline
Gamma&$0.6436$&$0.8724$ &$0.9425$&$0.9688$\\
($\alpha=3$)&&&&\\
\hline
Weibull&$0.6549$&$0.8859$ &$0.9556$&NaN\\
($a=2$)&&&&\\
\hline
Weibull &$0.2513$&$0.4289$&$0.5085$&$0.5650$\\
($a=0.5$)&&&&\\
\hline
Lognormal &$0.6911$&$0.9182$&$0.9821$&$0.9198$\\
($\sigma=\frac{1}{3}$)&&&&\\
\hline
Pareto &$0.6640$&$0.8629$&$0.8878$&NaN\\
($\alpha=3$)&&&&\\
\hline
Pareto&$0.3843$&$0.4462$&$0.4057$&$0.3590$\\
($\alpha=1.5$)&&&&\\
\hline
Burr&$0.5562$&$0.7621$&$0.8013$&$0.7989$\\
($c=2$, &&&&\\
$k=1.5$)&&&&\\
\hline
\end{tabular}
\caption{Computations of $H(\ell)=\frac{\log_2(\log(\frac{1}{s_\ell^*}))}{\ell}$  for $s_\ell^*$, $2\leq\ell\leq 5$.}
\label{tab:table-h}
\end{center}
\end{table}
 Due to the measure-valued representation, 
it is also possible to calculate the invariant states of 
other important performance measures in addition to the queue length.  
For example, consider the virtual waiting time at any time $t$, 
defined to be the amount of time that a 
virtual job hypothetically arriving at time $t$ would have to wait before entering service.  
Using \cite[Theorem 4.5]{AghLiRam15}, it can be easily deduced that
when $d=2$, the invariant mean virtual waiting time, $W^*$ is given by 
\begin{align}
\label{wait}
W^*=&\sum_{\ell\geq 2}Z_\ell(0)^2+\sum_{\ell\geq 1}[Z_\ell(0)+Z_{\ell+1}(0)]\int_0^\infty[Z_\ell(r)+Z_{\ell+1}(r)]dr,
\end{align}
where $Z_\ell(r):=\langle ({\bar{G}(\cdot+r)}/{\bar{G}(\cdot)}),\nu_\ell^*\rangle.$ We now compute the invariant mean virtual waiting time by using the following approximation to \eqref{wait}.

\begin{align*}
W^*=&\sum_{\ell= 2}^{L_0}Z_\ell(0)^2+\sum_{\ell= 1}^{L_0-1}[Z_\ell(0)+Z_{\ell+1}(0)]\sum_{j=0}^{\lfloor R_0/\delta\rfloor}[Z_\ell(r_j)+Z_{\ell+1}(r_j)]\delta,
\end{align*}
where $L_0=6$, $R_0=20$, $\delta=0.003$ and $r_j=j\delta$, with a tolerance of $10^{-12}.$\\
\begin{table}[h]
\begin{center}
\begin{tabular}{|c|c|c|c|c|c|c|}
\hline
& Exponential &Pareto&Pareto&Pareto&Pareto&Pareto\\
&&($\alpha=3$)&($\alpha=2.5$)&($\alpha=2$)&($\alpha=1.75$)&($\alpha=1.5$)\\
\hline
$W^*$&$ 0.4246$&$0.1673$&$0.1905$&$0.2449$&$0.3042$&$0.4287$\\
\hline
\end{tabular}
\caption{Invariant mean virtual waiting time for  Exponential and Pareto distributions.}
\end{center}
\end{table}
\begin{table}[h]
\begin{center}
\begin{tabular}{|c|c|c|c|c|c|}
\hline
&Weibull&Weibull&Weibull&Weibull&Gamma\\
&($a=2$)&($a=1.5$)&($a=1$)&($a=0.5$)&($\alpha=3$)\\
\hline
$W^*$&$0.1716$&$0.1966$&$0.2661$&$0.6781$&$0.1789$\\
\hline
\end{tabular}
\caption{Invariant mean virtual waiting time for Weibull and Gamma distributions.}
\end{center}
\end{table}

The numerics show that for a given distribution,
the invariant mean virtual waiting time increases with increase in the heaviness of  the tail of the distribution.
While this may appear intuitive, it is not completely obvious, since it is in contrast to  transient time results obtained in \cite[Section 5.2.2]{AghLiRam15}, where the relaxation time (defined to be the time it takes for the network backlog to decrease to the extent that the mean virtual waiting time reaches half of its initial value) is smaller in the heavy-tailed case when compared to the light-tailed case.

\subsection{Potential convergence of the $N$-server equilibria to the invariant state}
\label{sec-conv-inv}
     
  As mentioned in Section \ref{sec-background}, for service distributions with decreasing hazard rate function, it was shown in \cite{BraLuPra12} that the sequence of stationary distributions $\pi_N$ of $N$-server systems subject to our load balancing algorithm converges to a limit $\pi^*$ and that in the case of power law distributions with tail parameter $\beta$, this limit has a doubly exponential tail if and only if  $\beta>d/(d-1)$. They use a different approach not involving the study of dynamic behavior (or the hydrodynamic limit). 

An alternative approach is to characterize the hydrodynamic limit and show that the unique invariant state of the hydrodynamic equations coincides with the limiting stationary distribution. The results of this paper, along with that of \cite{AghRam17} and \cite{BraLuPra12} can be combined to show that this is true in the case of service distributions with decreasing hazard rate. We conjecture that the unique physical invariant state coincides with the limiting stationary distribution even for more general service distributions. Below, we provide numerical evidence to support this conjecture.

We only present results for the lognormal and Pareto distributions. Note that, while the Pareto distribution has a decreasing hazard rate function, and so one expects agreement of the two quantities due to \cite{AghRam17,BraLuPra12}, it is significant that the lognormal distribution still shows agreement despite not having a monotonic hazard rate function. Similar trends are observed for a variety of other service distributions and arrival rates.
Figures \ref{fig1} and \ref{fig2} below
present a comparison, for $\ell = 1, 2, 3$,
of $\ssl$ with   estimates of the probability that a typical queue has length no less than $\ell$ at time $t$, using 
Monte Carlo simulations of an $N$-server network with $N=600$ and $\lambda =0.5$ for the time interval $t=[0,15]$ for the Lognormal distribution, and $t=[0,10]$ for the Pareto distribution, using  600 realizations, with the initial condition that at time $t=0$, every server has exactly one job with age $0$. We then compute the probability that a typical queue has length $\geq\ell$ at time $t$, and compare this quantity with the invariant state $s_\ell^*$ obtained in Table \ref{tab:table-sl}. For large $\ell$ (e.g., $\ell \geq 4$) it becomes harder to use simulations to get accurate estimates      of the equilibrium probability of the queue exceeding $\ell$; whereas, the recursions can still be used to 
   get an approximation to $s_\ell^*$.  Figure \ref{fig1} confirms the results for the Lognormal distribution and Figure \ref{fig2} for the Pareto distribution and  provides
   support for the conjecture (similar trends are observed for a variety of initial conditions and service distributions.)
   
\vspace{-4pt}
 \begin{figure}[h!]
\begin{subfigure}{0.3\textwidth}
      \centering
      \includegraphics[width=0.9\textwidth]{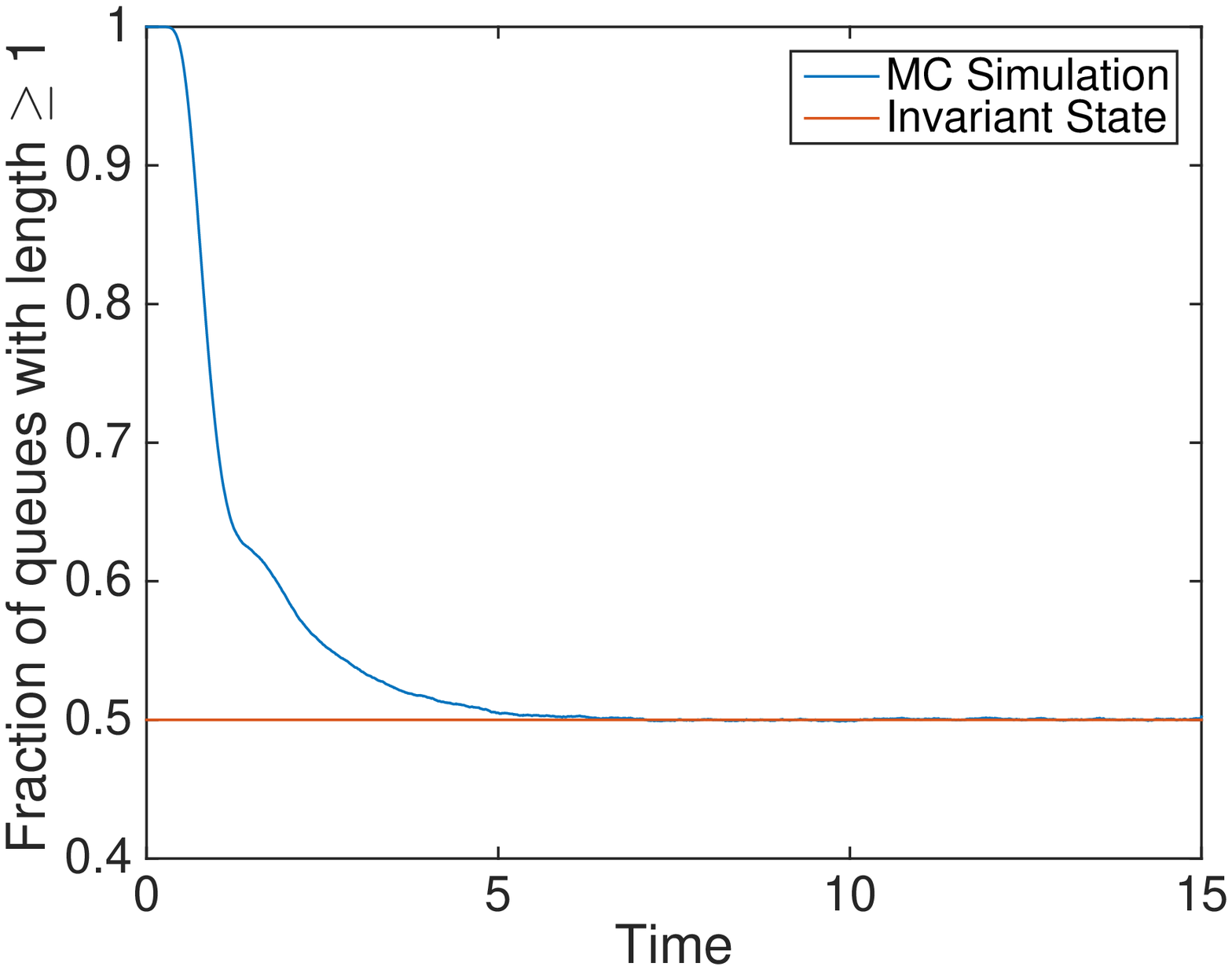} 
     \caption{$\ell=1$}
      \end{subfigure}
     \begin{subfigure}{0.3\textwidth}
       \centering
       \includegraphics[width=0.9\textwidth]{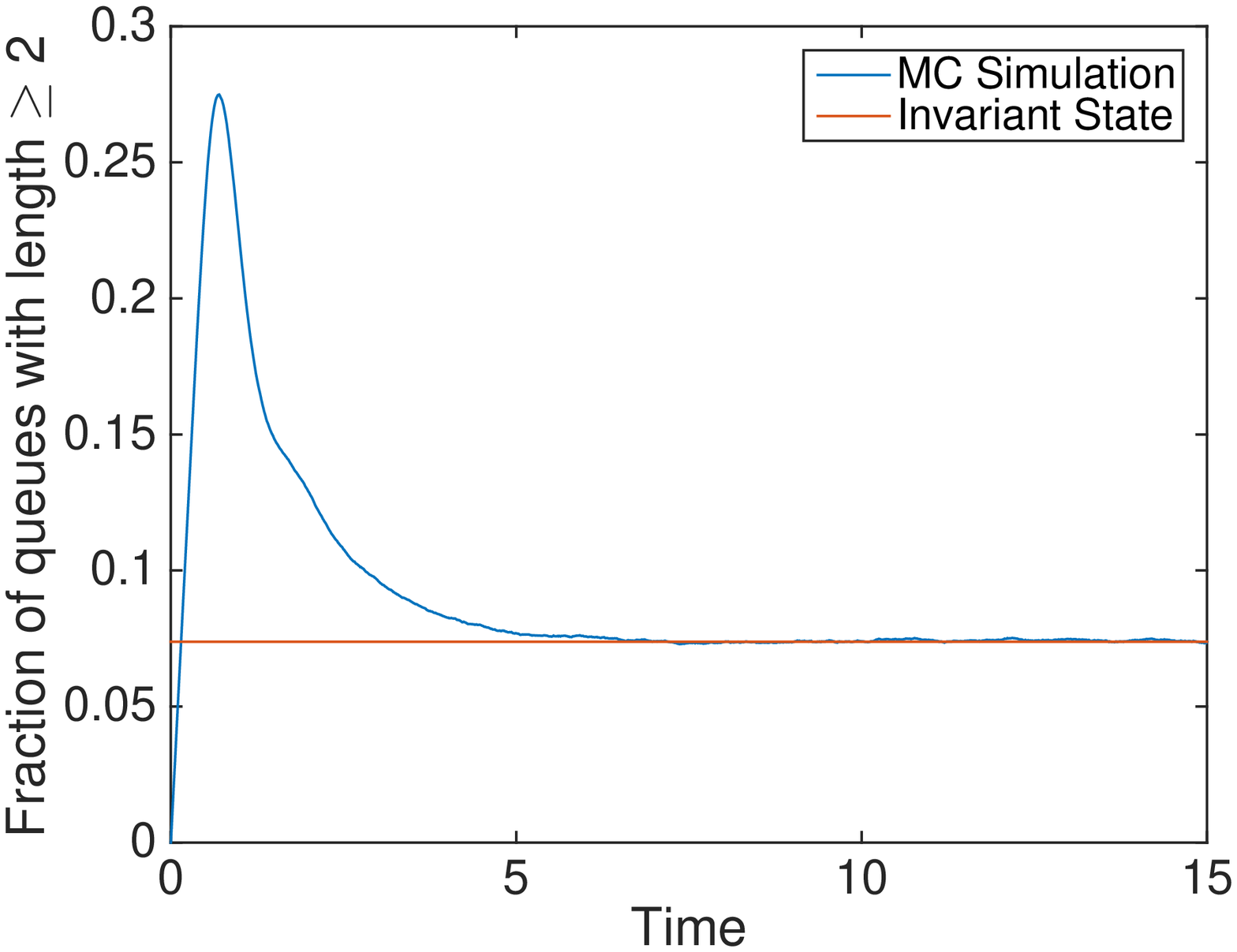} 
       \caption{$\ell=2$}
        \end{subfigure}
       \begin{subfigure}{0.3\textwidth}
       \centering
       \includegraphics[width=0.9\textwidth]{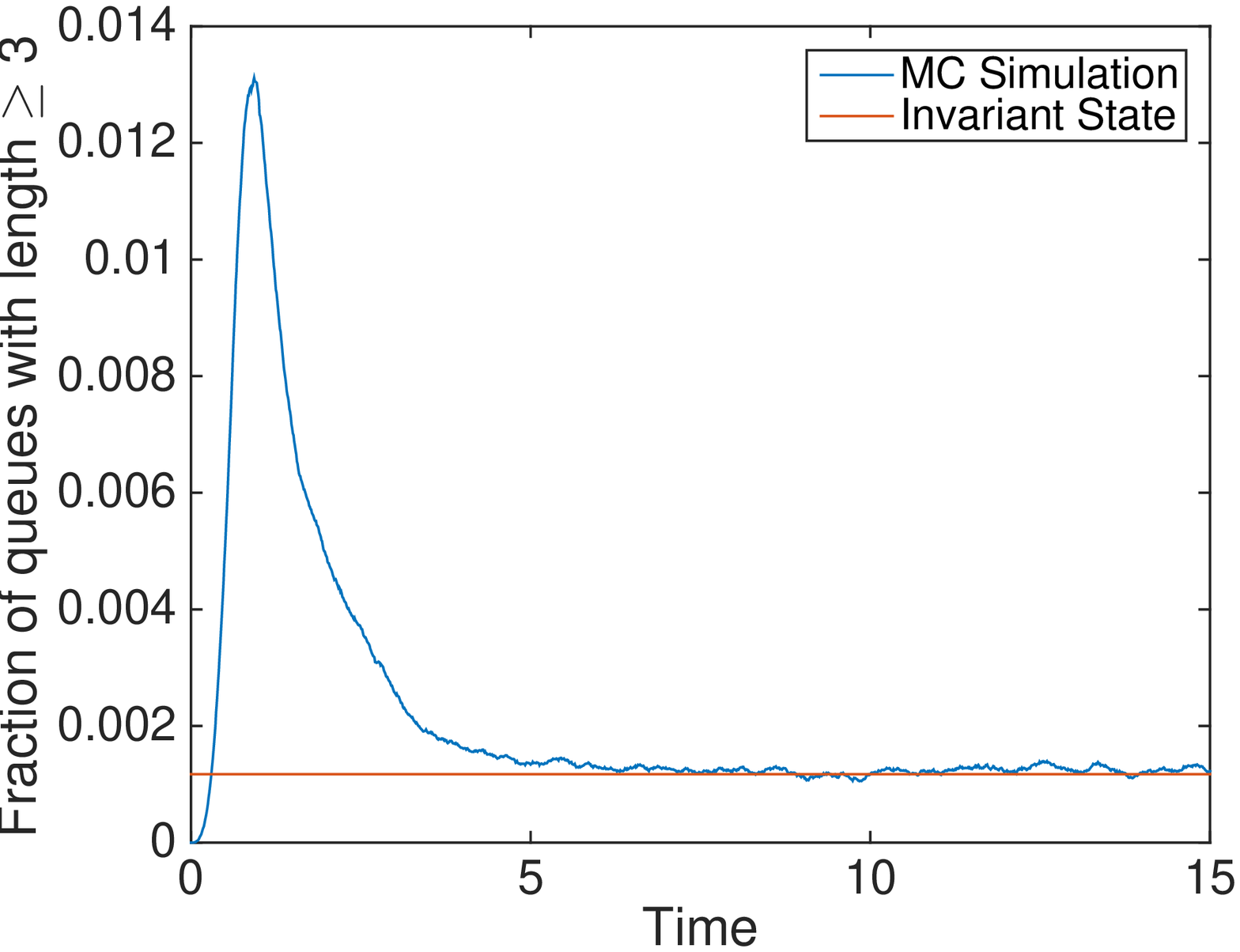} 
       \caption{$\ell=3$} 
 \end{subfigure}
  \caption{Comparison of the probability of the queue exceeding $\ell$ obtained 
    from MC simulation with the numerical approximation to the invariant state $s_\ell^*$ for the Lognormal distribution with $\sigma=1/3$}
  \label{fig1}
  \end{figure}

 \begin{figure}[h]
\begin{subfigure}{0.3\textwidth}
      \centering
      \includegraphics[width=0.9\textwidth]{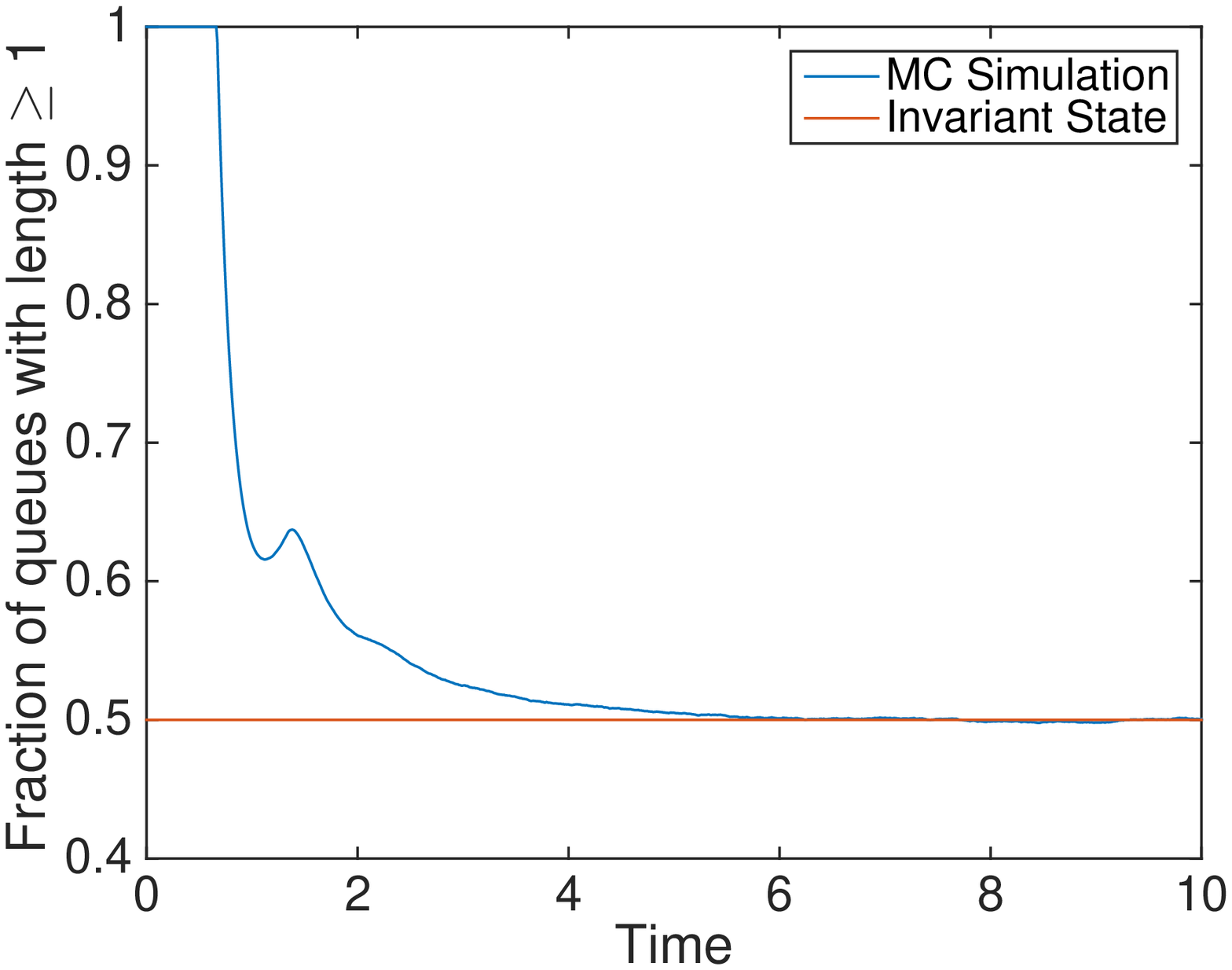} 
     \caption{$\ell=1$}
    \end{subfigure}
   \begin{subfigure}{0.3\textwidth}
       \centering
       \includegraphics[width=0.9\textwidth]{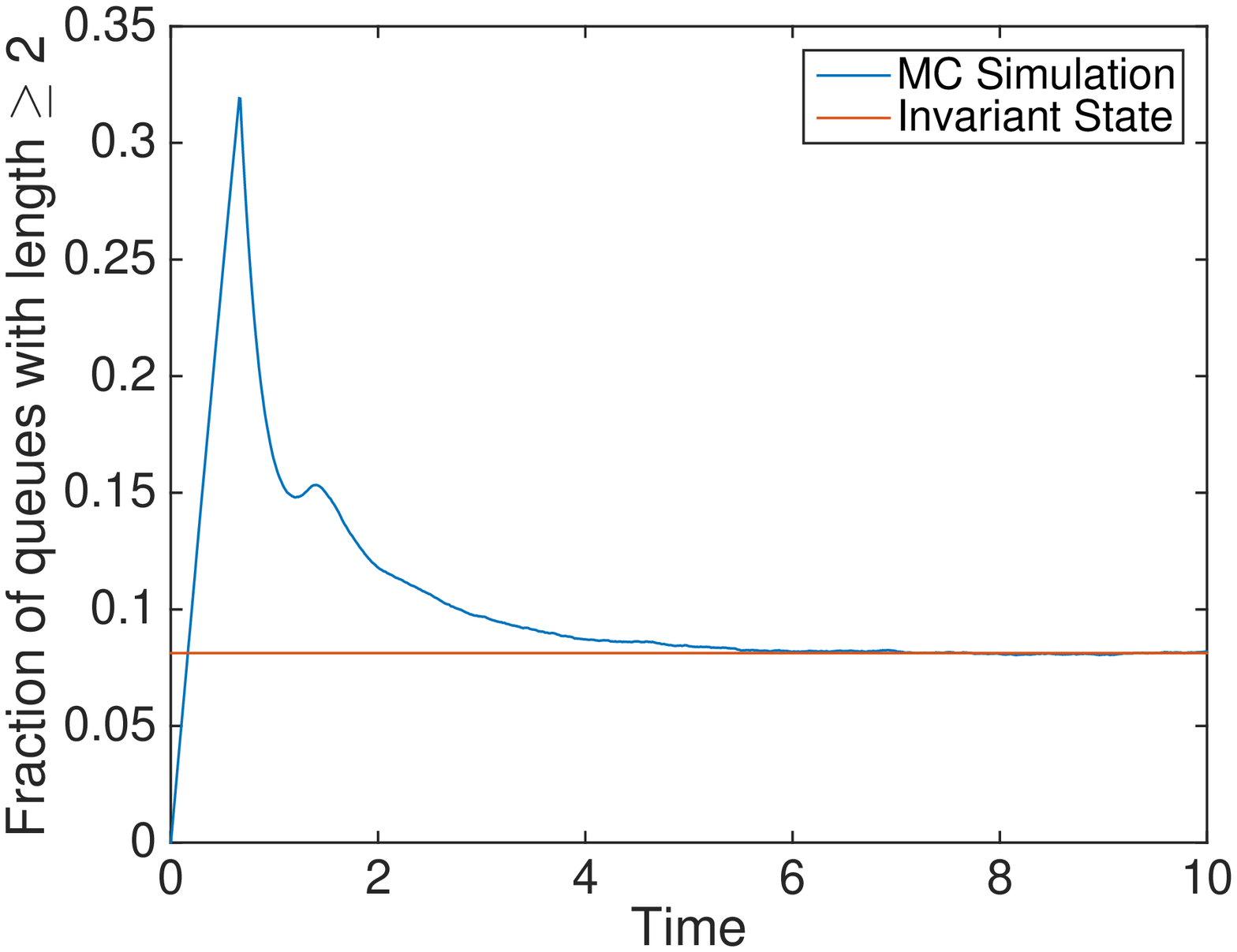} 
       \caption{$\ell=2$}
\end{subfigure}
  \begin{subfigure}{0.3\textwidth}
       \centering
       \includegraphics[width=0.9\textwidth]{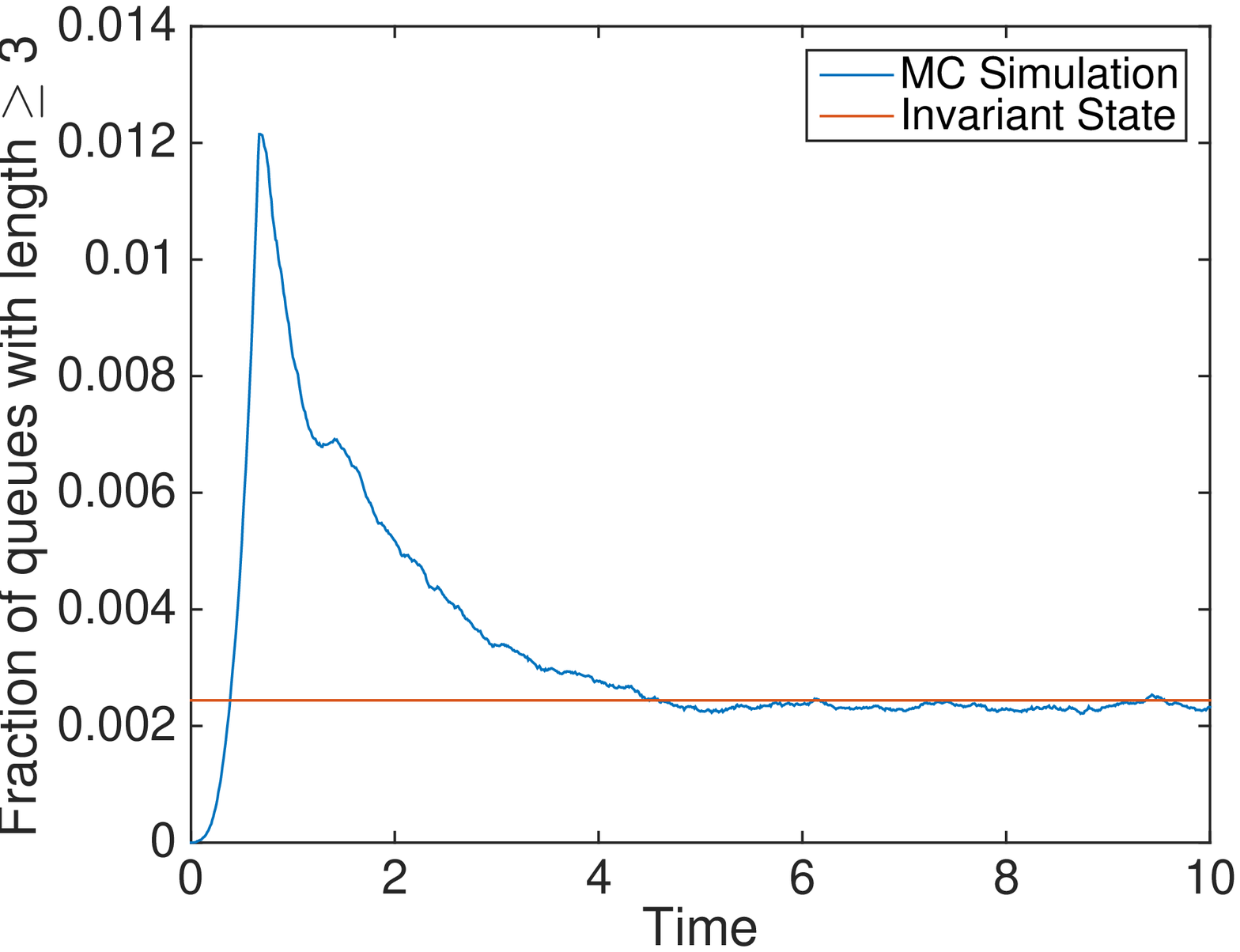} 
       \caption{$\ell=3$}
 \end{subfigure}
  \caption{Comparison of the probability of the queue exceeding $\ell$ obtained 
    from MC simulation with the numerical approximation to the invariant state$s_\ell^*$ for the Pareto distribution with $\alpha=3$}
  \label{fig2}
  \end{figure}

\section{Proof of Theorem \ref{th-reduction}}
\label{proof-thm}
Let $\nu^*=(\nu_\ell^*)_{\ell\in \N}$ be a physical invariant state of the hydrodynamic equations in the sense of
Remark \ref{rem-invstate}. 
Then (by definition)
$\nu(t)=(\nu_\ell^*)_{\ell\in \N}$, $t> 0$, is a solution to the hydrodynamic equations with initial condition $\nu^*$, and
condition \eqref{eq-tailcond} holds with $(\ssl)_{\ell\in \N}$ defined as in \eqref{def-ssl}. 
Define the associated invariant quantities 
\begin{equation}
\label{nu*}
\eta_\ell^*:=\left\{\begin{array}{ll}
\lambda(1-\langle\f1,\nu_1^*\rangle^d)\delta_0&\quad\text{if }\ell=1,\\
\lambda\mathfrak{P}_d\left(\langle\f1,\nu_{\ell-1}^*\rangle,\langle\f1,\nu_\ell^*\rangle\right)(\nu_{\ell-1}^*-\nu_\ell^*)&\quad\text{if }\ell\geq 2.
\end{array}\right.
\end{equation}
From \eqref{Fluid_D}, it follows that for every  $t\geq 0$, the departure
process with initial state $\nus$ satisfies 
\begin{equation}
\label{d*}
D_{\ell}(t)=\langle h,\nu_\ell^*\rangle t, \qquad\forall \ell\in\N,
\end{equation}
and using \eqref{Fluid_R}, the measure-valued routing process is given by
$\eta_\ell(t)=\eta_\ell^*$ for all $t\geq 0$.

The proof of Theorem \ref{th-reduction}, which is presented in Section \ref{pf-mainresult},
relies on several preliminary results. 
First, in Section \ref{subs-deprate}, we obtain a bound on the 
departure rate in an invariant state,  in Section \ref{sec-pfred} we obtain a useful characterization
of an invariant state (see Proposition \ref{prop-reduction}), and in
Section \ref{sec-exi} we establish existence of a physical invariant state.

\subsection{The departure rate in an invariant state}
\label{subs-deprate}

\begin{lemma}
\label{lem-hus}
Suppose $\nus = (\nu_\ell^*)_{\ell \in \N}$ is an
invariant state of the corresponding hydrodynamic equations, and let 
$\ssl = \langle \f1, \nu_\ell^*\rangle$ be the corresponding queue length distribution. 
 Then, for every $\ell \geq 1$,
\begin{equation}
\label{eq-hlambda}
\langle h, \nuslp\rangle = \lambda (\ssl)^d- \lambda  + \langle h, \nus_1\rangle.
\end{equation}
Moreover if  \eqref{eq-tailcond} holds, that is, $\ssl \rightarrow 0$ as $\ell \rightarrow \infty$, then
\begin{equation}
\label{hnu1-ineq}
\langle h, \nus_1\rangle\geq\lambda.
\end{equation}
\end{lemma}
\begin{proof}
Let $\nus = (\nu_\ell^*)_{\ell\in \N}$ be an
invariant state of the hydrodynamic equations. Then, 
 for every $\ell \geq 1$, \eqref{Fluid_bound} implies $\langle h, \nu_\ell^* \rangle < \infty$, and  \eqref{Fluid_Balance}, \eqref{nu*} and \eqref{d*} imply that 
\begin{equation}
\label{stat-hnu2}
\langle h, \nus_2\rangle - \langle h, \nus_1\rangle =- \lambda ( 1 - (\ss_1)^d), 
\end{equation}
and for $\ell \geq 2$, 
\begin{equation}
\label{stat-hnuls}
\langle h, \nuslp\rangle - \langle h, \nu_\ell^* \rangle= - \lambda \mathfrak{P}_d\Big( \sslm, \ssl\Big) (\sslm - \ssl)=-\lambda[(\sslm)^d - (\ssl)^d],
\end{equation}
where the last equality invokes \eqref{def_poly}.
Thus, defining $\ss_0 := 1$, we see that for $\ell \geq 1$, 
\begin{align}
\label{hnu-rel}
\langle h, \nuslp\rangle  & =  \sum_{i=1}^{\ell} \left(  \langle h, \nus_{i+1}
\rangle - \langle h, \nus_i \rangle \right) + \langle h, \nus_1
\rangle\nonumber \\
& =  -\lambda \sum_{i=1}^\ell  \left( (\ss_{i-1})^d - (\ss_{i})^d \right) + \langle
h, \nus_1 \rangle\nonumber \\
& =  \lambda (\ssl)^d - \lambda  + \langle h, \nus_1
\rangle,
\end{align}
which proves \eqref{eq-hlambda}. Furthermore, 
since $\langle h, \nus_{\ell+1} \rangle \geq
0$ for all $\ell$, and  $\lim_{\ell\rightarrow\infty}\ssl  =  0$,  \eqref{hnu-rel} implies \eqref{hnu1-ineq}.
\end{proof}

  \subsection{Characterization of  invariant states}
  
\label{sec-pfred}
We begin by stating a result establishing absolute continuity of the invariant state. Its  proof is relegated to Appendix \ref{subs-abscont}.
\begin{lemma}
\label{lem-abscont}
Let $\nu^*=(\nu_\ell^*)_{\ell\in \N}\in\mathbb{S}$ be an invariant state of the hydrodynamic equations. Then for every ${\ell\in \N}$, $\nu_\ell^*$ is absolutely
continuous with respect to Lebesgue measure on $[0,\infty)$. 
\end{lemma}

\begin{proposition}
\label{prop-reduction}
Suppose $\nus = (\nu_\ell^*)_{\ell\in \N}$ is an invariant state of the hydrodynamic equations, and
$\ssl$ is defined in terms of $\nu_\ell^*$ via \eqref{def-ssl} for each $\ell \geq 1$.   Then there exists a sequence $(\vsl)_{\ell\in \N}$, of measurable functions on $[0,\infty)$ that admit a version that is continuously differentiable on $[0,L)$,  such that  $r_\ell(x)=0$ for $x\notin [0,L)$,  \eqref{def-vsl} and \eqref{eqn-ssl} hold, 
\begin{equation}
\label{def-r1}
r_1\equiv s_1^*,\qquad \text{on }[0,L),
\end{equation}
and for  $\ell\geq 2$, and  $x\in[0,L)$,
\begin{align}
\label{vsl2}
    \vsl (x)   = & (\lambda (\ssl)^d+s_1^*-\lambda) e^{-\lambda \mathfrak{P}_d( \sslm, \ssl) x} + \lambda  \mathfrak{P}_d( \sslm, \ssl)\int_0^x   e^{-\lambda  \mathfrak{P}_d( \sslm, \ssl) (x-u)} r_{\ell-1} (u)
      du.
\end{align}
\end{proposition}
\begin{proof}

Since each $\nu_\ell^*$ is  absolutely continuous with respect to Lebesgue
measure by Lemma \ref{lem-abscont},   and $\bar{G}(x)>0$ for $x\in[0,L)$, 
there exists $(\vsl)_{\ell\in \N}$, a sequence of measurable functions on $[0,L)$ such that
$\frac{d\nu_\ell}{dx}(x)=r_\ell(x)\bar{G}(x)$. Hence,
 $\vsl(x)=\ind_{[0,L)}(x)\frac{1}{\bar{G}(x)}\frac{d\nu_\ell}{dx}(x)$ clearly satisfies 
    \eqref{def-vsl}, and   since $s_\ell^*=\langle\f1,\nu_\ell^*\rangle$, this immediately
    implies  \eqref{eqn-ssl}. 
Now using $\nu(t)=\nu^*$ in \eqref{Fluid_f}, and for $\ell = 1$,  combining 
\eqref{nu*}, \eqref{d*}, and \eqref{stat-hnu2}, we have for every $f\in\mathbb{C}_b[0,\infty)$ and $t\geq 0$,
\begin{align*}
\int_{0}^L f(x) \bar{G}(x) \vs_1(x)  dx  =&   \int_0^L f(x+t) \bar{G}(x+t) \vs_1 (x) dx + \langle h, \nus_2\rangle \int_0^t f(s) \bar{G}(s) ds + \lambda \left( 1  - (\ss_1)^d\right)  \int_0^t f(s) \bar{G}(s) ds \\
 =&  \int_0^L f(x+t) \bar{G}(x+t) \vs_1 (x) dx+ \langle h, \nus_1 \rangle \int_0^t f(s) \bar{G}(s) ds.
\end{align*} 
Recalling that $h=g/\bar{G}$, this implies that
\begin{align*}
  \int_0^L f(x) \bar{G}(x) &\left[ \vs_1 (x)  -   \ind_{\{x \geq t\}}\vs_1 (x-t)
-  \ind_{\{x < t\}} \left( \int_0^L g(s) \vs_1 (s) \, ds \right) \right]dx = 0. 
\end{align*}
A standard monotone class argument shows that this holds for every bounded  measurable function $f$. Hence, the above equation holds if and only if for a.e. $(x,t) \in [0,L) \times [0,\infty)$, 
\[  \vs_1 (x) =  \ind_{\{ x \geq t\}} \vs_1(x-t)+ \ind_{\{x < t \}}\left( \int_0^L g(s) \vs_1 (s) \, ds\right) . \]
It is easy to see that this implies that $r_1$ is a.s. equal to the constant $\int_0^L g(s)r_1(s) ds$, and with some abuse of notation, we denote this constant again by $r_1$. In turn, since both $\int_{[0,L)}g(x) dx=1$ 
and $\int_{[0,L)}\bar{G}(x) dx=1$, this implies that 
\begin{equation}
\label{h1s1}
\langle h,\nu_1^*\rangle=r_1\int_{[0,L)}h(x)\bar{G}(x)dx=r_1=r_1\int_{[0,L)}\bar{G}(x)dx=s_1^*, 
\end{equation}
 which proves \eqref{def-r1}.

Similarly, for $\ell \geq 2$, setting $\nu_\ell(0)=\nu_\ell^*$ in \eqref{Fluid_f}, and using the identity 
$\nu_\ell(t)=\nu_\ell^*$ for any $t$ along with \eqref{def-vsl}, 
\eqref{nu*} and \eqref{d*}, we have for every $f \in \mathbb{C}_b[0,\infty)$, and $t\geq 0$, 
\begin{align}
\label{fvl}
\int_0^L  f(x) \bar{G}(x) \vsl (x) dx =&  \int_0^L f(x+t) \bar{G}(x+t) \vsl (x) dx + \langle h, \nuslp \rangle \int_0^t f(s) \bar{G}(s)  ds\notag \\ 
& + \lambda \mathfrak{P}_d\Big( \sslm, \ssl\Big) \int_0^t \left( \int_0^L f(x+s) \bar{G}(x+s) 
\left( \vslm (x) - \vsl (x) \right) dx \right)  ds. 
\end{align}
  Since  $r_\ell$ is integrable with respect to $\bar{G}(x) dx$ by definition,  
we can use Fubini's theorem to change the order of integration in the last term  on the right-hand side of  \eqref{fvl}, 
and apply the identity $\bar{G}(x)=0$ for $x>L$, to obtain 
\begin{align*}
   \int_0^t \left( \int_{s\wedge L}^L f(x) \bar{G} (x) \left( \vslm (x-s) - \vsl (x-s) \right) dx \right) ds 
  = &  \int_0^{t\wedge L} f(x) \bar{G} (x) \left( \int_0^x \left( \vslm (x-s) - \vsl (x-s) \right) ds \right) dx \\
 &   +
\int_{t\wedge L}^L f(x) \bar{G} (x) \left( \int_0^t \left( \vslm (x-s) - \vsl (x-s) \right) ds \right) dx, 
\end{align*}
where we use the notation $a\wedge b$ to denote the minimum of $a$ 
and $b$. 
Thus, we can rewrite \eqref{fvl} as

\begin{equation}
\label{fvl2}
\int_0^L  f(x) \bar{G}(x) \left[r_\ell (x) -  \M_\ell^{(1)} (x,t) \ind_{\{x \geq t\wedge L\}} -  \M_{\ell}^{(2)} (x) \ind_{\{x < t\wedge L\}} \right]dx= 0,  
\end{equation}
where, for $(x,t) \in [0,L) \times [0,\infty)$,  
\begin{align*}
  \M_\ell^{(1)} (x,t)  & :=     \vsl (x-t)  + \lambda  \mathfrak{P}_d\Big( \sslm,\ssl\Big) \int_0^t \left( \vslm (x-s) - \vsl (x-s) \right) ds, \\
\M_\ell^{(2)} (x) & := \langle h, \nuslp \rangle + \lambda  \mathfrak{P}_d\Big( \sslm, \ssl\Big) 
\int_0^x \left( \vslm (u) - \vsl (u) \right) du. 
\end{align*}
Now,  \eqref{fvl2} clearly holds if and only if for every $t\geq 0$,    
$r_\ell (x) = \M_\ell^{(1)} (x,t) \ind_{\{x \geq t\wedge L\}} + \M_{\ell}^{(2)} (x) \ind_{\{x < t\wedge L\}}$  for a.e. $x \in [0,L)$, which in turn holds if and only if  for every $t\geq 0$,
\begin{equation}
\label{eqn1}
   r_\ell (x) = \M_{\ell}^{(1)} (x,0) =  \M_{\ell}^{(1)} (x,t), \quad   \text{for a.e. }\quad x \in[t \wedge L,L), 
\end{equation} 
and 
\begin{equation}
\label{eqn2}
r_\ell (x) = \M_{\ell}^{(2)} (x), \quad \text{for a.e.}\quad x < t\wedge L. 
\end{equation}

We will first consider the implications of the equality in  \eqref{eqn2}. 
Define 
\[ I_\ell (x) := \int_0^x r_\ell (u) du, \]
and 
\begin{equation}
\label{def-bl}
B_\ell (x) := \left[\langle h, \nu_\ell^*\rangle - \lambda  \mathfrak{P}_d\Big( \sslm, \ssl\Big)  (\sslm - \ssl) \right]+ \lambda  \mathfrak{P}_d\Big( \sslm, \ssl\Big) I_{\ell-1} (x). 
\end{equation}
Note that $B_\ell$ is continuous on $[0,L)$, and using the definition of $\M_\ell^{(2)}$ and \eqref{stat-hnuls}, we can rewrite \eqref{eqn2} as 
\begin{equation}
\label{ode}
 \frac{d I_\ell}{dx}(x) + \lambda  \mathfrak{P}_d\Big( \sslm, \ssl\Big) I_\ell (x)  = B_\ell (x), \end{equation}
with boundary condition $I_\ell (0) = 0$.  Solving this linear ordinary differential equation, we see that  
\[  I_\ell (x) = e^{-\lambda \mathfrak{P}_d( \sslm, \ssl) x} \int_0^x 
e^{\lambda \mathfrak{P}_d ( \sslm, \ssl) u}  B_\ell (u) du.   
\]
Combining this with \eqref{ode} and  recalling that $r_\ell = dI_\ell/dx$, we obtain 
\begin{equation}
\label{vell-1}
 r_\ell (x) = B_\ell (x) - \lambda \mathfrak{P}_d\Big( \sslm, \ssl\Big) \int_0^x e^{-\lambda  \mathfrak{P}_d( \sslm, \ssl) (x-u)} 
B_\ell (u) du,
\end{equation}
which in particular shows that $r_\ell$ is continuous. Also, to further simplify the right-hand side, note  that 
 \eqref{def-bl} shows that $B_\ell$ is differentiable with continuous derivative $\lambda  \mathfrak{P}_d( \sslm, \ssl) r_{\ell-1}$.   
An  integration by parts then yields 
\begin{align*}  
\lambda \mathfrak{P}_d( \sslm, \ssl)\int_0^x e^{\lambda  \mathfrak{P}_d( \sslm, \ssl) u} B_\ell (u) du 
=& 
e^{\lambda  \mathfrak{P}_d( \sslm, \ssl) x} B_\ell (x) - B_\ell (0)- \lambda  \mathfrak{P}_d( \sslm, \ssl)\int_0^x e^{\lambda  \mathfrak{P}_d( \sslm, \ssl) u} r_{\ell-1}(u) du. \end{align*}
Substituting this back into \eqref{vell-1},  and using \eqref{def-bl}, \eqref{def_poly}, \eqref{eq-hlambda} and \eqref{h1s1} we obtain 
\begin{align*}
    \vsl (x)  = &e^{-\lambda \mathfrak{P}_d( \sslm, \ssl) x} \left[
      \langle h, \nu_\ell^*\rangle- \lambda \mathfrak{P}_d( \sslm,\ssl) (\sslm - \ssl) \right]+ \lambda e^{-\lambda \mathfrak{P}_d( \sslm, \ssl) x} \mathfrak{P}_d( \sslm, \ssl)\int_0^x
      e^{\lambda  \mathfrak{P}_d( \sslm, \ssl) u} \vslm (u) du
     \\
 =&  (\lambda (\ssl)^d+s_1^*-\lambda) e^{-\lambda \mathfrak{P}_d( \sslm, \ssl) x} + \lambda  \mathfrak{P}_d( \sslm, \ssl)\int_0^x   e^{-\lambda  \mathfrak{P}_d( \sslm, \ssl) (x-u)} r_{\ell-1} (u)
      du, 
\end{align*}
 which coincides with the identity in \eqref{vsl2}.  

To complete the proof, it only remains to  check that with this definition of $\vsl$,
\eqref{eqn1} is also satisfied, and that each $r_\ell$ is continuously differentiable.
First, note that for any $t \geq 0$ and
$x \in[t\wedge L,L)$, we have 
\begin{equation*}
  \int_0^t \vsl(x-s) \, ds  =   \int_{x-t}^x   \vsl (u) du  = I_\ell
  (x) - I_\ell(x-t).  
\end{equation*}
When combined with \eqref{ode} both as is and with $x$ replaced by $x-t$, and the relation $r_\ell=\frac{dI_\ell}{dx}$, this  yields 
\[  \lambda  \mathfrak{P}_d( \sslm, \ssl) \int_0^t \vsl(x-s) \, ds  =
B_\ell (x) - B_\ell (x-t) - \vsl (x) + \vsl (x-t), 
\]
and when combined with \eqref{def-bl}, this  implies  
\begin{align*} 
\lambda  \mathfrak{P}_d( \sslm, \ssl) \int_0^t \vslm(x-s) ds 
&=  \lambda  \mathfrak{P}_d( \sslm, \ssl)  
(I_{\ell-1} (x) - I_{\ell-1}(x-t)) \\
&= B_\ell (x) - B_\ell (x-t). 
\end{align*}
Together, the last two displays imply that  
\[   \lambda  \mathfrak{P}_d( \sslm, \ssl) \int_0^t (\vslm (x-s) -
\vsl (x-s) )\, ds = \vsl(x)  - \vsl (x-t), 
\]
which, given the definition of $\M_{\ell}^{(1)}$, proves that
\eqref{eqn1} is also satisfied.

Since $r_1$ is a constant function, it is clearly continuously differentiable and $r_1^\prime(x)=0$ for all $x\in[0,L)$. Given that $r_{\ell-1}$ is continuously differentiable on $[0,L)$, the continuous differentiability of $r_\ell$  on $[0,L)$ is  an immediate consequence of \eqref{vsl2}. 
The continuously differentiable property of $(r_\ell)_{\ell\in \N}$ on $[0,L)$ then follows from the principle of mathematical induction. \\
\end{proof}

\subsection{Existence of a physical invariant state}
\label{sec-exi}

In this section, we use the characterization obtained in Proposition \ref{prop-reduction}
to show that an invariant state always exists (see  Proposition \ref{prop-exist}).
The proof of existence relies on a few preliminary results. 
The first is a   monotonicity property of the sequence of functions $(r_\ell)_{\ell\in \N}$ that satisfy \eqref{vsl2}.

\begin{lemma}
\label{vprop}
  Given an invariant state $(\nu_\ell)_{\ell\in\N}$, let $(r_\ell)_{\ell\in \N}$ be the associated measurable functions on $[0,\infty)$ that are continuously differentiable on $[0,L)$, as described in Proposition \ref{prop-reduction}. 
Then, for any $\ell\geq 1$,
$$r_{\ell+1}(x)\leq\vsl(x) \qquad \mbox{ and } \qquad
 \vsl'(x)\geq 0, \qquad\mbox{ for } x\in (0,L)$$ 
\end{lemma}
\begin{proof} 
                      The fact that   $\nu^*$ lies in $\s$ and satisfies relation \eqref{def-vsl},
                       which was established in Proposition
                       \ref{prop-reduction}, 
                       imply that $\int_A r_{\ell+1} (x) \bar{G} (x) dx \leq \int_A r_{\ell}(x) \bar{G}(x) dx$ for all $A \in {\mathcal B}[0,\infty)$ and
                         $\ell \geq 1$. 
            Since,  by Assumption \ref{ass-G} and Proposition
            \ref{prop-reduction},
            $r_{\ell+1} \bar{G}$ and $r_\ell \bar{G}$ are  continuously differentiable on $[0,L)$,  
              (a straightforward generalization of) Proposition 2.23 of  \cite{Folland} shows that
              $r_{\ell+1} \leq r_\ell$ on $[0,L)$. 
           
              To show $r_\ell^\prime (x) \geq 0$, again using the
            differentiability  of $r_\ell$, and 
        differentiating both sides of  \eqref{vsl2}, we obtain 
\begin{align*}
\vs_{\ell}'(x)=&-\lambda(\lambda(\ssl)^d+s_1^*-\lambda)\mathfrak{P}_d(\sslm,\ssl)e^{-\lambda\mathfrak{P}_d(\sslm,\ssl)x}\\
&\quad -(\lambda\mathfrak{P}_d(\sslm,\ssl))^2\int_0^xe^{-\lambda\mathfrak{P}_d(\sslm,\ssl)(x-u)}\vslm(u)du +\lambda\mathfrak{P}_d(\sslm,\ssl)\vslm(x).
\end{align*}
Using \eqref{vsl2} to replace the second term on the right-hand side above, it follows that  for $x \in [0,L)$,
\begin{align*}
  \vs_{\ell}'(x)=&-\lambda(\lambda(\ssl)^d+s_1^*-\lambda)\mathfrak{P}_d(\sslm,\ssl)e^{-\lambda\mathfrak{P}_d(\sslm,\ssl)x}+\lambda\mathfrak{P}_d(\sslm,\ssl)(\lambda(\ssl)^d+s_1^*-\lambda)e^{-\lambda\mathfrak{P}_d(\sslm,\ssl)x}\\
  & -\lambda\mathfrak{P}_d(\sslm,\ssl)\vsl(x)+\lambda\mathfrak{P}_d(\sslm,\ssl)\vslm(x)\\
=&\lambda\mathfrak{P}_d(\sslm,\ssl)(\vslm(x)-\vsl(x)).
\end{align*}
Using the identity $r_{\ell-1} \leq r_\ell$ on $[0,L)$ proved above, this implies $\vsl'\geq 0$ on $(0,L)$, as desired. 
 \end{proof}

We now consider physical invariant states, namely those  that satisfy condition \eqref{eq-tailcond}. 

\begin{lemma}
\label{lem-s1eqbm}
Let $\nu^* = (\nu_\ell^*)_{ \ell \in \N}$ be a physical invariant
state of the hydrodynamic equations, that is, for which the associated $(s^*_\ell)_{\ell \in \N}$ satisfies
$\lim_{\ell\rightarrow\infty}\ssl=0$,  with $\ssl$ as in \eqref{def-ssl}. 
Then $\langle h,\nu^*_1 \rangle =s_1^*= \lambda$. 
\end{lemma}
\begin{proof}  Fix an invariant state, and 
   let $(r_\ell)_{\ell \in  \N}$ be the associated sequence of functions 
   described in Proposition \ref{prop-reduction}. 
By Lemma \ref{vprop} for every $x\in [0,L)$, $r_\ell(x)$ is monotonically decreasing in $\ell$, and is bounded below by $0$. Together with \eqref{def-r1}, this shows $r_\ell(x)\leq r_1=s_1^*$, for all $x\in[0,L)$, $\ell \in \N$, and the limit
$r_\infty(x)=\lim_{\ell\rightarrow\infty}r_\ell(x)$ exists for all $x\in [0,L)$. 
      Since  $\bar{G}$ is integrable,  the dominated convergence theorem implies 
$\lim_{\ell\rightarrow\infty}\int_0^L r_\ell(x)\bar{G}(x)dx=\int_0^L r_\infty(x) \bar{G}(x)dx.$ 
On the other hand,  \eqref{eqn-ssl} and the tail condition $\lim_{\ell\rightarrow\infty}\ssl=0$ of \eqref{eq-tailcond} show that 
$\lim_{\ell\rightarrow\infty}\int_0^L r_\ell(x)\bar{G}(x)dx=\lim_{\ell\rightarrow\infty}s_\ell^*=0.$  
When combined, these limits imply 
$\int_0^L r_\infty(x) \bar{G}(x)dx=0.$ 
Since $\bar{G}(x)>0$ and $r_\infty(x)\geq 0$ for $x\in[0,L)$, this implies $r_\infty= 0$ a.e.. Likewise, using $h=g/\bar{G}$, \eqref{def-vsl},  the integrability of $g$, and the dominated convergence theorem, one has 
\begin{equation*}
\lim_{\ell\rightarrow\infty}\langle h,\nu_\ell^*\rangle=\lim_{\ell\rightarrow\infty}\int_0^L r_\ell(x)g(x)dx=\int_0^L r_\infty(x) g(x)dx=0.
\end{equation*}
Thus, taking the limit in equation \eqref{eq-hlambda} as $\ell\rightarrow\infty$, and using the fact that $s_\ell^*\rightarrow 0$, 
we see that
\begin{align*}
\lim_{\ell\rightarrow\infty}\langle h,\nu_\ell^*\rangle=&-\lambda+\langle h,\nu_1^*\rangle.
\end{align*}
Comparing the last two equations, 
we conclude that
$\langle h,\nu_1^*\rangle=\lambda.$   Since \eqref{def-vsl} and  \eqref{def-r1} imply $\langle h, \nu_1^* \rangle =
  s_1^* \int_{[0,L)} g(x) dx = s_1^*$, this completes the proof of the lemma. 
\end{proof}

  Let $\C_b^1[0,\infty)$ denote the space of bounded functions on  $[0,\infty)$ that are continuously differentiable on $[0,L)$. Fix $d \geq 2$, $\lambda \in (0,1)$ and
  consider   the associated map $\bar{F} = F^{(d)}:\C_b^1[0,\infty) \times (0,1) \mapsto \R_+$ defined as follows:
    for $r\in\C_b^1[0,\infty)$ define 
  \begin{equation}
    \label{eq-F}
    \bar{F} (r, s) = \lambda s^d \int_0^L e^{-\lambda  \mathfrak{P}_d( \bar{s},  s)  x} \bar{G}(x) dx
    + \lambda  \mathfrak{P}_d\Big( \bar{s},  s\Big) \int_0^L
      \left(\int_0^x e^{-\lambda \mathfrak{P}_d (\bar{s}, s) (x-u)} r (u) du \right)
      \bar{G}(x) dx, 
  \end{equation}
   where $\bar{s} = \bar{s}(r)$ is defined by 
  \begin{equation}
    \label{bars}
    \bar{s}(r)  := \int_0^L r(x) \bar{G}(x) dx,\end{equation}
Then the following criterion for existence of a physical invariant state is an immediate consequence of Proposition \ref{prop-reduction} and Lemma \ref{lem-s1eqbm}. 

\begin{proposition}
  \label{prop-exist}
  Suppose that whenever either $r \equiv \lambda$ or when $r$ is non-constant and 
  \begin{equation}
    \label{rcond}
    r(0) = \lambda \left(\int_0^L r(x) \bar{G}(x) dx \right)^d,
  \end{equation}
  then $s \mapsto \bar{F}(r, s)$ has at least one fixed point on $(0,\bar{s}(r))$. 
    Then there exists a physical  invariant state for the hydrodynamic equations with
    arrival rate $\lambda$.
    \end{proposition}
\begin{proof}  
Assume the supposition of the proposition holds. 
By Proposition \ref{prop-reduction} and Lemma \ref{lem-s1eqbm},
any physical invariant state must satisfy $s_1^* = \lambda$ and $r_1 \equiv \lambda$. 
  On the other hand,  \eqref{eq-hlambda} and Lemma \ref{lem-s1eqbm} imply 
  that $\langle h, \nuslp\rangle=\lambda (\ssl)^d$. When substituted into \eqref{vsl2} and \eqref{def-vsl}, this implies \eqref{eqn-vsl} and \eqref{eqn-vslbc2} must hold.
  Thus,  to show existence of an invariant state
it suffices to find sequences of continuously differentiable functions $(r_\ell)_{\ell \geq 2}$
and positive constants $(s_\ell^*)_{\ell \geq 2}$ that satisfy \eqref{eqn-vsl} and \eqref{eqn-ssl}. 
For $\ell = 2$, define $F_2 (s) = \bar{F}(r, s)$, with $r \equiv s_1^* = \lambda$,  and let $s_2^*$ be
a fixed point of $F_2$ in $(0,\lambda]$, which exists by the claim, and
let $r_2$ be as defined in terms of $s_2^*$ via \eqref{eqn-vsl}.  
Then $s_2^*$ clearly satisfies \eqref{eqn-ssl} due to \eqref{bars},
and there is a unique $s_2^*$ satisfying \eqref{eqn-ssl} with $r_2$ satisfying \eqref{eqn-vsl}
if and only if the fixed point of $F_2$ is unique.   
Now, suppose that for some $j \geq 2$,  there exist continuously differentiable
functions $(r_\ell)_{2 \leq \ell \leq j}$ and constants $(s_\ell^*)_{2 \leq \ell \leq j}$
that satisfy \eqref{eqn-vsl} and \eqref{eqn-ssl} for $2 \leq \ell \leq j$. 
Then \eqref{eqn-vsl} and \eqref{eqn-ssl} imply that $r = r_j$ satisfies \eqref{rcond}.
Therefore, setting $F_{j+1}(\cdot) = \bar{F}(r_j, \cdot)$, the claim above shows that
$F_{j+1}$ has a fixed point that we denote by $s_{j+1}^*$, and let $r_{j+1}$ be as defined in
\eqref{eqn-vsl}, and note that then \eqref{eqn-ssl} holds on account of \eqref{bars}.
Thus, we have constructed $(r_\ell)_{2 \leq \ell \leq j+1}$ and $(s_\ell^*)_{2 \leq \ell \leq j+1}$
that satisfy \eqref{eqn-vsl} and \eqref{eqn-ssl} for all $\ell \leq j+1$, and the first
assertion of the proposition follows by induction.  
\end{proof}

\subsection{Proof of Theorem \ref{th-reduction}}
\label{pf-mainresult}

\begin{proof}[Proof of Theorem  \ref{th-reduction}]
 To  establish Theorem \ref{th-reduction}
 it suffices to verify the supposition of Proposition \ref{prop-exist}.
 To see why the latter is true, fix $r$ that satisfies the stated conditions, and  define
$J (s) := \bar{F}(r,s) - s$. Note that then we need to show that 
$J$ has at least one zero on $(0,s)$. Note that  by \eqref{def_poly},
  $\mathfrak{P}_d (\bar{s}, s) > 0$ for all $s \in [0,1]$ whenever $\bar{s}  > 0$.  Since $r$ is not identically zero and
  $\bar{G}$ is strictly positive on $[0,L)$, \eqref{eq-F} implies that $\bar{F}(r,0) > 0$ and hence, $J(0) > 0$.
  Since $J$ is continuous, by
  the intermediate value theorem, to show $J$ has a zero on $(0,\bar{s}]$, it suffices to show that $J(\bar{s}) < 0$. 
Note that a simple integration by parts yields
  \begin{align}
    \label{IBP}
    \bar{F}(r, s)= & \int_0^L r (x) \bar{G}(x) dx -\left[r(0) -\lambda s^d\right]\int_0^Le^{-\lambda\mathfrak{P}_d(\bar{s},s)x}\bar{G}(x)dx-\int_0^L\left(\int_0^xe^{-\lambda\mathfrak{P}_d(\bar{s},s)(x-u)}r'(u)du\right)\bar{G}(x)dx. 
  \end{align}
  If $r \equiv \lambda$, then by \eqref{bars}, $\bar{s} = \lambda$ and \eqref{IBP} shows that
  $J(\bar{s}) = -\lambda (1 - s^d) \int_0^L e^{-\lambda\mathfrak{P}_d(\bar{s},s)x}\bar{G}(x)dx$.
  Thus,  $J(\bar{s}) \leq 0$ (with equality holding only if $\lambda = 1$) and so
  $J$ has a zero on $(0,\bar{s})$ if $\lambda < 1$ and at $\bar{s}$ if $\lambda = 1$.
On the other hand, if $r$ satisfies \eqref{rcond}, then \eqref{IBP} and \eqref{bars} show that 
\begin{align*} 
J(\bar{s})    = &    -\int_0^L\left(\int_0^xe^{-\lambda\mathfrak{P}_d(\bar{s},s)(x-u)}r'(u)du\right)\bar{G}(x)dx < 0, 
\end{align*}
where the last inequality also uses the fact that $r$ is non-constant.
This proves the first assertion  of Theorem \ref{th-reduction}.  Also, in view of \eqref{def-vsl}, 
the right-hand side of \eqref{eqn-vslbc2} is equal to $\langle h, \nu_\ell^*\rangle$. 
Therefore, \eqref{def-vsl} follows from  Proposition \ref{prop-reduction} and Lemma \ref{lem-s1eqbm}.  This completes the proof of Theorem \ref{th-reduction}.
   \end{proof}

\section{Tail Decay for the Invariant State}
\label{sec-decay}
Throughout this section recall that we fix $d\geq 2$, $\lambda \in (0,1)$, suppose Assumptions \ref{ass-G}-\ref{ass-service} hold and let $\beta>d/(d-1)$, $x_0>1/\lambda d$ and $C_0<\infty$ be as in Assumption \ref{ass-service}. Also, let $(s_\ell^*)_{\ell\in \N}$ be the queue length
distribution associated with a physical invariant state $(\nusl)_{\ell\in \N}$ of the hydrodynamic equations. Note that the analysis here does not explicitly require uniqueness of the physical invariant state, and only relies on the characterization stated in Theorem \ref{th-reduction} (and proved in Section \ref{proof-thm}).
\subsection{A reduction}
\label{subsec-reduc}
In this section we  show that the proof of Theorem \ref{thm-dec} can be reduced to establishing certain key estimates stated in Proposition \ref{prop-upbnd} below. We first introduce some relevant notation. Define
\begin{equation}
\label{def-hatbeta}
\hat{\beta}:=\beta-\lfloor \beta \rfloor,
\end{equation}
where $\lfloor \beta \rfloor$ denotes the largest integer less than or equal to $\beta$. Throughout this section, given $m\in\N$, $n\in\N$ and $x_i\in\R$ for all $i$,  we will use the convention that 
$$ \prod_{i=m}^n x_i=1\quad\mbox{ if } m>n.$$

\begin{proposition}
\label{prop-upbnd}
Suppose Assumptions \ref{ass-G} - \ref{ass-service}  hold. Let  $(s_\ell^*)_{\ell \in\N}$ be the queue length
distribution associated with a physical invariant state of the hydrodynamic equations. If $\beta$ is not an integer, then there exists  $1\leq C_\beta<\infty$ such that for all $\ell\geq 2$, 
\begin{equation}
\label{ineq-notinteger}
\ssl\leq \lambda  (3C_\beta)^{\ell-1}\left(\prod_{i=\ell-\lfloor\beta\rfloor +1}^{\ell-1}(s_i^*)^{d-1}\right)(s_{\ell-\lfloor\beta\rfloor}^*)^{\hat{\beta}(d-1)}.
\end{equation}
If $\beta$ is an integer, then there exists  $\delta\in (0,1)$, 
and $1\leq C_\beta<\infty$ such that for all $\ell\geq 2$,
\begin{equation}
\label{ineq-integer}
\ssl\leq \lambda(3C_\beta)^{\ell-1}\left(\prod_{i=\ell-\beta +2}^{\ell-1}(s_i^*)^{d-1}\right)(s_{\ell-\beta+1}^*)^{(1-\delta)(d-1)}.
\end{equation}
\end{proposition}
The proof of Proposition \ref{prop-upbnd}, which involves a careful analysis of the equations characterizing the invariant state, is lengthy and deferred to Section \ref{subsec-proof dec}.  We now use this result to prove Theorem \ref{thm-dec} in two steps. First, in Lemma \ref{prop-Lmin}, we show that Proposition \ref{prop-upbnd} directly implies that the sequence $\{\ssl\}_{\ell\in\N}$ has at least an exponentially fast decay rate. We then combine this {\em{a priori}} estimate with a result on inhomogeneous recursions (Lemma \ref{prop-dec} below) to show that the decay rate is at least  doubly exponential.

\begin{lemma}
\label{prop-Lmin}
Suppose Assumptions \ref{ass-G} -- \ref{ass-service}  hold. Let  $(s_\ell^*)_{\ell \in\N}$ be the queue length
distribution associated with a physical invariant state of the hydrodynamic equations and let $C_\beta\in[1,\infty)$ and $\delta\in(0,1)$ be as in Proposition \ref{prop-upbnd}. If $\beta$ is not an integer, then there exists $M(\beta)<\infty$ such that for all $\ell\geq M(\beta)$,
\begin{equation}
\label{Lmin-noint}
\lambda  (3C_\beta)^{\ell-1}\left(\prod_{i=\ell-\lfloor\beta\rfloor +1}^{\ell-1}(s_i^*)^{d-1}\right)(s_{\ell-\lfloor\beta\rfloor}^*)^{\hat{\beta}(d-1)}<1.
\end{equation}
If $\beta$ is an integer, then there exists $M(\beta)<\infty$ such that for all $\ell\geq M(\beta)$,
\begin{equation}
\label{Lmin-int}
 \lambda(3C_\beta)^{\ell-1}\left(\prod_{i=\ell-\beta +2}^{\ell-1}(s_i^*)^{d-1}\right)(s_{\ell-\beta+1}^*)^{(1-\delta)(d-1)}<1.
 \end{equation}
\end{lemma}
\begin{proof}
First note that Proposition \ref{prop-upbnd} implies that for some $C_\beta\in[1,\infty)$, $\delta\in(0,1)$ and $\ell\geq 2$,
\begin{align}
\label{apriori-upbnd}
\ssl\leq \left\{\begin{array}{lr}
\lambda  (3C_\beta)^{\ell-1}(s_{\ell-\lfloor\beta\rfloor}^*)^{(\beta-1)(d-1)}&\qquad\text{if }\beta\notin\N,\\
\lambda  (3C_\beta)^{\ell-1}(s_{\ell-\beta+1}^*)^{(\beta-1-\delta)(d-1)}&\qquad\text{if }\beta\in\N.
\end{array}
\right.
\end{align}
We consider two cases:\\
\textit{Case 1}: $\beta$ is not an integer. In this case, define 
\begin{equation}
\label{def-alpha}
\alpha:=(\beta-1)(d-1)-1,
\end{equation} 
and
\begin{equation}
\label{def-k}
k:=\frac{2}{\alpha}\log (3C_\beta).
\end{equation}
Note that $\beta>\frac{d}{d-1}$ and $C_\beta>1$ imply $\alpha>0$ and $k>0$. 
Let 
\begin{equation}
\label{def-L}
L:=\left\lceil\frac{(\lfloor\beta\rfloor+1+(\lfloor\beta\rfloor-1)\alpha)2}{\alpha}\right\rceil+1,
\end{equation}
where $\lceil x\rceil$ denotes the smallest integer greater than or equal to $x$. 
Now, define $c_k:=e^{k(L-1)}$. Since $k>0$ and $\ssl\leq 1$ for all $\ell$, 
\[\ssl\leq c_ke^{-k(L-1)}\leq c_k e^{-k\ell}\qquad \forall\ell\leq L-1.\]
As the inductive hypothesis, assume that there exists $m\geq L-1$ such that
\begin{equation}
\label{si-induct}
s_i^*\leq c_k e^{-ki},
\end{equation}
for all $i\leq m$. Applying \eqref{apriori-upbnd} with $\ell=m+1$ for the case $\beta\notin\N$, using \eqref{def-alpha}, \eqref{def-k} and the induction hypothesis \eqref{si-induct}, we see that $s_{m+1}^*$ is bounded above by 
\begin{align*}
 \lambda e^{\frac{\alpha}{2}km}(s_{m+1-\lfloor\beta\rfloor}^*)^{1+\alpha}
\leq & \lambda e^{\frac{\alpha}{2}km}c_ke^{-k(m+1-\lfloor\beta\rfloor)(1+\alpha)}
=c_ke^{-k(m+1)}\lambda e^{k(-\frac{\alpha m}{2}-\alpha+\lfloor\beta\rfloor(1+\alpha))}
\leq c_k e^{-k(m+1)},
\end{align*}
where the last inequality uses $\lambda\leq 1$, $m\geq L-1$ and \eqref{def-L}.
Thus, by the inductive hypothesis, \eqref{si-induct} holds for all $i\in\Z$. Together with the fact that $i\rightarrow s_i^*$ is non-decreasing, \eqref{def-alpha} and \eqref{def-k}, this implies
\begin{align*}
 \lambda  (3C_\beta)^{\ell-1}\left(\prod_{i=\ell-\lfloor\beta\rfloor +1}^{\ell-1}(s_i^*)^{d-1}\right)(s_{\ell-\lfloor\beta\rfloor}^*)^{\hat{\beta}(d-1)}
 \leq & \lambda  (3C_\beta)^{\ell-1}(s_{\ell-\lfloor\beta\rfloor}^*)^{(\beta-1)(d-1)}\\
 \leq & \lambda c_ke^{k(\lfloor\beta\rfloor(1+\alpha)-\frac{\alpha}{2})}e^{-k\ell(1+\frac{\alpha}{2})}.
\end{align*}
Since 
$\lim_{\ell\rightarrow\infty}e^{-k\ell(1+\frac{\alpha}{2})}=0$,
there exists $M(\beta)<\infty$ such that \eqref{Lmin-noint} holds for all $\ell\geq M(\beta)$.

\textit{Case 2}: $\beta$ is  an integer. In this case the proof of \eqref{Lmin-int} follows the same argument as in the non-integral case but with $\alpha:=(\beta-1-\delta)(d-1)-1$, $k:=\frac{2}{\alpha}\log (3C_\beta)$ and 
\begin{equation*}
L:=\left\lceil\frac{(\beta+(\beta-\frac{3}{2})\alpha)2}{\alpha}\right\rceil+1,
\end{equation*}
where note that $\beta\geq 2$ (since $\beta>d/(d-1)$ and $\beta\in\N$) and $C_\beta>1$ imply $\alpha>0$ and $k>0$.
\end{proof}

\begin{lemma}
\label{prop-dec}
Given strictly positive numbers $c_1$, $c_2$ and $(R_\ell)_{\ell\geq -j+1}$, for every $\eta \in(0,1)$ and $j\in\N$ such that $(d-1)(j+\eta-1)>1$, there exists $n_d=n_d(j,\eta)\in (0,\infty)$ such that if 
\begin{equation}
\label{def-Rl}
 R_\ell =c_1-(\ell-1)c_2+(d-1) \left(\sum_{i=\ell-j+1}^{\ell-1}  R_i +\eta R_{\ell-j}\right) \quad \text{for } \ell\geq1,
 \end{equation}
 then  
 \begin{equation*}
 \lim_{\ell\rightarrow\infty} \frac{1}{\ell} \log_d R_\ell = n_d.
 \end{equation*}
 \end{lemma}
 \begin{proof}
 Subtracting $R_{\ell-1}$ from $R_\ell$ and using \eqref{def-Rl}, we see that for $\ell\geq 2$,
\begin{align*}
 & R_\ell=-c_2 + R_{\ell-1} + (d-1)\left(R_{\ell-1}-(1-\eta) R_{\ell-j}-\eta R_{\ell-j-1}\right).
\end{align*}
Subtracting $R_{\ell-1}$ from $R_\ell$ again,  using the above equation, and rearranging, we have for $\ell\geq 3$,
\begin{align*}
R_\ell=(d+1)R_{\ell-1}-dR_{\ell-2}+ (d-1)\left(-(1-\eta)R_{\ell-j}+(1-2\eta)R_{\ell-j-1}+\eta R_{\ell-j-2}\right).
\end{align*}
 Define $F=F_{j,\eta}:\R\mapsto\R$ by
\begin{equation}
\label{def-polyF}
F(x):=1-(d+1)x+dx^2-(d-1)(-(1-\eta)x^{j}+(1-2\eta)x^{j+1}+\eta x^{j+2}).
\end{equation}
Noting that $0$ is not a root of $F$, let ${1}/{\gamma_i}$ be the $k$ distinct roots of $F$, with $k\leq j+2$, and let $d_i$ denotes the multiplicity of the root ${1}/{\gamma_i}$. Then
$F(x)=\sum_{i=1}^k(1-\gamma_ix)^{d_i}.$
Since $j\geq 1$ and $\eta (d-1)\neq 0$, by Theorem 4.1.1 of \cite{Stan86}, 
\begin{equation}
\label{Rell-expansion}
R_\ell=\sum_{i=1}^k \phi_i(\ell)\gamma_i^\ell, \qquad\forall \ell\geq 0,
\end{equation}
where 
 $\phi_i(\ell)$ is a polynomial in $\ell$ of degree no less than $d_i$. 
 
 A direct inspection of \eqref{def-polyF} shows that $F(1)=0$, the root $\gamma_1=1$ has degree $d_1=2$ , and $F(x)=(1-x)^2 P(x)$, where
\[P(x)=\left(1-(d-1)\sum_{i=1}^{j-1}x^i-(d-1)\eta x^{j}\right).\]
 By Descartes' rule of signs, there exists exactly one positive root of $P(\cdot)$. Let us denote it by $1/\gamma_2$. Note that $P(0)=1$ and $P(1)<0$ since $(d-1)(j+\eta - 1)>1$ by assumption. Thus,  $1/\gamma_2\in(0,1)$ or equivalently $\gamma_2>1$.  Using Rouche's theorem (see \cite[Exercise 237, page 321]{NarNie12}) it is easy to show that $P(\cdot)$ has no roots other than $1/\gamma_2$ in the disc $\{z:|z|\leq 1/\gamma_2\}$. Hence, $|\gamma_i|<\gamma_2$ for $i\geq 3$.  In view of the representation for $R_\ell$ in \eqref{Rell-expansion}, this implies
\begin{equation*}
\lim_{\ell\rightarrow\infty}\frac{ \log_d R_\ell}{\ell}=\log_d \gamma_2.
\end{equation*}
Setting $n_d=\log_d\gamma_2$ proves the claim.
 \end{proof}
 
\begin{proof}[Proof of Theorem \ref{thm-dec}]
Fix a physical invariant state $(\nusl)_{\ell\in\N}$ of the hydrodynamic equations, and let $(\ssl)_{\ell\in\N}$ be the associated queue length distribution. 
We consider two cases:\\
{\emph{Case 1:}} $\beta$ is not an integer. Let $C_\beta\in [0,\infty)$ be as in Proposition \ref{prop-upbnd}. In this case, using the monotonicity of the logarithm function along with \eqref{ineq-notinteger}, we see that for $\ell\geq 2$
\begin{align*}
\log \left(\frac{1}{\ssl}\right)\geq & -\log\lambda-(\ell-1)\log (3C_\beta)+(d-1)\left(\sum_{i=\ell-\lfloor\beta\rfloor+1}^{\ell-1}\log \left(\frac{1}{s_i^*}\right)+\hat{\beta}\log \left(\frac{1}{s_{\ell-\lfloor\beta\rfloor}^*}\right)\right).
\end{align*} 
 By Lemma \ref{prop-Lmin}, there exists $M=M(\beta)<\infty$ such that 
$$ \lambda  (3C_\beta)^{\ell-1}\left(\prod_{i=\ell-\lfloor\beta\rfloor +1}^{\ell-1}(s_i^*)^{d-1}\right)(s_{\ell-\lfloor\beta\rfloor}^*)^{\hat{\beta}(d-1)}<1,\qquad \forall \ell\geq M.$$
Set $Q_{i}:=-\log (s_i^*)$   for $\{M-\lfloor\beta\rfloor,\cdots, M-1\}$, and recursively define
\begin{equation*}
Q_\ell:=-\log\lambda-(\ell-1)\log (3C_\beta)+(d-1)\left(\sum_{i=\ell-\lfloor\beta\rfloor+1}^{\ell-1}Q_i+\hat{\beta}Q_{\ell-\lfloor\beta\rfloor}\right), \qquad\forall \ell\geq M.
\end{equation*}
Then, applying Lemma \ref{prop-dec} with $j=\lfloor\beta\rfloor$, $\eta=\hat{\beta}$ (noting that $(d-1)(\beta-1)>1$), $c_1=-\log\lambda>0$, $c_2=\log(3C_\beta)>0$ and $R_\ell:=Q_{M+\ell-1}$ for $ \ell\geq -\lfloor\beta\rfloor+1$, we conclude there exists $n_d=n_d(\beta,\delta)\in(0,\infty)$ such that
\begin{equation*}
\lim_{\ell\rightarrow\infty}\frac{\log Q_\ell}{\ell}=n_d.
\end{equation*}
Since $\log\left(1/{s_i^*}\right)\geq Q_i$ for all $i\geq M$, this proves \eqref{dbdecay}, in this case.\\
{\emph{Case 2}:} $\beta$ is an integer. Let $C_\beta\in[0,\infty)$ and $\delta\in (0,1)$ be as in Proposition \ref{prop-upbnd}. Then \eqref{ineq-integer}  implies that for every $\ell\geq 2$, 
\begin{align*}
\log \left(\frac{1}{\ssl}\right)\geq & -\log\lambda-(\ell-1)\log (3C_\beta)+(d-1)\left(\sum_{i=\ell-\beta+2}^{\ell-1}\log \left(\frac{1}{s_i^*}\right)+(1-\delta)\log \left(\frac{1}{s_{\ell-\beta+1}^*}\right)\right).
\end{align*} 
By Lemma \ref{prop-Lmin}, there exists $M=M(\beta)<\infty$ such that 
$$ \lambda  (3C_\beta)^{\ell-1}\left(\prod_{i=\ell-\beta +2}^{\ell-1}(s_i^*)^{d-1}\right)(s_{\ell-\beta+1}^*)^{(1-\delta)(d-1)}<1,\qquad\forall\ell\geq M.$$
Now, set $Q_{i}:=-\log (s_i^*)$   for $i\in\{M-\beta+1,\cdots, M-1\}$ and recursively define
\begin{equation*}
Q_\ell=-\log\lambda-(\ell-1)\log (3C_\beta)+(d-1)\left(\sum_{i=\ell-\beta+2}^{\ell-1}Q_i+(1-\delta)Q_{\ell-\beta+1}\right), \qquad\forall \ell\geq M.
\end{equation*}
 Since $\beta$ is an integer, $(d-1)(\beta-1)>1$ implies $\beta\geq 2$.
 Since $(1-\delta)\in(0,1)$, $\beta-1\geq 1$, $d\geq 2$, we have $(d-1)(\beta-1-\delta)> 1$. Thus, applying Lemma \ref{prop-dec} with $j=\beta-1$, $\eta=1-\delta$ and $R_\ell=Q_{M+\ell-1}$ for $ \ell\geq -\beta$, we conclude there exists $n_d=n_d(\beta,\delta)\in(0,\infty)$ such that
\begin{equation*}
\lim_{\ell\rightarrow\infty}\frac{\log Q_\ell}{\ell}=n_d.
\end{equation*}
Since $\log\left(1/{s_i^*}\right)\geq Q_i$ for all $i\geq M$, this proves \eqref{dbdecay}, for Case 2 as well.
\end{proof}
\begin{remark}
The estimates for $(s^*_\ell)_{\ell \in \N}$  in \eqref{ineq-notinteger} and \eqref{ineq-integer} are analogous to the estimates for corresponding 
probabilities obtained in (3.5) and (3.6) of Proposition 3.2 in \cite{BraLuPra13}.  One may expect that the proof of the tail decay 
property in Theorem \ref{thm-dec} could therefore be deduced from Proposition 
\ref{prop-upbnd} by simply referring to the argument used in \cite{BraLuPra13} to 
deduce their tail decay result (in Theorem 1.1 of \cite{BraLuPra13}) from Proposition 3.2 therein,  which only involves comparison with a  homogeneous linear recursion (see Proposition 3.3 of \cite{BraLuPra13}).  However, we could not quite resolve the argument in \cite{BraLuPra13}, and instead provide a self-contained proof that entails a two-step argument presented above, that involves comparison with the slightly more complicated inhomogeneous recursion analyzed in Lemma \ref{prop-dec}. 
\end{remark}
\subsection{Preliminary Estimates on the Density of the Physical Invariant State}
\label{subsec-PreEst}
In this section we  obtain  preliminary estimates that are used to prove Proposition \ref{prop-upbnd} in Section \ref{subsec-proof dec}. 
\begin{lemma}
\label{prop-rineq}
Suppose Assumptions \ref{ass-G} and \ref{ass-G2} hold. For all $x\in[0,\infty)$, the inequality
\begin{equation}
\label{ineq-r2}
r_2(x)\leq \lambda (s_2^*)^d	+\lambda (1-e^{-\lambda d (s_1^*)^{d-1}x}),
\end{equation}
holds, and for $\ell\geq 3$,
\begin{align}
\label{upbndr_ell}
r_\ell(x)\leq&\lambda (\ssl)^d+\lambda (s_{\ell-1}^*)^d+\lambda\sum_{j=2}^{\ell-2}(s_j^*)^d\prod_{i=j+1}^{\ell-1}(1-e^{-\lambda d(s_i^*)^{d-1}x})+\lambda\prod_{i=1}^{\ell-1}(1-e^{-\lambda d(s_i^*)^{d-1}x}),
\end{align}
where the third term on the right-hand side of \eqref{upbndr_ell} is treated as zero if $\ell=3$.
\end{lemma}
\begin{proof}
Since $(s_\ell^*)_{\ell\in \N}$ is the queue length
distribution associated with a physical invariant state of the hydrodynamic equations, by Theorem \ref{th-reduction}, $r_1=s_1^*=\lambda$, and by convention $s_0^*=1$. Hence, \eqref{ineq-r2} follows from \eqref{eqn-vsl} with $\ell=2$, and the inequality $\mathfrak{P}_d(s_1^*,s_2^*)\leq d(s_1^*)^{d-1}$, which holds by \eqref{Pdmax}. For the proof of \eqref{upbndr_ell} we will make repeated use of the fact that the function $x\mapsto 1-e^{-cx}$ is increasing for any $c>0$. Using \eqref{eqn-vsl} again, but now with $\ell=3$, and using the inequality \eqref{ineq-r2}, we obtain
\begin{align*}
r_3(x)\leq&\lambda (s_3^*)^de^{-\lambda\mathfrak{P}_d(s_2^*,s_3^*)x}+\lambda^2(s_2^*)^d\mathfrak{P}_d(s_2^*,s_3^*)\int_0^xe^{-\lambda\mathfrak{P}_d(s_2^*,s_3^*)(x-u)} du\\
&+\lambda^2\mathfrak{P}_d(s_2^*,s_3^*)\int_0^xe^{-\lambda\mathfrak{P}_d(s_2^*,s_3^*)(x-u)} (1-e^{-\lambda d (s_1^*)^{d-1}u})du\\
\leq&\lambda (s_3^*)^d+\lambda(s_2^*)^d(1-e^{-\lambda\mathfrak{P}_d(s_2^*,s_3^*)x})+\lambda^2 (1-e^{-\lambda d (s_1^*)^{d-1}x})\mathfrak{P}_d(s_2^*,s_3^*)\int_0^xe^{-\lambda\mathfrak{P}_d(s_2^*,s_3^*)(x-u)}du\\
\leq&\lambda (s_3^*)^d+(\lambda(s_2^*)^d+\lambda (1-e^{-\lambda d (s_1^*)^{d-1}x}))(1-e^{-\lambda\mathfrak{P}_d(s_2^*,s_3^*)x}).
\end{align*}
Together with the inequality $\mathfrak{P}_d(s_2^*,s_3^*)\leq d(s_2^*)^{d-1}$ from \eqref{Pdmax}, this implies,
\begin{equation*}
r_3(x)\leq\lambda (s_3^*)^d+\lambda(s_2^*)^d+\lambda (1-e^{-\lambda d (s_1^*)^{d-1}x})(1-e^{-\lambda d(s_2^*)^dx}),
\end{equation*}
which proves \eqref{upbndr_ell} for $\ell=3$.

Now, suppose \eqref{upbndr_ell} holds for some $\ell\geq 3$. 
Then using \eqref{eqn-vsl} with $\ell$ replaced by $\ell+1$ and \eqref{upbndr_ell}, we have 
\begin{align*}
r_{\ell+1}(x)=&\lambda (s_{\ell+1}^*)^de^{-\lambda\mathfrak{P}(\ssl,s_{\ell+1}^*)x }+\lambda\mathfrak{P}_d(\ssl,s_{\ell+1}^*)\int_0^x e^{-{\lambda\mathfrak{P}(\ssl,s_{\ell+1}^*)(x-u) }}r_\ell(u)du\\
\leq & \lambda (s_{\ell + 1}^*)^d + \lambda^2\left[
  (s_{\ell}^*)^d + (s_{\ell - 1}^*)^d\right] \mathfrak{P}_d (s_{\ell}^*, s_{\ell + 1}^*)
  \int_0^x e^{- \lambda \mathfrak{P}_d (s_{\ell}^*, s_{\ell + 1}^*)
  (x - u)} du \\
  &  + \lambda^2 \sum_{j = 2}^{\ell - 2}\left[ (s_j^*)^d \mathfrak{P}_d
  (s_{\ell}^*, s_{\ell + 1}^*)  \int_0^{x} e^{- \lambda
  \mathfrak{P}_d (s_{\ell}^*, s_{\ell + 1}^*) (x - u)} \prod_{i = j
  + 1}^{\ell - 1} (1 - e^{- \lambda d (s_i^*)^{d - 1} u}) du\right]\\
  &  + \lambda^2 \mathfrak{P}_d (s_{\ell}^*, s_{\ell + 1}^*)\int_0^x e^{- \lambda \mathfrak{P}_d (s_{\ell}^*, s_{\ell + 1}^*)
  (x - u)} \prod_{i = 1}^{\ell - 1} (1 - e^{- \lambda d (s_i^*)^{d - 1}
  u}) du\\
    \leq & \lambda (s_{\ell + 1}^*)^d + \lambda
  (s_{\ell}^*)^d (1 - e^{- \lambda \mathfrak{P}_d (s_{\ell}^*,
  s_{\ell + 1}^*) x}) + \lambda (s_{\ell - 1}^*)^d (1 - e^{- \lambda
  \mathfrak{P}_d (s_{\ell}^*, s_{\ell + 1}^*) x})\\
  &  + \lambda^2 \sum_{j = 2}^{\ell - 2} \left[(s_j^*)^d \prod_{i = j + 1}^{\ell - 1} (1 -
  e^{- \lambda d (s_i^*)^{d - 1} x}) \mathfrak{P}_d
  (s_{\ell}^*, s_{\ell + 1}^*)\int_{0}^x e^{- \lambda
  \mathfrak{P}_d (s_{\ell}^*, s_{\ell + 1}^*) (x - u)} du\right]\\
  &  + \lambda^2 
  \prod_{i = 1}^{\ell - 1} (1 - e^{- \lambda d (s_i^*)^{d - 1} x})
  \mathfrak{P}_d (s_{\ell}^*, s_{\ell + 1}^*)\int_{0}^x e^{- \lambda \mathfrak{P}_d (s_{\ell}^*, s_{\ell +
  1}^*) (x - u)} du.
\end{align*}
Therefore,
\begin{align*}
  r_{\ell + 1} (x)   \leq & \lambda (s_{\ell + 1}^*)^d + \lambda (s_{\ell}^*)^d +
  \lambda (s_{\ell - 1}^*)^d (1 - e^{- \lambda \mathfrak{P}_d
  (s_{\ell}^*, s_{\ell + 1}^*) x}) \\
  &+ \lambda \sum_{j = 2}^{\ell - 2} (s_j^*)^d \prod_{i = j +
  1}^{\ell - 1} (1 - e^{- \lambda d (s_i^*)^{d - 1} x}) (1 - e^{- \lambda
  \mathfrak{P}_d (s_{\ell}^*, s_{\ell + 1}^*) x})  + \lambda \prod_{i = 1}^{\ell - 1} (1 - e^{- \lambda d (s_i^*)^{d
  - 1} x}) (1 - e^{- \lambda \mathfrak{P}_d (s_{\ell}^*, s_{\ell +
  1}^*) x}).
\end{align*}
Using the inequality
$\mathfrak{P}_d (s_{\ell}^*, s_{\ell + 1}^*) \leq d
(s_{\ell}^*)^{d - 1}$ from \eqref{Pdmax}, we then obtain
\begin{align*}
     r_{\ell + 1} (x)  \leq & \lambda (s_{\ell + 1}^*)^d + \lambda
     (s_{\ell}^*)^d + \lambda \sum_{j = 2}^{\ell -1}
     (s_j^*)^d \prod_{i = j + 1}^{\ell } (1 - e^{- \lambda d
     (s_i^*)^{d - 1} x}) + \lambda \prod_{i = 1}^{\ell } (1 -
     e^{- \lambda d (s_i^*)^{d - 1} x}),
  \end{align*}
which shows that \eqref{upbndr_ell} holds with $\ell$ replaced by $\ell+1$. 
Thus, by the principle of mathematical induction, \eqref{upbndr_ell} holds for all $\ell\geq 3$.

\end{proof}

We now state a result that is an immediate consequence of Lemma \ref{prop-rineq} and the elementary inequalities $s_\ell^* \leq s_{\ell-1}^*$ and $1 - e^{- x} \leq 1 \wedge x$ for $x\geq 0$.

\begin{corollary}
\label{cor-rdecay}
Suppose Assumptions \ref{ass-G} and \ref{ass-G2} hold. For $x\in[0,\infty)$,
\begin{equation} 
\label{r2-upbnd}
r_2 (x) \leq \lambda (s_1^*)^d + \lambda (1 \wedge \lambda d
   (s_1^*)^{d - 1} x), \end{equation}
   and for $\ell\geq 3$,
   \begin{align} 
   \label{rl-upbnd}
   r_{\ell} (x) \leq &2 \lambda (s_{\ell - 1}^*)^d + \lambda \sum_{j =
   2}^{\ell - 2} (s_j^*)^d \prod_{i = j + 1}^{\ell - 1} (1 \wedge \lambda
   d (s_i^*)^{d - 1} x) + \lambda \prod_{i = 1}^{\ell - 1} (1 \wedge
   \lambda d (s_i^*)^{d - 1} x), \end{align}
   where the second term on the right-hand side of \eqref{rl-upbnd} is treated as zero if $\ell=3$.
 \end{corollary}
 We now make an observation that will be used repeatedly in the next section.
 Recalling that, $x_0,C_0\in(0,\infty)$, and $\beta$ are constants from  Assumption \ref{ass-service}, now define
\begin{equation}
\label{def-C}
 C : = \frac{\int_0^{x_0} \bar{G} (x) dx}{\int_{x_0}^{\infty} \bar{G}
   (x) dx}, \qquad C_1:= C_0(C+1). 
   \end{equation}
Note that $C \in (0,\infty)$. Fix $f$ to be a  non-decreasing measurable function $f:[0,\infty)\mapsto[0,\infty)$. Then clearly
\begin{eqnarray*}
  \frac{\int_0^{x_0} f(x) \bar{G} (x) dx}{\int_{x_0}^{\infty}
  f (x) \bar{G} (x) dx}  \leq  \frac{f (x_0)
  \int_0^{x_0} \bar{G} (x) dx}{f(x_0) \int_{x_0}^{\infty}
  \bar{G} (x) dx}=C.
\end{eqnarray*}
Using the above inequality and \eqref{Gdecay}, and rescaling, we see that 
\begin{align}
 \label{aft-C1-def}
  \int_0^\infty f(x)\bar{G}(x) dx  =  \int_0^{x_0} f (x) \bar{G} (x) dx +
  \int_{x_0}^{\infty} f (x) \bar{G} (x) dx
  \leq &  (C + 1) \int_{x_0}^{\infty} f (x) \bar{G} (x) dx\nonumber\\
   \leq & (C + 1) C_0 \int_{x_0}^{\infty} f (x) x^{- \beta} dx\nonumber\\
=& \frac{C_1}{x_0^{\beta - 1}} \int_1^{\infty} f
   (x_0 y) y^{- \beta} dy . 
\end{align}
Define $C_\beta:=C_\beta(\lambda,\beta,d)$ as follows:
\begin{equation}
\label{def-C2}
C_\beta:=\left\{\begin{array}{ll}
\max \left( 1, \frac{ C_1 (\lfloor \beta \rfloor+1)}{\min
(\hat{\beta},   1 - \hat{\beta})} (\lambda d)^{\beta - 1}
\right)&\qquad\text{if }\beta\notin\mathbb{N},\\
 \max\left(1,C_1 (\beta+1) (\lambda d)^{\beta-1}\right)&\qquad\text{if }\beta\in\mathbb{N}.
 \end{array}\right.
\end{equation}
Using the convention $s_i^* =1$ for all  $i\leq 0$,
also define
\begin{equation}
\label{def-mi}
m_i:=\frac{1}{\lambda d(s_i^*)^{d-1}x_0},\qquad i\in\mathbb{Z}.
\end{equation}
 Fix $d\geq 2$, $\lambda<1$ and let $x_0>{1}/{\lambda d}$, $\beta>d/(d-1)$ be as in Assumption \ref{ass-service}, and 
let $C_\beta=C_\beta(\lambda,\beta,d)$ be as defined in \eqref{def-C2}.
Then, recalling $\ssl=\int_0^L r_\ell(x)\bar{G}(x) dx$ from \eqref{eqn-ssl}, integrating the inequality \eqref{rl-upbnd} with respect to $\bar{G}(x) dx$ and using the identity $\int_0^\infty\bar{G}(x) dx =1$, it follows that for $\ell\geq 3$
  \begin{align*} 
   s_{\ell}^* \leq &2 \lambda (s_{\ell - 1}^*)^d + \int_0^\infty\lambda \sum_{j=
   2}^{\ell - 2} (s_j^*)^d \prod_{i = j + 1}^{\ell - 1} (1 \wedge \lambda
   d (s_i^*)^{d - 1} x) \bar{G}(x)dx+ \int_0^\infty \lambda \prod_{i = 1}^{\ell - 1} (1 \wedge
   \lambda d (s_i^*)^{d - 1} x) \bar{G}(x) dx .
   \end{align*}
   Applying \eqref{aft-C1-def} with $f(x)=\prod_{i=j+1}^{\ell-1}(1\wedge \lambda d(s_i^*)^{d-1}x)$ for $j\in\{0,2,\cdots\ell-1\}$, and using \eqref{def-mi}, this implies 
   \begin{align} 
   \label{s-bndell}
   s_{\ell}^* \leq &2 \lambda (s_{\ell - 1}^*)^d +\frac{\lambda C_1}{x_0^{\beta - 1}}  \sum_{j =
   2}^{\ell - 2} (s_j^*)^d \psi_{\ell,j}+ \frac{\lambda C_1}{x_0^{\beta - 1}} \psi_{\ell,\ell-1} ,
   \end{align}
   where
   \begin{equation}
   \label{def-psi}
   \psi_{\ell,\ell-1}:=\int_1^\infty\prod_{i=1}^{\ell-1}(1\wedge m_i^{-1}y) dy\quad \text{and}\quad  \psi_{\ell,j}:=\int_1^\infty\prod_{i=j+1}^{\ell-1}(1\wedge m_i^{-1}y) dy, \quad j=2,\cdots,\ell-2.
   \end{equation}
  The proof of Proposition \ref{prop-upbnd} will proceed by bounding the right-hand side of \eqref{s-bndell} by \eqref{ineq-notinteger} or \eqref{ineq-integer}, depending on whether $\beta$ is or is not an integer, using an inductive argument and the estimates in the following lemma. In what follows, recall the convention that $s_i^* =1$ for all $i \leq 0$.
  \begin{lemma}
\label{lem-int}
Fix $\ell\geq 2$ and $j\leq \ell-2$. If $\beta\in\left({d}/{d-1},\infty\right)\setminus\N$, we have
\begin{align}
\label{last-int-nl-noint}
\frac{C_1}{x_0^{\beta - 1}}\psi_{\ell,j}\leq\left\{
  \begin{array}{ll}
    C_\beta \left(\prod_{i=\ell-\lfloor\beta\rfloor+1}^{\ell-1}(s_{i}^*)^{d-1}\right)
  (s_{\ell-\lfloor\beta\rfloor}^*)^{\hat{\beta}(d - 1)} &\quad \text{ if }j=\ell-\lfloor\beta\rfloor-1,\\
   C_\beta \left(\prod_{i=j+1}^{\ell-1}(s_{i}^*)^{d-1}\right) &\quad \text{ if }j\geq\ell-\lfloor\beta\rfloor,
\end{array}  \right.
\end{align}
whereas if $\beta\in\left({d}/{(d-1)},\infty\right)\cap\N$, then there exists $\delta\in (0,1)$ such that
\begin{align}
\label{last-int-nl-int}
\frac{C_1}{x_0^{\beta - 1}}\psi_{\ell,j}\leq\left\{
  \begin{array}{ll}
    C_\beta \left(\prod_{i=\ell-\beta+2}^{\ell-1}(s_{i}^*)^{d-1}\right)
  (s_{\ell-\beta+1}^*)^{(1-\delta)(d - 1)} &\quad \text{ if }j=\ell-\beta,\\
   C_\beta \left(\prod_{i=j+1}^{\ell-1}(s_{i}^*)^{d-1}\right) &\quad \text{ if }j\geq\ell-\beta+1.
\end{array}  \right.
\end{align}
Moreover, we also have
\begin{align}
\label{last-int-1}
\frac{C_1}{x_0^{\beta - 1}}\psi_{\ell,\ell-1}\leq\left\{
  \begin{array}{lr}
    C_\beta \left(\prod_{i=\ell-\lfloor\beta\rfloor+1}^{\ell-1}(s_{i}^*)^{d-1}\right)
  (s_{\ell-\lfloor\beta\rfloor}^*)^{\hat{\beta}(d - 1)} &\quad \text{ if }\beta\in\left(\frac{d}{d-1},\infty\right)\setminus\N,\\
   C_\beta \left(\prod_{i=\ell-\beta+2}^{\ell-1}(s_{i}^*)^{d-1}\right)(s_{\ell-\beta + 1}^*)^{(1-\delta)(d - 1)} &\quad \text{ if }\beta\in\left(\frac{d}{d-1},\infty\right)\cap\N.
\end{array}  \right.
\end{align}
\end{lemma}
 The proof of Lemma \ref{lem-int} is somewhat involved, and hence relegated to Section  \ref{subsec-est-decay}.
\subsection{Proof of Proposition \ref{prop-upbnd}}
\label{subsec-proof dec}
 First note that the convention $s_i^*=1$ for $i\leq 0$ and the assumption $\sum_{\ell\geq 1}\ssl<\infty$ imply $s_1^*=\lambda$ by Theorem \ref{th-reduction}. Suppose $\beta$ is not an integer.  Using \eqref{eqn-ssl}, \eqref{r2-upbnd} of Corollary \ref{cor-rdecay}  and the fact that $\int_0^\infty \bar{G}(x)dx =1$,  we have 
\begin{equation*}
s_2^*\leq \lambda (\ss_1)^d +\int_0^\infty\lambda(1\wedge \lambda d (s_1^*)^{d-1}x)\bar{G}(x)dx\leq \lambda (\ss_1)^d +\frac{C_1}{x_0^{\beta-1}}\int_1^\infty\lambda(1\wedge m_1^{-1}y)y^{-\beta}dy,
\end{equation*}
where the second inequality uses \eqref{aft-C1-def} with $f(x)=\lambda(1\wedge m_1^{-1}x)$, where $m_1$ is defined in \eqref{def-mi}. Applying \eqref{last-int-1} of Lemma \ref{lem-int}, with $\ell=2$, and using $C_\beta>1$ and $s_i^*\leq 1$ for all $i\in\Z$, we obtain
\begin{equation}
\label{bnd-s2}
s_2^*\leq \lambda (\ss_1)^d +\lambda C_\beta\left(\prod_{i=3-\lfloor\beta\rfloor}^{1}(s_{i}^*)^{d-1}\right) (s_{2-\lfloor\beta\rfloor}^*)^{\hat{\beta}(d - 1)} \leq 2\lambda C_\beta\left(\prod_{i=3-\lfloor\beta\rfloor}^{1}(s_{i}^*)^{d-1}\right) (s_{2-\lfloor\beta\rfloor}^*)^{\hat{\beta}(d - 1)} ,
\end{equation}
  which shows that \eqref{ineq-notinteger} holds for $\ell=2$.
  
  Similarly, using \eqref{eqn-ssl} and \eqref{rl-upbnd} of Corollary \ref{cor-rdecay}, with $\ell=3$, and \eqref{aft-C1-def}, with $f(x)=\lambda \prod_{i=1}^2 (1\wedge m_i^{-1}x)$, we have
  \begin{equation}
\label{s3}
s_3^*\leq \lambda (\ss_2)^d +\frac{C_1}{x_0^{\beta-1}}\int_1^\infty\lambda\prod_{i=1}^2(1\wedge m_i^{-1}y)y^{-\beta}dy.
  \end{equation} 
We now claim that for any $k\geq 2$ and $L_k<\infty$,
\begin{equation}
\label{sd-dec}
s_{k}^* \leq L_{k} \left(\prod_{i=k-\lfloor\beta\rfloor+1}^{k-1}(s_{i}^*)^{d-1}\right) (s_{k-\lfloor\beta\rfloor}^*)^{\hat{\beta}(d - 1)}\implies
 (s_{k}^*)^d \leq L_{k} \left(\prod_{i=k-\lfloor\beta\rfloor+2}^{k}(s_{i}^*)^{d-1}\right) (s_{k-\lfloor\beta\rfloor+1 }^*)^{\hat{\beta}(d - 1)}. \end{equation}  
  Indeed, this follows from the fact that
\begin{align*}
  (s_{k}^*)^d =(s_k^*)^{d-1}s_k^*
   \leq & (s_{k}^*)^{d - 1} L_{k}
  \left(\prod_{i=k-\lfloor\beta\rfloor+1}^{k-1}(s_{i}^*)^{d-1}\right) (s_{k-\lfloor\beta\rfloor}^*)^{\hat{\beta}(d - 1)}\\
  = & L_{k} \left(\prod_{i=k-\lfloor\beta\rfloor+2}^{k}(s_{i}^*)^{d-1}\right) (s_{k-\lfloor\beta\rfloor+1}^*)^{\hat{\beta}(d - 1)},
\end{align*}
and the fact that $s_{k-\lfloor\beta\rfloor}^*\leq 1$ and $s_{k-\lfloor\beta\rfloor+1}^*\leq (s_{k-\lfloor\beta\rfloor+1}^*)^{\hat{\beta}}$. 
Together, \eqref{bnd-s2}, \eqref{s3} and \eqref{sd-dec}, with $k=2$ and $L_2 = 2 C_\beta$, yield 
  \begin{equation}
  \label{bnd-s3}
  s_3^*\leq 5\lambda C_\beta\left(\prod_{i=4-\lfloor\beta\rfloor}^{2}(s_{i}^*)^{d-1}\right) (s_{3-\lfloor\beta\rfloor}^*)^{\hat{\beta}(d - 1)}.
  \end{equation} 
  
We now use induction to obtain estimates for all $\ell\geq 2$.  Define $K_2:=2$, $K_3:=5$, and 
\begin{equation}
\label{def-Kell}
K_\ell:=2K_{\ell-1}+\sum_{j =2}^{\ell - 2}K_j+1\qquad \forall \ell\geq 4.
\end{equation}
 Moreover, suppose for some $\ell> 3$,  
  \begin{equation}
  \label{inductionstep}
  s_j^*\leq \lambda K_j(C_\beta)^{j-2}\left(\prod_{i=j-\lfloor\beta\rfloor+1}^{j-1}(s_{i}^*)^{d-1}\right) (s_{j-\lfloor\beta\rfloor}^*)^{\hat{\beta}(d - 1)},\qquad j\in\{3,\cdots\ell-1\}.
  \end{equation}
  To upper bound the first term on the right-hand side of \eqref{s-bndell}, we combine \eqref{inductionstep}, with $j=\ell-1$, and \eqref{sd-dec}, with $k=\ell-1$ and $L_{\ell-1}=K_{\ell-1}C_\beta^{\ell-3}$, to conclude that
  
 \begin{equation}
 \label{induct-term1}
  2\lambda (\sslm)^d\leq 2\lambda K_{\ell-1}C_\beta^{\ell-3} \left(\prod_{i=\ell-\lfloor\beta\rfloor+1}^{\ell-1}(s_{i}^*)^{d-1}\right)(s_{\ell-\lfloor\beta\rfloor}^*)^{\hat{\beta}(d - 1)} .
  \end{equation} 
  Next, to estimate the last term on the right-hand side of \eqref{s-bndell}, combine \eqref{def-psi} with \eqref{last-int-1} of Lemma \ref{lem-int}, to obtain
\begin{equation}
\label{induct-term3}
\frac{C_1}{x_0^{\beta-1}}\psi_{\ell,\ell-1}\leq  C_\beta \left(\prod_{i=\ell-\lfloor\beta\rfloor+1}^{\ell-1}(s_{i}^*)^{d-1}\right)(s_{\ell-\lfloor\beta\rfloor}^*)^{\hat{\beta}(d - 1)}.
\end{equation}  
We now identify respective upper bounds for each of the summand in the middle term on the right-hand side of \eqref{s-bndell}. First, note that when $2\leq j\leq \ell-\lfloor\beta\rfloor - 1$, \eqref{last-int-nl-noint} and the inequalities $s_i^*\leq 1$ for all $i$, and $C_\beta>1$ imply
  \begin{equation}
  \label{small-j-bnd}
   \frac{C_1}{x_0^{\beta - 1}} (s_j^*)^d\psi_{\ell,j}\leq  \frac{C_1}{x_0^{\beta - 1}} (s_j^*)^d \psi_{\ell,\ell-\lfloor\beta\rfloor-1}\leq C_\beta^{\ell-2}\left(\prod_{i=\ell-\lfloor\beta\rfloor+1}^{\ell-1}(s_{i}^*)^{d-1}\right) (s_{\ell-\lfloor\beta\rfloor}^*)^{\hat{\beta}(d - 1)}.
  \end{equation}
  On the other hand, if $ \ell-\lfloor\beta\rfloor \leq j\leq \ell-2$, using the fact that $j+1<\ell$ and $s_i^*\leq 1$ for all $i$, as well as \eqref{inductionstep} and \eqref{sd-dec}, with $k=j$ and $L_k=K_jC_\beta^{j-2}$, one obtains
  \begin{equation}
  \label{large-j-bnd}
  \frac{C_1}{x_0^{\beta-1}}\psi_{\ell,j}\leq K_jC_\beta^{j-1} \left(\prod_{i=\ell-\lfloor\beta\rfloor+1}^{\ell-1}(s_{i}^*)^{d-1}\right) (s_{\ell-\lfloor\beta\rfloor}^*)^{\hat{\beta}(d - 1)}.
  \end{equation}
  We now observe that $K_\ell\leq 3^{\ell-1}$. Since $K_2=2$, $K_3=5$ and for $\ell>3$, by definition 
\begin{align*}
K_{\ell-1}=2K_{\ell-2}+\sum_{j=2}^{\ell-3}K_j+1&=K_{\ell-2}+\sum_{j=2}^{\ell-2}K_j+1\\
&=K_{\ell-2}+K_\ell-2K_{\ell-1}\\
&\geq K_\ell-2K_{\ell-1},
\end{align*}
which implies $K_\ell\leq 3^{\ell-1}$ 
for all $\ell\geq 2$. Combining this with \eqref{induct-term1}-\eqref{large-j-bnd}, the definition \eqref{def-Kell} and the inequality $C_\beta>1$, we conclude that
  
  \begin{align*}
\ssl&\leq  \left(2\lambda K_{\ell-1}C_\beta^{\ell-3}+\lambda C_\beta^{\ell-2}\sum_{j =2}^{\ell - 2}K_j +\lambda C_\beta\right)
\left(\prod_{i=\ell-\lfloor\beta\rfloor+1}^{\ell-1}(s_{i}^*)^{d-1}\right) (s_{\ell-\lfloor\beta\rfloor}^*)^{\hat{\beta}(d - 1)}\\
   &\leq \lambda K_\ell C_\beta^{\ell-2}\left(\prod_{i=\ell-\lfloor\beta\rfloor+1}^{\ell-1}(s_{i}^*)^{d-1}\right) (s_{\ell-\lfloor\beta\rfloor}^*)^{\hat{\beta}(d - 1)}.
\end{align*}
This shows that \eqref{inductionstep} also holds when $\ell-1$ is replaced by $\ell$. By the principle of induction, \eqref{inductionstep} holds for all $\ell\geq2$. Since $K_\ell\leq 3^{\ell-1}$ and $C_\beta>1$, this proves \eqref{ineq-notinteger} of Proposition \ref{prop-upbnd}.

If $\beta$ is an integer, the proof for 
\eqref{ineq-integer} uses the same argument as in the non-integral case except that $\hat{\beta}$  is replaced with $1-\delta$, and $i-\lfloor\beta\rfloor$ with $i-\beta+1$ for all $i$, and the upper bounds in \eqref{last-int-nl-int} of Lemma \ref{lem-int} are used in place of the bounds in \eqref{last-int-nl-noint}. Thus, we omit the details.
\subsection{Proof of Key Estimates}

\label{subsec-est-decay}

 To complete the proof of Proposition \ref{prop-upbnd}, in this section we present the proof of Lemma \ref{lem-int}. We start with a preliminary result in Lemma \ref{lem-integrals} that will be used in the proof. 
Recall the definition of $m_i$  from \eqref{def-mi} and note that $m_0<1$ and the sequence $\{m_i\}_{i\in\Z}$ is non-decreasing. Define 
\begin{equation}
\label{def-kappa}
\kappa:=\min\{i\in\N:m_i\geq 1\},
\end{equation}
and note that $\kappa<\infty$ since $s_i^*\rightarrow 0$ as $i\rightarrow\infty$; see  \eqref{eq-tailcond}. As always, $\beta\in(d/(d-1),\infty)$.
\begin{lemma}
\label{lem-integrals}
Fix $\ell\in\N$. Let $j'\leq\ell-2$. Then for any $z>0$,
\begin{equation}
\label{int1toinfty}
\int_z^\infty  \prod_{i=j'+1}^{\ell-1} (1 \wedge
  m_i^{- 1} y)y^{-\beta}dy\leq\frac{ z^{-(\beta-1)}}{(\beta-1)}.
\end{equation}
Moreover, if $j'\leq n-2$ and $\ell-\lfloor \beta\rfloor +2\leq n\leq \ell-1$, then
\begin{align}
\label{intmn-1to mn}
\int_{m_{n-1}}^{m_n}  \prod_{i=j'+1}^{\ell-1} (1 \wedge
  m_i^{- 1} y)y^{-\beta}dy
  \leq
  \frac{ (\lambda d x_0)^{\beta - 1}}{(\beta - \ell + n - 1)} 
  \left(\prod_{i=n}^{\ell-1}(s_{i}^*)^{d-1}\right)     (s_{n -
  1}^*)^{(\beta - \ell + n- 1)(d - 1)}.
  \end{align}
  If $\beta\in\left({d}/{(d-1)},\infty\right)\setminus \N$ and $\ell-\lfloor \beta\rfloor <j'+1\leq\kappa\leq \ell-1$, then
  \begin{equation}
\label{int1tomk}
\int_{1}^{m_{\kappa}} \ \prod_{i=j'+1}^{\ell-1} (1 \wedge
  m_i^{- 1} y)y^{-\beta}dy\leq\frac{(\lambda d x_0)^{\beta -\hat{\beta}- 1}}{\hat{\beta}} 
  \left(\prod_{i=j'+1}^{\ell-1}(s_{i}^*)^{d-1}\right).
  \end{equation}
\end{lemma}
\begin{proof}
Fix $j'\leq\ell-2$. The relation \eqref{int1toinfty} follows immediately since $1\wedge x\leq 1$, and $\beta>1$ imply
\begin{align*}
\int_z^\infty \prod_{i =j'+1}^{\ell - 1} (1 \wedge
   m_i^{- 1} y)y^{-\beta}dy
   \leq&\int_z^\infty y^{-\beta}dy
   =\frac{z^{-(\beta-1)}}{(\beta-1)}.
   \end{align*}
Next, fix $j',\ell,n \in\N$ with $j'\leq n-2$ and $\ell-\lfloor \beta\rfloor +2\leq n\leq \ell-1$. 
 Then, for  $y\in [m_{n-1},m_n)$, the fact that $m_i$ is increasing in $i$ implies $m_i^{-1}y\geq 1$ for all $i\leq n-1$ and $m_i^{-1}y\leq 1$ for all $i\geq n$. Hence,   the definition of $m_{n-1}$ in \eqref{def-mi} and the relation $\ell-\beta-n+1<0$ yield \eqref{intmn-1to mn} as follows:
 \begin{align*}
  \int_{m_{n- 1}}^{m_n}  \prod_{i= j'+1}^{\ell - 1} (1 \wedge
   m_i^{- 1} y)y^{- \beta} dy
    &=\int_{m_{n - 1}}^{m_n}  (\lambda d x_0)^{\ell
  - n} \left(\prod_{i=n}^{\ell-1}(s_{i}^*)^{d-1}\right) y^{- \beta + \ell - n}
  dy\nonumber\\
  & \leq  \frac{ (\lambda d x_0)^{\ell - n}}{(\beta -
  \ell + n - 1) } \left(\prod_{i=n}^{\ell-1}(s_{i}^*)^{d-1}\right) (\lambda d (s_{n- 1}^*)^{d - 1} x_0)^{\beta - \ell + n - 1}\nonumber\\
  &=  \frac{(\lambda d x_0)^{\beta - 1}}{(\beta - \ell + n - 1)} 
  \left(\prod_{i=n}^{\ell-1}(s_{i}^*)^{d-1}\right)   (s_{n -
  1}^*)^{(\beta - \ell + n- 1)(d - 1)}.
\end{align*}

Next, let $\ell-\lfloor \beta\rfloor<j'+1\leq\kappa\leq \ell-1$.  Then,  for  $y\in [1,m_\kappa)$, the fact that $m_i$ is increasing in $i$ implies $m_i^{-1}y\geq 1$ for all $i\leq \kappa-1$ and $m_i^{-1}y\leq 1$ for all $i\geq \kappa$.  Hence, combining this with the definition of $m_i$ in \eqref{def-mi}, and the relation $\beta-\lfloor\beta\rfloor+1=\hat{\beta}+1>0$, we obtain \eqref{int1tomk} as follows: 
\begin{align*}
 \int_{1}^{m_\kappa}  \prod_{i= j'+1}^{\ell - 1} (1 \wedge
   m_i^{- 1} y)y^{- \beta} dy
  & = (\lambda d x_0)^{\ell - \kappa} \left(\prod_{i=\kappa}^{\ell-1}(s_{i}^*)^{d-1}\right)  \int_1^{m_\kappa} y^{- \beta + \ell - \kappa}
  dy\nonumber\\
   & \leq  (\lambda d x_0)^{\ell - \kappa} \left(\prod_{i=\kappa}^{\ell-1}(s_{i}^*)^{d-1}\right) \int_1^{m_\kappa} \left(\prod_{i=j'+1}^{\kappa-1}m_i^{-1}y\right)y^{- \beta + \ell - \kappa}
  dy\nonumber\\
  & = (\lambda d x_0)^{\ell -j'-1} \left(\prod_{i=j'+1}^{\ell-1}(s_{i}^*)^{d-1}\right)  \int_1^{m_\kappa} y^{-\beta+\ell-j' -1}
  dy\nonumber\\
 &\leq  (\lambda
  dx_0)^{\lfloor \beta\rfloor-1} \left(\prod_{i=j'+1}^{\ell-1}(s_{i}^*)^{d-1}\right)
  \int_1^{m_\kappa} y^{  -\beta+ \lfloor\beta\rfloor-1}
  dy\nonumber\\
  & \leq  \frac{(\lambda dx_0)^{\beta-\hat{\beta} - 1}}{\hat{\beta} }
   \left(\prod_{i=j'+1}^{\ell-1}(s_{i}^*)^{d-1}\right) ,
\end{align*}
where the fourth step uses the inequalities $\ell-\lfloor\beta\rfloor< j'+1$ and   $\lambda dx_0\geq 1$.
\end{proof}

Recall the definition of $C_1$ from \eqref{def-C}, $C_\beta$ from \eqref{def-C2} and $\psi$ from \eqref{def-psi}.

\begin{proof}[Proof of Lemma \ref{lem-int}]
Fix $\ell\in\N$, $\ell\geq 2$ and $\beta\in (d/(d-1),\infty)$. 
 We first show that \eqref{last-int-nl-noint} and \eqref{last-int-nl-int} imply \eqref{last-int-1}. First, suppose $\beta\notin\N$. Recalling the definition of $\psi_{\ell,j}$ from \eqref{def-psi}, we now claim (and justify below) that
 \begin{align}
\label{psi-equiv}
\frac{C_1}{x_0^{\beta - 1}}\int_1^\infty \prod_{i = 1}^{\ell - 1} (1 \wedge
  m_i^{- 1} y)y^{-\beta}dy\leq \frac{C_1}{x_0^{\beta - 1}}\int_1^\infty& \prod_{i = \ell-\lfloor\beta\rfloor}^{\ell - 1} (1 \wedge
  m_i^{- 1} y)y^{-\beta}dy= \frac{C_1}{x_0^{\beta - 1}}\psi_{\ell,\ell-\lfloor\beta\rfloor-1}.
\end{align}
  Indeed, note that if $i\leq 0$, then  $1\wedge m_i^{-1}=1\wedge \lambda dx_0(s_i^*)^{d-1}=1$ since $\lambda d x_0>1$, and $s_i^* =1$ for $i\leq 0$. This immediately implies \eqref{psi-equiv} when  $\ell-\lfloor\beta\rfloor\leq 1$.
On the other hand, if $\ell-\lfloor\beta\rfloor>1$, then the trivial inequality $1\wedge x\leq 1$ and \eqref{def-psi} imply \eqref{psi-equiv}.
 Noting that $\lambda d x_0\geq 1$ implies $C_1/x_0^{\beta-1}$ is bounded above by the constant $C_\beta$ defined in \eqref{def-C2}, the last two displays imply \eqref{last-int-1}.
 If $\beta\in\N$, \eqref{last-int-nl-int}  can be shown to imply \eqref{last-int-1}  using the same argument, but with $\ell-\lfloor\beta\rfloor$ replaced by $\ell-\beta+1$.

 We now turn to the proofs of \eqref{last-int-nl-noint} and \eqref{last-int-nl-int}.
Recall the definition of  $\kappa$ in \eqref{def-kappa}.

\noindent\underline{\textbf{Case 1}}: Suppose $\kappa\geq\ell$ and $\ell-\lfloor\beta\rfloor-1\leq j\leq \ell-2$.   Then, using  the definition of $\psi_{\ell,j}$ in \eqref{def-psi} and (since $j\leq \ell-2$) applying \eqref{int1toinfty} with $z=1$ and $j'=j$, as well as \eqref{def-C2}, this implies
\begin{align}
\frac{C_1}{x_0^{\beta-1}}\psi_{\ell,j}=
\frac{C_1}{x_0^{\beta-1}}\int_1^\infty \prod_{i = j+1}^{\ell - 1} (1 \wedge
   m_i^{- 1} y)y^{-\beta}dy
   \leq &\frac{C_1}{(\beta-1)x_0^{\beta-1}}\leq \frac{C_\beta}{(\lambda d x_0)^{\beta-1}}\nonumber.
   \end{align}
  Combining this with the observation that $m_i^{-1}=\lambda d x_0 (s_i^*)^{d-1}\geq 1$ for all $ i\leq \ell-1$ due to \eqref{def-kappa} and the case assumption, we have, for any $r_{j,\beta}\geq 1$,
   \begin{align}
   \label{case1i}
   \frac{C_1}{x_0^{\beta-1}}\psi_{\ell,j}\leq&\frac{ C_\beta}{(\lambda dx_0)^{\beta-1}}
  \left(\prod_{i=j+1}^{\ell-1}m_i^{-1}\right)r_{j,\beta}= C_\beta{(\lambda d x_0)^{-(j-\ell+\beta)}}
  \left(\prod_{i=j+1}^{\ell-1}(s_{i}^*)^{d-1}\right)r_{j,\beta}.
  \end{align}  
  Since $\lambda d x_0>1$, choosing $r_{j,\beta}=1$, we obtain
\begin{equation*}
\frac{C_1}{x_0^{\beta-1}}\psi_{\ell,j}\leq C_\beta
  \left(\prod_{i=j+1}^{\ell-1}(s_{i}^*)^{d-1}\right)\qquad\text{if } j\geq \ell-\beta.
\end{equation*}  
This immediately implies \eqref{last-int-nl-noint} when $\beta\notin\N$ and $j\geq \ell-\lfloor\beta\rfloor$, and \eqref{last-int-nl-int} when $\beta\in\N$ and $j\geq\ell-\beta+1$ and also, noting that $s^*_{\ell-\beta+1}\leq 1$ implies  $s^*_{\ell-\beta+1}\leq (s^*_{\ell-\beta+1})^{(1-\delta)}$ for any $\delta>0$, when $j=\ell-\beta$. When $\beta\notin\N$, since $j\geq \ell-\beta$ implies $j\geq \ell-\lfloor\beta\rfloor$, \eqref{case1i} also proves \eqref{last-int-nl-noint} in that case. The remaining case when $\beta\notin\N$ and $j=\ell-\lfloor\beta\rfloor-1$ can be deduced similarly from \eqref{case1i} by setting $r_{j,\beta}=m_{\ell-\lfloor\beta\rfloor}^{-(1-\hat{\beta})}=(\lambda d x_0)^{-(1-\hat{\beta})}(s^*_{\ell-\lfloor\beta\rfloor})^{-(1-\hat{\beta})(d-1)}$ therein. 
  
  \noindent\underline{\textbf{Case 2}}: Suppose $\kappa<\ell$ and $\ell-\lfloor\beta\rfloor-1\leq j<\kappa-1$. We 
now look at the partition $\pi_2=\{[1,m_\kappa),[m_\kappa,m_{\kappa+1}),$ $\cdots,$ $[m_{\ell-1},\infty)\}$ of the interval $[1,\infty)$. Note that 
\begin{equation}
\label{case2-intervals}
 \mbox{number of intervals in } \pi_2 =\ell-\kappa+1\leq \lfloor\beta\rfloor -1+1 <\lfloor\beta\rfloor+1.
\end{equation}
Since $j\leq \ell-2$, using  \eqref{int1toinfty} with $z=m_{\ell-1}=(\lambda d(\sslm)^{d-1}x_0)^{-1}$ and $j'=j$, we have
\begin{align}
 \frac{C_1}{x_0^{\beta - 1}} \int_{m_{\ell - 1}}^{\infty}  \prod_{i = j+1}^{\ell - 1} (1 \wedge
   m_i^{- 1} y)y^{- \beta} dy
\leq& \frac{ C_1}{(\beta - 1)} (\lambda d)^{\beta - 1}
  (s_{\ell-1}^*)^{(d - 1) (\beta - 1)}.\label{eqn-lastpart}
\end{align}
 Next, since $\ell-\lfloor\beta\rfloor\leq j +1 < \kappa$ implies  $\ell-\lfloor\beta\rfloor +1 \leq \kappa$ , we have for $n \in\{ \kappa+1, \kappa+2, \cdots, \ell
- 1\}$, $n\geq \kappa +1\geq \ell-\lfloor\beta\rfloor +2$ and $n\geq \kappa +1> j+2$.  Since  $\ell-\lfloor\beta\rfloor +2\leq n\leq \ell-1$ and $j\leq n-2$, we use \eqref{intmn-1to mn} with $j'=j$ to obtain
\begin{align}
 \frac{C_1}{x_0^{\beta - 1}} \int_{m_{n- 1}}^{m_n}  \prod_{i = j+1}^{\ell - 1} (1 \wedge
   m_i^{- 1} y)y^{- \beta} dy   
    \leq  \frac{ C_1}{(\beta - \ell + n - 1)} (\lambda d)^{\beta - 1} \sigma_n^{d-1},\label{eqn-penultpart}
\end{align}
where
\begin{equation}
\label{def-sigma}
\sigma_n :=   \left(\prod_{i=n}^{\ell-1}s_{i}^*\right)   (s_{n -
  1}^*)^{(\beta - \ell + n- 1)}.
\end{equation}
Now, for any $n	\in\{\kappa +1, \cdots, \ell - 1\}$ 
the case assumption, which in particular implies $j <  n-2$,  the inequality $s_{n-1}^* \leq 1$, and the fact that 
$i \mapsto s_i^*$ is increasing imply that  $(s_{\ell-1}^*)^{\beta} \leq \sigma_n$ and 
$\sigma_n \leq \sigma_n (s_{n-1}^*)^{\ell - j - \beta} = \left( \prod_{i=n}^{\ell - 1} s_i^* \right) (s_{n-1}^*)^{n-1- j}    
\leq \prod_{j+1}^{n-1} s_i^*$, 
and hence, that 
\begin{align} 
\label{max-case2a-large-j}
&\max_{n\in\{\kappa+1,\cdots,\ell-1\}}\Big((s_{\ell-1}^*)^{\beta-1},\sigma_n,
  \prod_{i=j+1}^{\ell-1}s_{i}^* \Big)=\prod_{i=j+1}^{\ell-1}s_{i}^*,
  \end{align}
 \noindent\underline{\textbf{Case 2A}}: Suppose $\beta\notin\N$. Restricting $j$ so that $\ell-\lfloor\beta\rfloor-1< j<\kappa-1$, and using \eqref{int1tomk} with $j'=j$, we have for any $r_{j,\beta}\geq 1$, 
 \begin{align}
\frac{C_1}{x_0^{\beta - 1}} \int_{1}^{m_\kappa}  \prod_{i = j+1}^{\ell - 1} (1 \wedge
   m_i^{- 1} y)y^{- \beta} dy
  & \leq  \frac{ C_1}{\hat{\beta}x_0^{\beta-1} }
  (\lambda dx_0)^{\beta-\hat{\beta} - 1} \left(\prod_{i=j+1}^{\ell-1}(s_{i}^*)^{d-1}\right)r_{j,\beta} .\label{case2i3}
\end{align}
Since $\hat{\beta}>0$ and $\lambda d x_0\geq 1$, choosing $r_{j,\beta}=1$, we obtain
\begin{equation}
\label{case2a-large-j}
\frac{C_1}{x_0^{\beta - 1}} \int_{1}^{m_\kappa}  \prod_{i = j+1}^{\ell - 1} (1 \wedge
   m_i^{- 1} y)y^{- \beta} dy
   \leq  \frac{ C_1}{\hat{\beta} }
  (\lambda d)^{\beta - 1} \left(\prod_{i=j+1}^{\ell-1}(s_{i}^*)^{d-1}\right).
\end{equation}
  Combining \eqref{case2-intervals} -- \eqref{def-sigma}, \eqref{case2a-large-j} and \eqref{max-case2a-large-j} with the fact that
   \begin{equation}
   \label{min-beta-noint}
   \min_{n\in\{\kappa+1,\cdots,\ell-1\}}({\beta-1},{\beta-\ell+n-1},{\hat{\beta}})={\min(\hat{\beta},1-\hat{\beta})} ,
\end{equation}   
    and  the definitions of $C_\beta$ and $\psi_{\ell,j}$ in \eqref{def-C2} and \eqref{def-psi}, respectively, we obtain 
 \eqref{last-int-nl-noint} whenever $j\geq \ell-\lfloor\beta\rfloor$.
 
 Setting $j=\ell-\lfloor\beta\rfloor$ and $r_{j,\beta}=m_{\ell-\lfloor\beta\rfloor}^{-\hat{\beta}}=(\lambda d x_0)^{\hat{\beta}}(s^*_{\ell-\lfloor\beta\rfloor})^{\hat{\beta}(d-1)}$ in \eqref{case2i3}, we obtain
\begin{align}
\frac{C_1}{x_0^{\beta - 1}} \int_{1}^{m_\kappa}  \prod_{i = \ell-\lfloor\beta\rfloor}^{\ell - 1} (1 \wedge
   m_i^{- 1} y)y^{- \beta} dy&\leq \frac{C_1}{x_0^{\beta - 1}} \int_{1}^{m_\kappa}  \prod_{i = \ell-\lfloor\beta\rfloor+1}^{\ell - 1} (1 \wedge
   m_i^{- 1} y)y^{- \beta} dy\nonumber\\
  \label{bnd-al-less0}
    & \leq  \frac{ C_1}{\hat{\beta} }
  (\lambda d)^{\beta - 1} \left(\prod_{i=\ell-\lfloor\beta\rfloor+1}^{\ell-1}(s_{i}^*)^{d-1}\right)(s_{\ell-\lfloor\beta\rfloor}^*)^{\hat{\beta}(d - 1)} .
\end{align}
  Now, noting that $\ell-\lfloor\beta\rfloor -1< n-2$ and $i\mapsto s_i^*$ is non-increasing, we conclude that  
\begin{align} 
\label{max-case2a-small-j}
&\max_{n\in\{\kappa+1,\cdots,\ell-1\}}\left((s_{\ell-1}^*)^{\beta-1},\sigma_n,
  \left(\prod_{i=\ell-\lfloor\beta\rfloor+1}^{\ell-1}s_{i}^*\right)(s_{\ell-\lfloor\beta\rfloor}^*)^{\hat{\beta}}  \right)=\left(\prod_{i=\ell-\lfloor\beta\rfloor+1}^{\ell-1}s_{i}^*\right)
  (s_{\ell-\lfloor\beta\rfloor}^*)^{\hat{\beta}} ,
  \end{align}
  Combining  \eqref{case2-intervals} -- \eqref{def-sigma}, \eqref{min-beta-noint} --  \eqref{max-case2a-small-j} and   the definitions of $C_\beta$ and $\psi_{\ell,j}$ in \eqref{def-C2} and \eqref{def-psi}, respectively, we obtain 
 \eqref{last-int-nl-noint} when $j= \ell-\lfloor\beta\rfloor-1$.\\ 
\noindent\underline{\textbf{Case 2B}}: Suppose 
  $\beta\in\N$. Note that $y\in [1,m_\kappa)$ and the fact that $i\mapsto m_i$ is increasing implies $m_i^{-1}y\geq 1$ for all $i\leq \kappa-1$ and $m_i^{-1}y\leq 1$ for all $i\geq \kappa$. Then, substituting the expression for $m_i^{-1}$ in \eqref{def-mi}, and $m_{j+1}^{-1}\geq 1$, we obtain
\begin{align}
\frac{C_1}{x_0^{\beta - 1}} \int_{1}^{m_\kappa}  \prod_{i = j+1}^{\ell - 1} (1 \wedge
   m_i^{- 1} y)y^{- \beta} dy&\leq \frac{C_1}{x_0^{\beta - 1}} \int_{1}^{m_\kappa}  \prod_{i = \kappa}^{\ell - 1} (1 \wedge
   m_i^{- 1} y)y^{- \beta} dy\nonumber\\
   &\leq \frac{ C_1}{x_0^{\beta - 1}}(\lambda dx_0)^{\ell-\kappa} \int_{1}^{m_\kappa} \left( \prod_{i = \kappa}^{\ell - 1} (s_i^*)^{d-1}\right)y^{- \beta+\ell-\kappa} dy\nonumber\\
   &\leq \frac{ C_1}{x_0^{\beta - 1}}(\lambda dx_0)^{\ell-\kappa}\left( \prod_{i = \kappa}^{\ell - 1} (s_i^*)^{d-1}\right) \int_{1}^{m_\kappa} \left(\prod_{i=j+2}^{\kappa-1}m_i^{-1}y\right)m_{j+1}^{-1} y^{- \beta+\ell-\kappa} dy\nonumber\\
   &= \frac{ C_1}{x_0^{\beta - 1}}(\lambda dx_0)^{\ell-j-1} \left(\prod_{i=j+1}^{\ell-1}(s_i^*)^{d-1}\right)\int_{1}^{m_\kappa} y^{-\beta+\ell-j-2} dy\nonumber\\
    &\leq\frac{ C_1}{x_0^{\beta - 1}}(\lambda dx_0)^{\beta-1} \left(\prod_{i=j+1}^{\ell-1}(s_i^*)^{d-1}\right)\int_{1}^{m_\kappa} y^{-2} dy\nonumber\\
   \label{case2-int1tomk-intpart2}
     &\leq  C_1(\lambda dx_0)^{\beta-1}  \left( \prod_{i = j+1}^{\ell - 1} (s_i^*)^{d-1}\right).
\end{align}
Combining \eqref{case2-intervals} --  \eqref{max-case2a-large-j} and \eqref{case2-int1tomk-intpart2} with the fact that 
\begin{equation}
\label{min-beta-int}
\min_{n\in\{\kappa+1,\cdots,\ell-1\}}({\beta-1},{\beta-\ell+n-1},1)=1,
\end{equation}
  the observation that $s_{\ell-\beta+1}^*\leq (s_{\ell-\beta+1}^*)^{1-\delta}$ for any $\delta \in (0,1)$ and  the definitions of $C_\beta$ and $\psi_{\ell,j}$ in \eqref{def-C2} and \eqref{def-psi}, respectively, we obtain \eqref{last-int-nl-int}.

\noindent\underline{\textbf{Case 3}}: Suppose $\kappa< \ell$, $j\geq \ell-\lfloor\beta\rfloor-1$ and $\kappa-1\leq j\leq \ell-2$. 
Using \eqref{def-kappa}, this implies  $m_i\geq 1$ for all $i\geq j+1$. We now look at the partition 
$\pi_3=\{[1,m_{j+1}),[m_{j+1},m_{j+2}),\cdots,$ $[m_{\ell-1},\infty)$ 
of the interval $[1,\infty)$.  Note that 
\begin{equation}
\label{case3-intervals}
 \mbox{number of intervals in } \pi_3 =\ell-j\leq \lfloor\beta\rfloor +1.
\end{equation}
Since $j+1\leq \ell-1$, using  \eqref{int1toinfty} with $z=m_{\ell-1}=(\lambda d (\sslm)^{d-1} x_0)^{-1}$ and $j'=j$, we have
\begin{align}
\label{case3intell-1toinfty}
 \frac{C_1}{x_0^{\beta - 1}} \int_{m_{\ell-1}}^{\infty}  \prod_{i = \ell-\lfloor\beta\rfloor}^{\ell - 1} (1 \wedge
   m_i^{- 1} y)y^{- \beta} dy
\leq& \frac{ C_1}{(\beta - 1)} (\lambda d)^{\beta - 1}
  (s_{\ell-1}^*)^{(d - 1) (\beta - 1)}.
\end{align}

 For the  interval $[1,m_{j+1})$ note that  $y\in [1,m_{j+1})$ and the fact that $m_{i}$ is increasing in $i$ implies  $m_i^{-1}y\leq 1$ for all $i\geq j+1$. Hence,  we have 
\begin{align}
\frac{ C_1}{x_0^{\beta - 1}} \int_1^{m_{j+1}} \prod_{i = j+1}^{\ell
   - 1} (1 \wedge m_i^{-1} y) y^{- \beta} dy
  & =  \frac{ C_1}{x_0^{\beta - 1}} (\lambda d x_0)^{\ell - j-1}
   \left(\prod_{i=j+1}^{\ell-1}(s_{i}^*)^{d-1}\right) \int_1^{m_{j+1}}
  y^{- \beta + \ell - j-1} dy.\label{int1nellj+1}
  \end{align}
   Additionally, if $j\geq \ell-\lfloor\beta\rfloor$, note that for $n \in\{ j+2, j+3, \cdots, \ell
- 1\}$, $n\geq j+2\geq\ell-\lfloor\beta\rfloor +2$.  Since  $\ell-\lfloor\beta\rfloor +2 \leq n\leq \ell-1$, we use \eqref{intmn-1to mn} with $j'=j$ and the definition of $\sigma_n$ in \eqref{def-sigma} to obtain
\begin{align}
\label{case3-partition-gen-integral}
 \frac{C_1}{x_0^{\beta - 1}} \int_{m_{n- 1}}^{m_n}  \prod_{i = j+1}^{\ell - 1} (1 \wedge
   m_i^{- 1} y)y^{- \beta} dy   
    \leq  \frac{ C_1}{(\beta - \ell + n - 1)} (\lambda d)^{\beta - 1}
  \sigma_n^{d-1}.
\end{align}
   \noindent\underline{\textbf{Case 3A}}: Suppose $\beta\notin\N$. First, let  $j=\ell-\lfloor\beta\rfloor-1$. Since $\beta-\hat{\beta}=\lfloor\beta\rfloor$ by \eqref{def-hatbeta}, \eqref{int1nellj+1} simplifies to
  \begin{align}
  \label{nointcase3}
\frac{ C_1}{x_0^{\beta - 1}} \int_1^{m_{\ell-\lfloor\beta\rfloor}} \prod_{i = \ell-\lfloor\beta\rfloor}^{\ell
   - 1} (1 \wedge m_i^{-1} y) y^{- \beta} dy   & =  \frac{ C_1}{x_0^{\beta - 1}} (\lambda d x_0)^{\beta-\hat{\beta}}
   \left(\prod_{i=\ell-\lfloor\beta\rfloor}^{\ell-1}(s_{i}^*)^{d-1}\right) \int_1^{m_{\ell-\lfloor\beta\rfloor}}
  y^{- \hat{\beta}} dy\nonumber\\
  & \leq  \frac{ C_1}{(1 - \hat{\beta}) x_0^{\beta - 1}} (\lambda d
  x_0)^{\beta - \hat{\beta}}  \left(\prod_{i=\ell-\lfloor\beta\rfloor}^{\ell-1}(s_{i}^*)^{d-1}\right) (\lambda d (s_{\ell-\lfloor\beta\rfloor}^*)^{d - 1}
  x_0)^{\hat{\beta} - 1}\nonumber\\
  & =  \frac{ C_1}{(1 - \hat{\beta})} (\lambda d)^{\beta - 1}
  \left(\prod_{i=\ell-\lfloor\beta\rfloor+1}^{\ell-1}(s_{i}^*)^{d-1}\right)
  (s_{\ell-\lfloor\beta\rfloor}^*)^{\hat{\beta}(d - 1)} .
\end{align}
 Note that for $n \in\{ \ell-\lfloor\beta\rfloor+2, \ell-\lfloor\beta\rfloor+3, \cdots, \ell
- 1\}$, $n\geq \ell-\lfloor\beta\rfloor +2$.  Since  $\ell-\lfloor\beta\rfloor +2\leq n\leq \ell-1$ and, we use \eqref{intmn-1to mn}, with $j'=\ell-\lfloor\beta\rfloor-1$ and the definition of $\sigma_n$ in \eqref{def-sigma} to obtain
\begin{align}
\label{case3mn-1tomn}
 \frac{C_1}{x_0^{\beta - 1}} \int_{m_{n- 1}}^{m_n}  \prod_{i = \ell-\lfloor\beta\rfloor}^{\ell - 1} (1 \wedge
   m_i^{- 1} y)y^{- \beta} dy   
    \leq  \frac{ C_1}{(\beta - \ell + n - 1)} (\lambda d)^{\beta - 1}
  \sigma_n^{d-1}.
\end{align}
Also, note that  $y\in [m_{\ell-\lfloor\beta\rfloor},m_{\ell-\lfloor\beta\rfloor+1})$ and the fact that $i\mapsto m_i$ is increasing implies $m_i^{-1}y\geq 1$ for all $i\leq \ell-\lfloor\beta\rfloor$ and $m_i^{-1}y\leq 1$ for all $i\geq \ell-\lfloor\beta\rfloor+1$. Hence, using \eqref{def-mi}, we obtain
 \begin{align}
 \label{noint-mnltomnl+1}
 \frac{C_1}{x_0^{\beta - 1}}\int_{m_{\ell-\lfloor\beta\rfloor}}^{m_{\ell-\lfloor\beta\rfloor+1}}  \prod_{i= \ell-\lfloor\beta\rfloor}^{\ell - 1} (1 \wedge
   m_i^{- 1} y)y^{- \beta} dy
    &= \frac{C_1}{x_0^{\beta - 1}}\int_{m_{\ell-\lfloor\beta\rfloor}}^{m_{\ell-\lfloor\beta\rfloor+1}}  (\lambda d x_0)^{\lfloor\beta\rfloor-1} \left(\prod_{i=\ell-\lfloor\beta\rfloor+1}^{\ell-1}(s_{i}^*)^{d-1}\right) y^{- \hat{\beta} -1}
  dy\nonumber\\
   & \leq  \frac{ C_1(\lambda d x_0)^{\beta-\hat{\beta}-1}}{\hat{\beta}x_0^{\beta-1} } \left(\prod_{i=\ell-\lfloor\beta\rfloor+1}^{\ell-1}(s_{i}^*)^{d-1}\right) (\lambda d (s_{\ell-\lfloor\beta\rfloor}^*)^{d - 1} x_0)^{\hat{\beta}}\nonumber\\
  &=  \frac{C_1(\lambda d )^{\beta - 1}}{\hat{\beta}} 
  \left(\prod_{i=\ell-\lfloor\beta\rfloor+1}^{\ell-1}(s_{i}^*)^{d-1}\right)   (s_{\ell-\lfloor\beta\rfloor}^*)^{\hat{\beta}(d - 1)}.
\end{align}
  Combining \eqref{case3-intervals}, \eqref{case3intell-1toinfty}, \eqref{nointcase3} -- \eqref{noint-mnltomnl+1}, \eqref{max-case2a-small-j}, \eqref{min-beta-noint} (since $\kappa\leq \ell-\lfloor\beta\rfloor$)
  and the definitions of $C_\beta$ and $\psi_{\ell,j}$ in \eqref{def-C2} and \eqref{def-psi}, respectively, we obtain \eqref{last-int-nl-noint} when $j=\ell-\lfloor\beta\rfloor-1$.
  
  Now, let $j\geq\ell-\lfloor\beta\rfloor$. Using  \eqref{def-mi}, \eqref{int1nellj+1} simplifies to
  \begin{align}
  \label{case3-int1tomj+1-nointpart2}
\frac{ C_1}{x_0^{\beta - 1}} \int_1^{m_{j+1}} \prod_{i = j+1}^{\ell
   - 1} (1 \wedge m_i^{-1} y) y^{- \beta} dy &\leq\frac{ C_1}{x_0^{\beta-1}}(\lambda d x_0)^{\lfloor\beta\rfloor-1} \left(\prod_{i=j+1}^{\ell-1}(s_{i}^*)^{d-1}\right)\int_1^{m_{j+1}}y^{-\hat{\beta}-1}dy \nonumber\\
   &\leq  \frac{ C_1}{ \hat{\beta} } (\lambda d
  )^{\beta - 1}  \left(\prod_{i=j+1}^{\ell-1}(s_{i}^*)^{d-1}\right).
\end{align} 
  Combining  \eqref{case3-intervals}, \eqref{case3intell-1toinfty}, \eqref{case3-partition-gen-integral}, \eqref{case3-int1tomj+1-nointpart2}, \eqref{min-beta-noint}, \eqref{max-case2a-large-j} (since $\kappa\leq j+1$)
  and  the definitions of $C_\beta$ and $\psi_{\ell,j}$ in \eqref{def-C2} and \eqref{def-psi}, respectively, we obtain \eqref{last-int-nl-noint} when $j\geq\ell-\lfloor\beta\rfloor$.
  
\noindent\underline{\textbf{Case 3B}}: Suppose $\beta\in\N$. First, let $j=\ell-\beta$. Using \eqref{def-mi} and the fact that there exists  $\delta\in (0,1)$ such that $\log x\leq x^\delta$ for all $x\geq 1$, \eqref{int1nellj+1} simplifies to
\begin{align}
\label{case3-int1tomj+1-intpart1}
 \frac{ C_1}{x_0^{\beta - 1}}\int_1^{m_{\ell-\beta+1}} \prod_{i = \ell-\beta+1}^{\ell
   - 1} (1 \wedge m_i^{-1} y) y^{- \beta} dy & \leq  \frac{ C_1}{ x_0^{\beta - 1}}(\lambda d
  x_0)^{\beta-1 }  \left(\prod_{i=\ell-\beta+1}^{\ell-1}(s_{i}^*)^{d-1}\right) \log(\frac{1}{\lambda d (s_{\ell-\beta+1}^*)^{d - 1}
  x_0})\nonumber\\
 & \leq   C_1 (\lambda d)^{\beta - 1}
  \left(\prod_{i=\ell-\beta+2}^{\ell-1}(s_{i}^*)^{d-1}\right)  (s_{\ell-\beta+1}^*)^{(1-\delta)(d - 1)} .
\end{align}
Now, noting that $\ell-\beta\leq n-2$, $s_{n-1}^*\leq 1$, $\delta\in (0,1)$ and $i\mapsto s_i^*$ is non-increasing, we conclude that 
  \begin{align} 
  \label{case3b-max-small-j}
&\max_{n\in\{ (\ell-\beta+2,\cdots,\ell-1\}}\left((s_{\ell-1}^*)^{\beta-1},\sigma_n,
  \left(\prod_{i=\ell-\beta+2}^{\ell-1}s_{i}^*\right)(s_{\ell-\beta+1}^*)^{1-\delta}  \right)=\left(\prod_{i=\ell-\beta+2}^{\ell-1}s_{i}^*\right)
  (s_{\ell-\beta+1}^*)^{1-\delta} .
  \end{align}
  Combining \eqref{case3-intervals},  \eqref{case3intell-1toinfty}, \eqref{case3-partition-gen-integral}, \eqref{case3-int1tomj+1-intpart1},  \eqref{case3b-max-small-j}, \eqref{min-beta-int} (since $\kappa\leq \ell-\beta+1$)
  and the definitions of $C_\beta$ and $\psi_{\ell,j}$ in \eqref{def-C2} and \eqref{def-psi}, respectively, we obtain \eqref{last-int-nl-int} when $j=\ell-\beta$.
  
On the other hand, if $j\geq\ell-\beta+1$, then \eqref{int1nellj+1} simplifies to
\begin{align}
\label{case3-int1tomj+1-intpart2}
\frac{ C_1}{x_0^{\beta - 1}} \int_1^{m_{j+1}} \prod_{i = j+1}^{\ell
   - 1} (1 \wedge m_i^{-1} y) y^{- \beta} dy & \leq  \frac{ C_1}{x_0^{\beta - 1}} (\lambda d x_0)^{\beta-1}
   \left(\prod_{i=j+1}^{\ell-1}(s_{i}^*)^{d-1}\right) \int_1^{m_{j+1}}
  y^{- 2} dy\nonumber\\
  &\leq  C_1 (\lambda d)^{\beta-1}
   \left(\prod_{i=j+1}^{\ell-1}(s_{i}^*)^{d-1}\right).
\end{align}
Combining \eqref{case3-intervals},  \eqref{case3intell-1toinfty}, \eqref{case3-partition-gen-integral}, \eqref{case3-int1tomj+1-intpart2}, \eqref{max-case2a-large-j}, \eqref{min-beta-int} (since $\kappa\leq j+1$) and the definitions of $C_\beta$ and $\psi_{\ell,j}$ in \eqref{def-C2} and \eqref{def-psi}, respectively, we obtain \eqref{last-int-nl-int} when $j\geq\ell-\beta+1$.
\end{proof}

\noindent\textbf{Acknowledgements:} The authors were partially supported by NSF grants DMS-1713032 and  DMS-1407504 and ARO grant  W911NF2010133.

\begin{appendices}
\section{Types of Distributions}
\label{distribution}
In this section, we will look at the distribution functions of the various distributions that we considered in Section \ref{num}.
\begin{itemize}
\item
Gamma distribution: The Gamma distribution with unit mean is parameterized by $\alpha$, and
$$\bar{G}(x)=\frac{\gamma (\alpha,\alpha x)}{\Gamma(\alpha)},$$
where $\Gamma$ denotes the Gamma function
$$\Gamma(z)=\int_0^\infty x^{z-1}e^{-x}dx,$$
 and $\gamma$ denotes the lower incomplete gamma function,
 $$\gamma(a,x)=\int_0^x z^{a-1}e^{-z} dz.$$
\item
Weibull distribution: The Weibull distribution with unit mean is parameterized by $a$, and
$$\bar{G}(x)=e^{(-x/b)^{a}},$$
where $b=\frac{1}{\Gamma(1+\frac{1}{a})}$.
\item
Lognormal distribution: The Lognormal distribution with unit mean is parameterized by $\sigma$, and
$$\bar{G}(x)=\frac{1}{2}-\frac{1}{2}erf\left[\frac{\log x-\mu}{\sqrt{2}\sigma}\right],$$
where $\mu=-\frac{\sigma^2}{2}$, and $erf$ denotes the error function
$$erf(x)=\frac{2}{\sqrt{\pi}}\int_0^xe^{-z^2}dz.$$
\item
Pareto distribution: The Pareto distribution with unit mean is parameterized by $\alpha$, and
\begin{equation*}
\bar{G}(x)=\left\{\begin{array}{ll}
1&\qquad \text{if }x\in[0,x_m)\\
\left(\frac{x_m}{x}\right)^\alpha&\qquad \text{if }x\geq x_m,
\end{array}
\right.
\end{equation*}where $x_m=\frac{\alpha-1}{\alpha}$.
\item
Burr distribution: The Burr distribution with unit mean is parameterized by $k$ and $c$, and
$$\bar{G}(x)=(1+x^c)^{-k},$$
where $kB(k-\frac{1}{c},1+\frac{1}{c})=1$, and $B$ denotes the beta function, i.e.,
$$B(a,b)=\int_0^1t^{(a-1)}(1-t)^{(b-1)}dt.$$
\end{itemize}
\section{Absolute Continuity of the Invariant State}
\label{subs-abscont}

We start by making a simple observation.

\begin{lemma}
If $(\nu_\ell)_{\ell\in \N}$ is a solution to the hydrodynamic equations with initial condition $\nu(0)$, then for all Lebesgue integrable functions $f$ on $[0,\infty)$, and all $\ell\geq 1$,
\begin{align}\label{Fluid_integrable}
    \langle f, \nu_\ell(t)\rangle = &\langle f(\cdot+t)\frac{\bar G(\cdot+t)}{\bar G(\cdot)},\nu_\ell(0)\rangle +\int_{[0,t]} f(t-s)\bar G(t-s)dD_{\ell+1}(s)\\
    & + \int_0^t \langle f(\cdot+t-s)\frac{\bar G(\cdot+t-s)}{\bar G(\cdot)},\eta_\ell(s)\rangle ds.\notag
\end{align}
\end{lemma}
\begin{proof}
We  show below that \eqref{Fluid_integrable} holds for indicator functions. 
Indeed, fix $0\leq a\leq b<\infty$ consider a sequence $(\varepsilon_n)_{n\geq 1}$ with $\varepsilon_n<1$ for each 
$n\geq 1$ and $\varepsilon_n\downarrow 0$. For each $n\geq 1$, consider 
$f_n\in\mathbb{C}_b[0,\infty)$ defined by
\begin{eqnarray*}
f_n(x):=\left\{\begin{array}{rl}
1\quad& x\in[a,b],\\
\text{linear}\quad &x\in[a-\varepsilon_n,a)\cup(b,b+\varepsilon_n],\\
0\quad&\text{otherwise}.
\end{array}
\right.
\end{eqnarray*}
Then $f_n\rightarrow\mathbbm{1}_{[a,b]}$ pointwise as $n\rightarrow\infty$. Using \eqref{Fluid_f} with $f=f_n$ and the Dominated Convergence theorem, we have
\begin{align*}
  \lim_{n\rightarrow\infty}  \langle f_n, \nu_\ell(t)\rangle = &\lim_{n\rightarrow\infty}\left[\langle f_n(\cdot+t)\frac{\bar G(\cdot+t)}{\bar G(\cdot)},\nu_\ell(0)\rangle\right. +\int_{[0,t]} f_n(t-s)\bar G(t-s)dD_{\ell+1}(s)\\
    & \qquad\quad\left.+ \int_0^t \langle f_n(\cdot+t-s)\frac{\bar G(\cdot+t-s)}{\bar G(\cdot)},\eta_\ell(s)\rangle ds\right]\\
    = &\langle \mathbbm{1}_{[a,b]}(\cdot+t)\frac{\bar G(\cdot+t)}{\bar G(\cdot)},\nu_\ell(0)\rangle +\int_{[0,t]} \mathbbm{1}_{[a,b]}(t-s)\bar G(t-s)dD_{\ell+1}(s)\\
    & + \int_0^t \langle \mathbbm{1}_{[a,b]}(\cdot+t-s)\frac{\bar G(\cdot+t-s)}{\bar G(\cdot)},\eta_\ell(s)\rangle ds
\end{align*}
By linearity, it then follows that \eqref{Fluid_integrable} holds for all simple functions. Since any Lebesgue integrable function can be represented as a monotone limit of simple functions, another application of the DCT shows that \eqref{Fluid_integrable} is satisfied whenever $f$ is a Lebesgue integrable function.
\end{proof}

In what follows, $\nu^*$ is an invariant state of the hydrodynamic equations.
\begin{proof}[Proof of Lemma \ref{lem-abscont}]
For $0 \leq a < b < \infty$, substituting $\nu=\nu^*$ and $f=\mathbbm{1}_{[a,b]}$ in \eqref{Fluid_integrable}, and using \eqref{d*}, we have for each $t>0$,
\begin{align}
\label{nu_del}
    \nu_\ell^*[a,b] = &\int_0^L \mathbbm{1}_{[a,b]}(x+t)\frac{\bar G(x+t)}{\bar G(x)}\nu_\ell^*(dx) +\int_{[0,t]} \mathbbm{1}_{[a,b]}(t-s)\bar G(t-s)\langle h,\nu_{\ell+1}^*\rangle ds\nonumber\\
    & + \int_0^t \int_0^L \mathbbm{1}_{[a,b]}(x+t-s)\frac{\bar G(x+t-s)}{\bar G(x)}\eta_\ell^*(dx) ds\nonumber\\
    = &\int_{(a-t)^+\wedge L}^{(b-t)^+\wedge L} \frac{\bar G(x+t)}{\bar G(x)}\nu_\ell^*(dx) +\langle h,\nu_{\ell+1}^*\rangle \int_{(t-b)^+}^{(t-a)^+} \bar G(t-s)ds + \int_0^t \int_{(a-(t-s))^+\wedge L}^{(b-(t-s))^+\wedge L}\frac{\bar G(x+t-s)}{\bar G(x)}\eta_\ell^*(dx) ds\nonumber\\
    \leq &\nu_\ell^*[(a-t)^+\wedge L,(b-t)^+\wedge L] +\langle h,\nu_{\ell+1}^*\rangle (b-a) + \int_0^t \int_{(a-(t-s))^+\wedge L}^{(b-(t-s))^+\wedge L}\frac{\bar G(x+t-s)}{\bar G(x)}\eta_\ell^*(dx) ds .
\end{align}
The relation \eqref{nu*} implies that for $\ell=1$,
\begin{align}
\label{nu1*est}
 \int_0^t \int_{(a-(t-s))^+\wedge L}^{(b-(t-s))^+\wedge L}\frac{\bar G(x+t-s)}{\bar G(x)}\eta_1^*(dx) ds
 \leq& \int_{(t-b)^+}^{(t-a)+}\bar G(t-s)\lambda(1-\langle\f1,\nu_1^*\rangle^d) ds
 \leq \lambda(b-a).
\end{align}
\textbf{Claim:}
Given  $0\leq a<b<\infty$,
\begin{equation}
\label{abs-cont-rec}
\nu_\ell^*[a,b]\leq A_\ell (b-a)\qquad\ell\geq 1,
\end{equation}
where
\begin{equation}
\label{Aell-def}
A_\ell:=\left\{\begin{array}{ll}
\langle h,\nu_2^*\rangle+\lambda&\qquad\ell=1,\\
\langle h,\nu_{\ell+1}^*\rangle+d\lambda tA_{\ell-1}&\qquad\ell\geq 2.
\end{array}
\right.
 \end{equation}
The claim for $\ell =1$ follows on using  \eqref{nu1*est} and \eqref{nu_del}, with $\ell=1$, and sending $t\rightarrow\infty$.
 
Now, suppose \eqref{abs-cont-rec} holds
 for some $\ell-1$ with $\ell\geq 2$. For $0\leq s\leq t<\infty$, 
 using \eqref{nu*}, the inequality $\mathfrak{P}_d(x,y)\leq d$ $\forall x,y\geq 0$, and \eqref{abs-cont-rec}, 
 with $\ell$ replaced by $\ell-1$, we have
\begin{align}
 \int_0^t \int_{(a-(t-s))^+\wedge L}^{(b-(t-s))^+\wedge L}\frac{\bar G(x+t-s)}{\bar G(x)}\eta_\ell^*(dx) ds
 &\leq 
 \int_0^t \int_{(a-(t-s))^+\wedge L}^{(b-(t-s))^+\wedge L}\frac{\bar G(x+t-s)}{\bar G(x)}d\lambda (\nu_{\ell-1}^*(dx)-\nu_{\ell}^*(dx)) ds\nonumber\\
&\leq d\lambda\int_0^t\nu_{\ell-1}^*((a-(t-s))^+\wedge L,(b-(t-s))^+\wedge L) ds\nonumber\\
&\leq d\lambda t A_{\ell-1}(b-a).\nonumber
\end{align}
 When combined with \eqref{nu_del} and\eqref{Aell-def}, and sending $t\rightarrow\infty$, the above equation yields \eqref{abs-cont-rec}. 
 Given $\varepsilon_1>0$, define $\delta_1(\varepsilon_1):=\varepsilon_1/{2(\langle h,\nu_{2}^*\rangle+\lambda)}$.  Let $E\subset[0,\infty)$ be a Lebesgue measurable set  with $\mu(E)<\delta_1(\varepsilon_1)$, where $\mu$ denotes Lebesgue measure. Since
(by, e.g., \cite[Lemma 1.17]{Folland}) 
\[\mu(E)=\inf\left\{\sum_{i\in\mathbb{N}}\mu((a_i,b_i)):E\subset\bigcup_{i\in\mathbb{N}}(a_i,b_i)\right\},\]
  there exists $((a_i,a_i+\Delta_i))_{i\in\mathbb{N}}$ such that $E\subset\bigcup_{i\in\mathbb{N}}(a_i,a_i+\Delta_i)$ and $\sum_{i\in\mathbb{N}} \mu((a_i,a_i+\Delta_i))=\sum_{i\in\mathbb{N}}\Delta_i<2\delta_1(\varepsilon_1)$.
Since $\langle h,\nu_2^*\rangle\geq 0$ by Lemma \ref{lem-hus},  using \eqref{abs-cont-rec} and \eqref{Aell-def} with $\ell=1$, we obtain for all $i$,
\begin{eqnarray*}
\nu_1^*(E)&\leq& \sum_{i\in\mathbb{N}} \nu_1^*(a_i,a_i+\Delta_i)\\
&\leq&\sum_{i\in\mathbb{N}}( \langle h,\nu_{2}^*\rangle+\lambda)\Delta_i\\
&=&2(\langle h,\nu_{2}^*\rangle+\lambda)\delta_1(\varepsilon_1)\\
&=&\varepsilon_1.
\end{eqnarray*}

Likewise, fix $\ell\geq 2$. Given $\varepsilon_\ell>0$, define $\delta_\ell(\varepsilon_\ell)={\varepsilon_\ell}/{2n_\ell}$. Now suppose $E\subset[0,\infty)$ is  Lebesgue measurable with $\mu(E)<\delta_\ell$. Then there exist $\left((a_i,a_i+\Delta_i)\right)_{i\in\mathbb{N}}$ such that $E\subset\bigcup_{i\in\mathbb{N}}(a_i,a_i+\Delta_i)$ and $\sum_{i\in\mathbb{N}} \mu((a_i,a_i+\Delta_i))=\sum_{i\in\mathbb{N}}\Delta_i<2\delta_\ell$.  Using \eqref{abs-cont-rec}, we see that,
 \begin{equation*}
   \sum_{i\in\mathbb{N}}\nu_\ell^*[a_i,a_i+\Delta_i] \leq \sum_{i\in\mathbb{N}}A_\ell\Delta_i. \end{equation*}
Hence, using $\langle h,\nu_\ell^*\rangle \geq 0$ for all $\ell\geq 1$ from Lemma \ref{lem-hus}, we have,
\begin{equation*}
\nu_\ell^*(E)\leq \sum_{i\in\mathbb{N}}\nu_\ell^*(a_i,a_i+\Delta_i)
\leq \sum_{i\in\mathbb{N}} A_\ell\Delta_i
<2A_\ell\delta_\ell(\varepsilon_\ell)
=\varepsilon_\ell.
\end{equation*}
Thus, we have shown that for every $\ell\geq 1$, given  $\varepsilon_\ell>0$, $\exists\delta_\ell(\varepsilon_\ell)>0$ such that
 $\mu(E)<\delta_\ell(\varepsilon_\ell)\implies\nu_\ell^*(E)<\varepsilon_\ell$. This proves $\nu_\ell^*$ is absolutely continuous with respect to $\mu$.
\end{proof}

\section{Uniqueness of the Invariant State when $d=2$}
\label{sec-unique}
In this section we prove Proposition \ref{th-uniqueness}. The proof will rely on the following elementary lemma.

\begin{lemma}
\label{uniquefixed}
Let $b<\infty$ and let $f$ be a function on $[0,b]$ such that $f$ is continuously differentiable on $(0,b)$, $f(x)\neq 0$ for $x\in\{0,b\}$, has at least one root in $(0,b)$ and the derivative of $f$ at every root in $(0,b)$ is negative, i.e., 
$$ f(x)=0 \implies f'(x)<0,\qquad \forall x\in(0,b) .$$
 Then there exists a unique $x\in(0,b)$ such that $f(x)=0$.
\end{lemma}
\begin{proof}
First note that since $f'$ is negative at all its roots, there does not exist any interval $I\subset(0,b)$ such that $f\equiv 0$ on $I$. Let $0<x_1<x_2<b$ be two distinct roots of $f$. Since $f'(x_1)<0$, $ f'(x_2)<0$ and $f$ has a  continuous derivative, there exist $y_1>x_1$, and $y_2<x_2$ such that $f(y_1)<0$ and $f(y_2)>0$. By the intermediate value theorem there exists at least one $x_*\in(x_1,x_2)$ such that $f(x_*)=0$ and  $f'(x_*)\geq0$. This  contradicts our assumption, and hence, $f$ has a unique root.
\end{proof}

\begin{proof}[Proof of Proposition \ref{th-uniqueness}]
 Fix $d=2$. By definition we know $s_1^*=\lambda$ and we set $r_1\equiv 1$ and note that \eqref{eqn-ssl} holds. Now, suppose for any $m\geq 2$, $\vsl$ and $\ssl$ satisfy \eqref{eqn-vsl} and \eqref{eqn-ssl} for $\ell=1,2, \cdots, m-1$. Then by Remark \ref{rem-s-fixed}, it suffices to show that  $F_{m}$ has a unique fixed point in $[0,s_{m-1}^*]$, which then must be $s_m^*$. 
 Define $H_m= F_m(s)-s$, $s\in[0,s_{m-1}^*]$. Then, setting $d=2$ in the expression for $\beta_2$ in \eqref{def_poly} and using \eqref{eq-fld},
\begin{align}
\label{def-Gl}
H_m(s)
=&\lambda s^2 \int_0^L e^{-\lambda  ( s_{m-1}^* + s)  x} \bar{G}(x) dx
+ \lambda  ( s_{m-1}^* +  s) \int_0^L \left(
  \int_0^x e^{-\lambda  (s_{m-1}^*+ s) (x-u)} r_{m-1} (u) du
\right) \bar{G}(x) dx-s.
\end{align} 
 Note that 
\[H_m(0)= \lambda   s_{m-1}^*  \int_0^L \left(
  \int_0^x e^{-\lambda  s_{m-1}^* (x-u)} r_{m-1} (u) du
\right) \bar{G}(x) dx>0.\]
Also, integrating the right-hand side of \eqref{def-Gl} by parts, and using the fact that \eqref{eqn-vsl} and \eqref{eqn-ssl}, with $\ell=m-1$, imply $r_{m-1}(0)=\lambda (s_{m-1}^*)^2$ and $s_{m-1}^*=\int_0^Lr_{m-1}(x)\bar{G}(x) dx$, we conclude that
$$H_m(s)=s_{m-1}^*-s-\left[\lambda(s_{m-1}^*)^2-\lambda s^2\right]\int_0^Le^{-\lambda(s_{m-1}^*+s)x}\bar{G}(x)dx-\int_0^L\left(\int_0^xe^{-\lambda(s_{m-1}^*+s)(x-u)}r_{m-1}'(u)du\right)\bar{G}(x)dx.$$
Since $r_{m-1}'\geq 0$ by Lemma \ref{vprop}, we have
$H_m(s_{m-1}^*)<0$.
Thus, $H_m(x)\neq 0$ for $x\in\{0,s_{m-1}^*\}$.
On differentiating \eqref{def-Gl}, we obtain
\begin{align*}
H_m'(s)=&\int_0^Le^{-\lambda(s_{m-1}^*+s)x}\left[2\lambda  s-\lambda^2 s^2  x\right] \bar{G}(x)dx+\lambda\int_0^L\left(\int_0^xe^{-\lambda(s_{m-1}^*+s)(x-u)}r_{m-1}(u) du\right)\bar{G}(x) dx\\
&-\int_0^L\left(\int_0^xe^{-\lambda(s_{m-1}^*+s)(x-u)}r_{m-1}(u)\lambda^2(s_{m-1}^*+s)(x-u)du\right)\bar{G}(x) dx-1.
\end{align*}
Using \eqref{eq-fld} with $d=2$ and $\ell=m-1$ 
 we see that
\begin{align*}
\lambda\int_0^L\left(\int_0^xe^{-\lambda(s_{m-1}^*+s)(x-u)}r_{m-1}(u)du\right) \bar{G}(x)dx
&=\frac{1}{(s_{m-1}^*+s)}\left(F_m(s)-\lambda s^2\int_0^Le^{-\lambda (s_{m-1}^*+s)x} \bar{G}(x)dx\right).
\end{align*}
Now, since $s_1^*=\lambda$ and $r_{m-1}$ satisfies \eqref{eqn-vsl} and therefore \eqref{vsl2}, Lemmas \ref{vprop} and \ref{lem-s1eqbm} imply the monotonicity of $r_{m-1}$, we have for all $u\in [0,L)$, $r_{m-1}(u)\geq r_{m-1}(0)\geq \lambda(s_{m-1}^*)^2$. Thus,
\begin{align*}
&\int_0^L\left(\int_0^xe^{-\lambda(s_{m-1}^*+s)(x-u)}r_{m-1}(u)\lambda^2(s_{m-1}^*+s)(x-u)du\right) \bar{G}(x)dx\\
&\quad \geq \lambda^3 (s_{m-1}^*)^2(s_{m-1}^*+s)\int_0^L\left(\int_0^xe^{-\lambda(s_{m-1}^*+s)(x-u)}(x-u)du\right)\bar{G}(x) dx\\
&\quad =\lambda^3 (s_{m-1}^*)^2(s_{m-1}^*+s)\int_0^L\left[\frac{1}{\lambda^2(s_{m-1}^*+s)^2}-\frac{e^{-\lambda(s_{m-1}^*+s)x}}{\lambda^2(s_{m-1}^*+s)^2}-\frac{xe^{-\lambda(s_{m-1}^*+s)x}}{\lambda(s_{m-1}^*+s)}\right]\bar{G}(x)dx\\
&\quad=\frac{\lambda(s_{m-1}^*)^2}{(s_{m-1}^*+s)}\int_0^L(1-e^{-\lambda(s_{m-1}^*+s)x})\bar{G}(x)dx-\lambda^2 (s_{m-1}^*)^2\int_0^Le^{-\lambda(s_{m-1}^*+s)x}x \bar{G}(x)dx\\
&\quad=\frac{\lambda(s_{m-1}^*)^2}{(s_{m-1}^*+s)}\left(1-\int_0^Le^{-\lambda(s_{m-1}^*+s)x}\bar{G}(x)dx\right) -\lambda^2 (s_{m-1}^*)^2\int_0^Le^{-\lambda (s_{m-1}^*+s)x}x \bar{G}(x)dx.
\end{align*}
Combining the last three displays and rearranging, we obtain
\begin{align*}
H_m'(s)\leq &\int_0^Le^{-\lambda(s_{m-1}^*+s)x}\left[2\lambda  s-\lambda^2 s^2 x\right] \bar{G}(x)dx+\frac{F_\ell(s)}{(s_{m-1}^*+s)}-1-\frac{\lambda s^2}{(s_{m-1}^*+s)}\int_0^Le^{-\lambda(s_{m-1}^*+s)x}\bar{G}(x) dx\\
&-\frac{\lambda(s_{m-1}^*)^2}{(s_{m-1}^*+s)}\left(1-\int_0^Le^{-\lambda(s_{m-1}^*+s)x}\bar{G}(x) dx\right)+\lambda^2 (s_{m-1}^*)^2\int_0^Le^{-\lambda (s_{m-1}^*+s)x}x \bar{G}(x) dx\\
=&\frac{F_\ell(s)}{(s_{m-1}^*+s)}-1-\frac{\lambda(s_{m-1}^*)^2}{(s_{m-1}^*+s)}+\left(2\lambda  s-\frac{\lambda s^2}{(s_{m-1}^*+s)}+\frac{\lambda(s_{m-1}^*)^2}{(s_{m-1}^*+s)}\right) \int_0^Le^{-\lambda(s_{m-1}^*+s)x}\bar{G}(x)dx\\
&+\left(\lambda^2 (s_{m-1}^*)^2-\lambda^2 s^2  \right)\int_0^Le^{-\lambda (s_{m-1}^*+s)x}x \bar{G}(x)dx\\
=&\frac{(F_\ell(s)-\lambda(s_{m-1}^*)^2)}{(s_{m-1}^*+s)}-1 +\left(2\lambda  s+\frac{\lambda((s_{m-1}^*)^2-s^2)}{(s_{m-1}^*+s)}\right) \int_0^Le^{-\lambda(s_{m-1}^*+s)x}\bar{G}(x)dx\\
&+\lambda^2 ((s_{m-1}^*)^2-s^2) \int_0^Le^{-\lambda (s_{m-1}^*+s)x}x  \bar{G}(x)dx\\
=&\frac{(F_\ell(s)-\lambda(s_{m-1}^*)^2)}{(s_{m-1}^*+s)}-1 +\lambda(s_{m-1}^*+s) \int_0^Le^{-\lambda(s_{m-1}^*+s)x}\bar{G}(x)dx\\
&+\lambda^2 ((s_{m-1}^*)^2-s^2)  \int_0^Le^{-\lambda (s_{m-1}^*+s)x}x \bar{G}(x)dx.
\end{align*}
Using H\"older's inequality, we have 
\begin{align*}
\int_0^Le^{-\lambda(s_{m-1}^*+s)x}\bar{G}(x)dx &\leq \left(\int_0^Le^{-2\lambda(s_{m-1}^*+s)x}dx\right)^{1/2}\left(\int_0^L\bar{G}^2(x)dx\right)^{1/2}\\
&\leq\left(\int_0^L e^{-2\lambda(s_{m-1}^*+s)x}dx\right)^{1/2}\\
&\leq\frac{1}{\sqrt{2\lambda(s_{m-1}^*+s)}},
\end{align*}
and
\begin{align*}
\int_0^Lx e^{-\lambda(s_{m-1}^*+s)x}\bar{G}(x)dx &\leq \left(\int_0^L x^2 e^{-2\lambda(s_{m-1}^*+s)x}dx\right)^{1/2}\left(\int_0^L\bar{G}^2(x)dx\right)^{1/2}\\
&\leq\left(\int_0^L x^2 e^{-2\lambda(s_{m-1}^*+s)x}dx\right)^{1/2}\\
&\leq\frac{1}{2\sqrt{\lambda^3(s_{m-1}^*+s)^3}}.
\end{align*}
Combining the last three displays and evaluating $H_m'$ at any fixed point $\tilde{s}_{m}^*$ of $F_m$ in $[0,s_{m-1}^*]$, 
\begin{align}
\label{Hprime-upbnd}
H_m'(\tilde{s}_{m}^*)\leq &-\frac{s_{m-1}^*+\lambda(s_{m-1}^*)^2}{(s_{m-1}^*+\tilde{s}_{m}^*)}+\left(\frac{1}{\sqrt{2}}+\frac{1}{2}\right)\sqrt{\lambda(s_{m-1}^*+\tilde{s}_{m}^*)}
-\frac{\lambda \tilde{s}_{m}^*}{\sqrt{\lambda(s_{m-1}^*+\tilde{s}_{m}^*)}} =K_{m}(\tilde{s}^*_m),
\end{align}
where for $s\in[0,s_{m-1}^*]$,
$$K_m(s)=-\frac{s_{m-1}^*+\lambda(s_{m-1}^*)^2}{(s_{m-1}^*+s)}+\left(\frac{1}{\sqrt{2}}+\frac{1}{2}\right)\sqrt{\lambda(s_{m-1}^*+s)}
-\frac{\lambda s}{\sqrt{\lambda(s_{m-1}^*+s)}}  .$$
Differentiating $K_m$ we obtain
\begin{align*}
K_m'(s)=&\frac{s_{m-1}^*+\lambda(s_{m-1}^*)^2}{(s_{m-1}^*+s)^2}+\frac{1}{2}\left(\frac{1}{\sqrt{2}}+\frac{1}{2}\right)\frac{\lambda}{\sqrt{\lambda(s_{m-1}^*+s)}}-\frac{\lambda }{\sqrt{\lambda(s_{m-1}^*+s)}}+\frac{\lambda^2 s}{2\lambda(s_{m-1}^*+s)\sqrt{\lambda(s_{m-1}^*+s)}}\\
=&\frac{(s_{m-1}^*+\lambda(s_{m-1}^*)^2)}{(s_{m-1}^*+s)^2}+\frac{1}{2}\left(\frac{1}{\sqrt{2}}+\frac{1}{2}\right)\frac{\lambda}{\sqrt{\lambda(s_{m-1}^*+s)}}-\frac{2\lambda s_{m-1}^*+\lambda s}{2(s_{m-1}^*+s)\sqrt{\lambda(s_{m-1}^*+s)}}\\
=&\frac{4 s_{m-1}^*+4\lambda(s_{m-1}^*)^2-(3 s_{m-1}^*+s)\sqrt{\lambda( s_{m-1}^*+s)}}{4(s_{m-1}^*+s)^2}+\frac{1}{2\sqrt{2}}\frac{\lambda}{\sqrt{\lambda( s_{m-1}^*+s)}}.
\end{align*}
To show that $K_m'(s)\geq 0$ for $s\in[0,s_{m-1}^*]$, it suffices to show the numerator of the first term is positive. Since $a^2>b^2$ implies $a>b$ when $a,b>0$, this follows from
\begin{align*}
&16(s_{m-1}^*)^2+16\lambda^2(s_{m-1}^*)^4+32\lambda(s_{m-1}^*)^3-\left(\lambda(s_{m-1}^*+s)(9(s_{m-1}^*)^2+s^2+6s_{m-1}^* s)\right)\\
&\geq 16(s_{m-1}^*)^2+16\lambda^2(s_{m-1}^*)^4+32\lambda(s_{m-1}^*)^3-2\lambda s_{m-1}^*(16 s_{m-1}^*)\\
&\geq 16( s_{m-1}^*)^2+16\lambda^2(s_{m-1}^*)^4\\
&\geq 0.
\end{align*}
 Hence, $K_m$ is increasing on $[0,s_{m-1}^*]$. When combined with \eqref{Hprime-upbnd}, this implies
\begin{align*}
H'_m(\tilde{s}_m^*)\leq K_m(\tilde{s}_m^*)\leq K_m(s_{m-1}^*)
&=-\frac{1+\lambda s_{m-1}^*}{2}+\left(1+\frac{1}{\sqrt{2}}\right)\sqrt{\lambda s_{m-1}^*}-\frac{\sqrt{\lambda s_{m-1}^*}}{\sqrt{2}}\\
&=-\frac{1}{2}\left(1+\lambda s_{m-1}^*-2\sqrt{\lambda s_{m-1}^*}\right)\\
&=-\frac{1}{2}(1-\sqrt{\lambda s_{m-1}^*})^2\\
&<0.
\end{align*}
Hence using the definition of $H_m$ in \eqref{def-Gl},  continuous differentiability  of $F_m$ from Proposition \ref{prop-reduction} and substituting $b=s_{m-1}^*$ and $f=H_m$ in Lemma \ref{uniquefixed}, we conclude that $F_m$ has a unique fixed point in $[0, s_{m-1}^*]$.
\end{proof}

\end{appendices}
\bibliographystyle{acm}
\bibliography{loadbal}

\end{document}